\documentclass[11pt]{article}

\usepackage[letterpaper,margin=1in]{geometry}

\usepackage{amsmath, amssymb}
\usepackage{amsthm}

\usepackage[pdfpagelabels,bookmarks=false]{hyperref}
\hypersetup{colorlinks, linkcolor=darkblue, citecolor=darkgreen, urlcolor=darkblue}
\usepackage[capitalize, nameinlink]{cleveref}
\crefname{theorem}{Theorem}{Theorems}
\Crefname{lemma}{Lemma}{Lemmas}
\Crefname{claim}{Claim}{Claims}
\Crefname{fact}{Fact}{Facts}
\Crefname{remark}{Remark}{Remarks}
\Crefname{observation}{Observation}{Observations}
\Crefname{line}{Line}{Lines}
\crefalias{AlgoLine}{line}

\newtheorem{theorem}{Theorem}
\newtheorem{lemma}[theorem]{Lemma}

\newtheorem{definition}[theorem]{Definition}

\newtheorem{corollary}[theorem]{Corollary}

\newtheorem{observation}[theorem]{Observation}

\newtheorem{claim}[theorem]{Claim}

\usepackage{times}

\usepackage{anyfontsize}

\usepackage[inline]{enumitem}
\setlist[enumerate,1]{label=(\roman*), leftmargin=2.2em}
\setlist[enumerate,2]{label=(\alph*)}
\setlist{nosep,topsep=0.1em}
\setlist[itemize,1]{label={\bfseries--}}

\usepackage{mathtools}
\DeclarePairedDelimiter\ceil{\lceil}{\rceil}
\DeclarePairedDelimiter\floor{\lfloor}{\rfloor}
\makeatletter
\newtheorem*{rep@theorem}{\rep@title}
\newcommand{\newreptheorem}[2]{\newenvironment{rep#1}[1]{\def\rep@title{#2 \ref{##1}}\begin{rep@theorem}}{\end{rep@theorem}}}
\makeatother

\newreptheorem{theorem}{Theorem}
\newreptheorem{lemma}{Lemma}
\newreptheorem{corollary}{Corollary}

\usepackage{mdframed}

\makeatletter
\let\@@pmod\pmod
\DeclareRobustCommand{\pmod}{\@ifstar\@pmods\@@pmod}
\def\@pmods#1{\mkern8mu({\operator@font mod}\mkern 6mu#1)}
\makeatother

\makeatletter
\let\@@mod\mod
\DeclareRobustCommand{\mod}{\@ifstar\@mods\@@mod}
\def\@mods#1{\mkern8mu{\operator@font mod}\mkern 6mu#1}
\makeatother
\usepackage[color=darkgreen!40!white]{todonotes}

\usepackage{xcolor}
\definecolor{darkblue}{rgb}{0,0,0.38}
\definecolor{darkred}{rgb}{0.6,0,0}
\definecolor{darkgreen}{rgb}{0.1,0.35,0}

\usepackage[
backend=biber,
style=alphabetic,
citestyle=alphabetic,
maxalphanames=6,
maxcitenames=99,
mincitenames=98,
maxbibnames=99,
giveninits=true,
url=false,
doi=true,
isbn=true,
backref=true
]{biblatex}

\usepackage{xpatch}

\makeatletter
\patchcmd\blx@bblinput{\blx@blxinit}
                      {\blx@blxinit
                      }{}{\fail}
\makeatother

\makeatletter
\DeclareFieldFormat{eprint:arxiv}{arXiv\addcolon\space
  \ifhyperref
    {\href{https://arxiv.org/abs/#1}{\nolinkurl{#1}\iffieldundef{eprintclass}
     {}
{\addspace\mkbibbrackets{\thefield{eprintclass}}}}}
    {\nolinkurl{#1}
     \iffieldundef{eprintclass}
       {}
{\addspace\mkbibbrackets{\thefield{eprintclass}}}}}
\makeatother

\makeatletter
\newcommand{\labeltarget}[1]{\Hy@raisedlink{\hypertarget{#1}{}}}
\makeatother

\usepackage[vlined,ruled,algo2e,linesnumbered]{algorithm2e}
\SetAlCapHSkip{0.2em}
\SetAlgoInsideSkip{smallskip}
\SetCustomAlgoRuledWidth{0.95\textwidth}
\makeatletter
\patchcmd{\@algocf@start}{\begin{lrbox}{\algocf@algobox}}{\rule{0.025\textwidth}{\z@}\begin{lrbox}{\algocf@algobox}\begin{minipage}{0.95\textwidth}}{}{}
\patchcmd{\@algocf@finish}{\end{lrbox}}{\end{minipage}\end{lrbox}}{}{}
\makeatother

\usepackage{xfrac}
\ExplSyntaxOn
\DeclareRestrictedTemplate { xfrac } { text } { math }
  {
    numerator-font      = \number \fam ,
    slash-symbol        = /            ,
    slash-symbol-font   = \number \fam ,
    denominator-font    = \number \fam ,
    scale-factor        = 0.7          ,
    scale-relative      = false        ,
    scaling             = true         ,
    denominator-bot-sep = 0 pt         ,
    math-mode           = true         ,
    phantom             = ( }
\DeclareInstance { xfrac } { mathdefault } { math }
  { numerator-top-sep = 0pt }
\ExplSyntaxOff

\usepackage{wrapfig}

\usepackage[margin=10pt,font=small,labelfont=bf]{caption}

\usepackage{graphicx}
\graphicspath{{.}{graphics/}}
\makeatletter
\newcommand\appendtographicspath[1]{\g@addto@macro\Ginput@path{#1}}
\makeatother

\usepackage{tikz}
\usetikzlibrary{calc}

\newcommand{\ksum}[1][]{\mathbin{\oplus_{#1}}}
\newcommand{\pivot[1]}[]{\operatorname{pivot}_{#1}}

\makeatletter
\DeclareRobustCommand{\cev}[1]{{\mathpalette\do@cev{#1}}}
\newcommand{\do@cev}[2]{\vbox{\offinterlineskip
    \sbox\z@{$\m@th#1 x$}\ialign{##\cr
      \hidewidth\reflectbox{$\m@th#1\vec{}\mkern4mu$}\hidewidth\cr
      \noalign{\kern-\ht\z@}
      $\m@th#1#2$\cr
    }}}
\makeatother

\newcommand{\CCTU}{\hyperlink{prb:CCTU}{CCTU}}
\newcommand{\CCTUF}{\hyperlink{prb:CCTUF}{CCTUF}}
\newcommand{\RCCTUF}{\hyperlink{prb:R-CCTUF}{$R$-CCTUF}}
\newcommand{\CTC}{\hyperlink{prb:CTC}{CTC}}
\newcommand{\CCC}{\hyperlink{prb:CCC}{CCC}}
\newcommand{\XLC}{\hyperlink{prb:XLC}{XLC}}
\newcommand{\CCSM}{\hyperlink{prb:CCSM}{CCSM}}

\DeclareMathOperator{\vertices}{vertices}

\newcommand{\MOORdot}{}

\title{Congruency-Constrained TU Problems Beyond the Bimodular Case\thanks{This project received funding from Swiss National Science Foundation grants 200021\_184622 and P500PT\_206742, the European Research Council (ERC) under the European Union's Horizon 2020 research and innovation programme (grant agreement No 817750), and the Deutsche Forschungsgemeinschaft (DFG, German Research Foundation) under Germany's \hbox{Excellence Strategy~-~GZ~2047/1}, Projekt-ID 390685813.
}}
\author{
Martin N{\"a}gele\thanks{Research Institute for Discrete Mathematics and Hausdorff Center for Mathematics, University of Bonn, Bonn, Germany.
Email: \href{mailto:naegele@or.uni-bonn.de}{naegele@or.uni-bonn.de}.
Most of this work was done while the author was employed at ETH Zurich.
}
\and
Richard Santiago\thanks{
Department of Mathematics, ETH Zurich, Zurich, Switzerland.
Email: \href{mailto:rtorres@ethz.ch}{rtorres@ethz.ch}.}
\and
Rico Zenklusen\thanks{
Department of Mathematics, ETH Zurich, Zurich, Switzerland.
Email: \href{mailto:ricoz@ethz.ch}{ricoz@ethz.ch}.}
}

\date{}

\begin{document}

\maketitle

\begin{abstract}
A long-standing open question in Integer Programming is whether integer programs with constraint matrices with bounded subdeterminants are efficiently solvable.
An important special case thereof are congruency-constrained integer programs
$\min\{c^\top x\colon Tx\leq b, \gamma^\top x\equiv r\pmod*{m}, x\in\mathbb{Z}^n\}$ with a totally unimodular constraint matrix $T$.
Such problems have been shown to be polynomial-time solvable for $m=2$, which led to an efficient algorithm for integer programs with bimodular constraint matrices, i.e., full-rank matrices whose $n\times n$ subdeterminants are bounded by two in absolute value.
Whereas these advances heavily relied on existing results on well-known combinatorial problems with parity constraints, new approaches are needed beyond the bimodular case, i.e., for $m>2$.

We make first progress in this direction through several new techniques. In particular, we show how to efficiently decide feasibility of congruency-constrained integer programs with a totally unimodular constraint matrix for $m=3$ using a randomized algorithm.
Furthermore, for general $m$, our techniques also allow for identifying flat directions of infeasible problems, and deducing bounds on the proximity between solutions of the problem and its relaxation.
\end{abstract}

\thispagestyle{empty}
\addtocounter{page}{-1}

\begin{tikzpicture}[overlay, remember picture, shift = {(current page.south east)}]
\node[anchor=south east, outer sep=5mm] {
\begin{tikzpicture}[outer sep=0] \node (ERC) {\includegraphics[height=13mm]{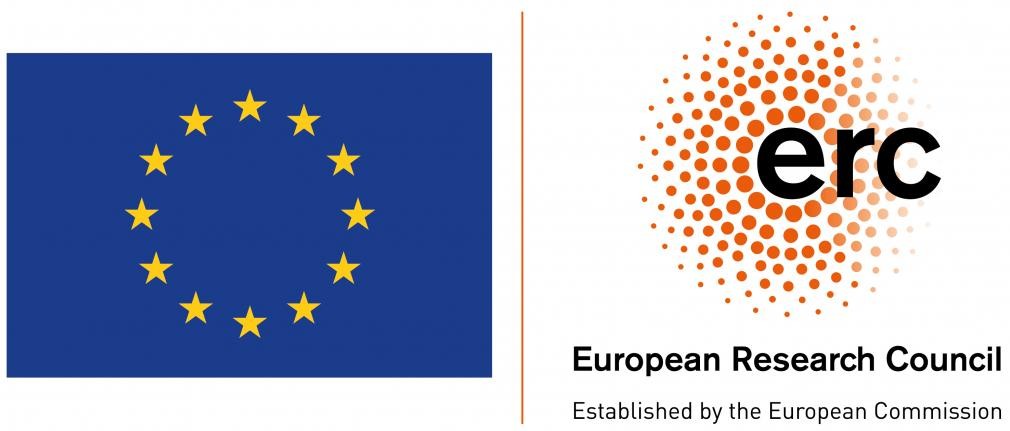}};
\node[left=5mm of ERC] (SNSF) {\includegraphics[height=7mm]{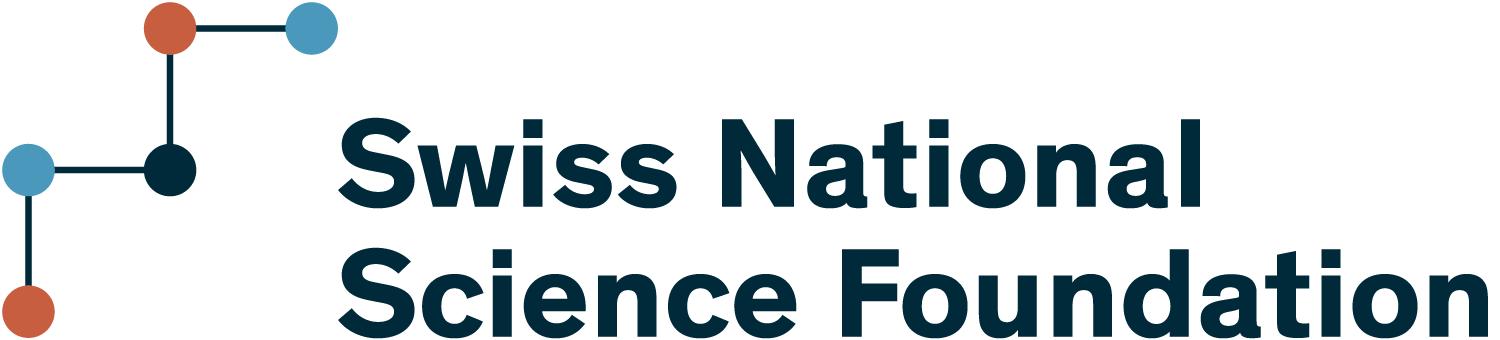}};
\node[right=5mm of ERC] (DFG) {\includegraphics[height=5mm]{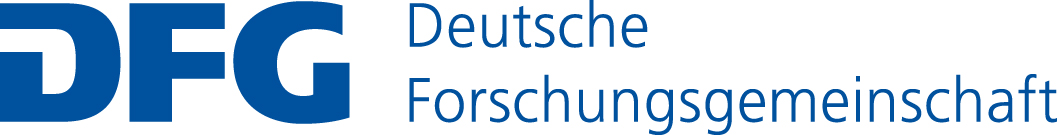}};\end{tikzpicture}
};
\end{tikzpicture}

\newpage

\section{Introduction\MOORdot}

Integer linear programs (ILPs) $\min\{c^\top x\colon Ax\leq b, x\in\mathbb{Z}^n\}$ for $A\in\mathbb{Z}^{k\times n}$, $b\in\mathbb{Z}^k$, and $c\in\mathbb{Z}^n$ are one of the most basic yet powerful discrete optimization problems.
They are well-known to be NP-hard, and extensive research is dedicated to identify efficiently solvable subclasses.
One of the best known such classes is when the constraint matrix $A$ is required to be \emph{totally unimodular} (TU), i.e., all square submatrices of $A$ have a determinant in $\{-1, 0, 1\}$.
The class of totally unimodular ILPs still comprises a large number of interesting and heavily studied problems, as for example network flow and cut problems, bipartite matching problems, and many others.

Intriguingly, it is still badly understood what kind of generalizations of this classical result on ILPs with totally unimodular constraint matrices are possible to obtain larger classes of efficiently solvable ILPs.
In particular, there is a long-standing open question on whether ILPs are efficiently solvable if their constraint matrix is $\Delta$-modular for constant $\Delta$.
Here, a matrix $A\in\mathbb{Z}^{k\times n}$ is $\Delta$-modular for $\Delta\in\mathbb{Z}_{>0}$ if it has full column rank $n$, and all $n\times n$ submatrices have determinants bounded by $\Delta$ in absolute value.\footnote{One may also consider \emph{totally $\Delta$-modular} matrices $A$, where \emph{all} square subdeterminants of $A$ are bounded by $\Delta$ in absolute value. The notion of $\Delta$-modularity is more general in the sense that totally $\Delta$-modular ILPs can be reduced to $\Delta$-modular ILPs. In particular, reducing to a problem with full-rank constraint matrix can be achieved by a standard transformation to non-negative variables.}
Besides TU constraint matrices, progress has only been achieved for the bimodular case $\Delta=2$, for which an efficient algorithm was presented by Artmann, Weismantel, and Zenklusen~\cite{artmann_2017_strongly}.
A relevant special case of such problems are Congruency-Constrained TU Problems.\footnote{A \CCTU{} problem with modulus $m$ can be written as an $m$-modular ILP by transforming the congruency constraint into a linear equality constraint as follows.
First append the row $\gamma^\top$ to the matrix $T$ and then append a column with zeros everywhere except for the last entry (the one corresponding to the newly added row), which is set to $m$.
Finally, the right-hand side of the newly added constraint is set to $r$, the target residue.
\label{footn:red_CCTU_ILP}}\begin{mdframed}[innerleftmargin=0.5em, innertopmargin=0.5em, innerrightmargin=0.5em, innerbottommargin=0.5em, userdefinedwidth=0.95\linewidth, align=center]
{\textbf{Congruency-Constrained TU Optimization (\hypertarget{prb:CCTU}{CCTU}):}}
Let $T\in\{-1, 0, 1\}^{k\times n}$ be TU, $b\in\mathbb{Z}^k$, $c\in\mathbb{Z}^n$, $m\in\mathbb{Z}_{> 0}$, $\gamma\in\mathbb{Z}^n$, and $r\in\mathbb{Z}$. The task is to show infeasibility, unboundedness, or find a  minimizer of
\begin{equation*}
\min\left\{c^\top x \colon Tx \leq b,\ \gamma^\top x \equiv r \pmod*{m},\ x\in\mathbb{Z}^n \right\}\enspace.
\end{equation*}
\end{mdframed}
Even for $m=2$, \CCTU{} problems capture classical combinatorial optimization problems like the minimum odd cut problem.
Moreover, there are reasons to believe that insights on \CCTU{} problems may be key to make further progress on the open question of bounded subdeterminant ILPs.
For $\Delta=2$, a result of Veselov and Chirkov~\cite{veselov_2009_integer} implies that bimodular ILPs reduce to \CCTU{} problems with $m=2$, i.e., with parity constraints (see~\cite{artmann_2017_strongly}).
The result in~\cite{veselov_2009_integer} does not extend to $\Delta > 2$, and it remains open whether another reduction to \CCTU{} problems may exist.
Questions closely related to \CCTU{} have also appeared in recent progress of Fiorini, Joret, Weltge, and Yuditsky~\cite{fiorini_2021_integerPrograms}, who obtained an efficient algorithm for totally $\Delta$-modular ILPs with a constraint matrix having at most two non-zeros in each row.
This algorithm computes certain circulations with parity constraints, which can be interpreted as \CCTU{} problems with a bounded number of additional constraints.

Moreover, we highlight that for prime numbers $m$, \CCTU{} problems with modulus $m$ are equivalent to ILPs with a constraint matrix $A$ that has full column rank and all of whose $n\times n$ subdeterminants are within $\{0,\pm m\}$, in the sense that any of the two problems can be efficiently transformed to the other one.\footnote{The reduction mentioned in \cref{footn:red_CCTU_ILP} from a \CCTU{} problem to a $\Delta$-modular ILP shows one direction. The other one follows by an analogous reduction to the one used in the bimodular case~\cite{artmann_2017_strongly}.
We highlight that in the conference version of this paper~\cite{nagele_2022_congruency}, we missed adding that $m$ needs to be prime for such an analogous reduction to work out.}

Typically, we consider \CCTU{} problems with a constant modulus $m$, since \CCTU{} with arbitrary non-constant modulus $m$ is NP-hard (one can, for example, model the minimum bisection problem).

\subsection{Our results\MOORdot}

We present the first progress towards solving \CCTU{} problems beyond the parity-constrained case by approaching the corresponding feasibility problem.

\begin{mdframed}[innerleftmargin=0.5em, innertopmargin=0.5em, innerrightmargin=0.5em, innerbottommargin=0.5em, userdefinedwidth=0.95\linewidth, align=center]
{\textbf{Congruency-Constrained TU Feasibility (\hypertarget{prb:CCTUF}{CCTUF}):}}
Let $T\in\{-1, 0, 1\}^{k\times n}$ be a totally unimodular matrix, let $b\in\mathbb{Z}^k$, $m\in\mathbb{Z}_{>0}$, $\gamma\in\mathbb{Z}^n$, and $r\in\mathbb{Z}$. The task is to show infeasibility or find a solution of the system
$$
Tx \leq b,\ \gamma^\top x \equiv r \pmod*{m},\ x\in\mathbb{Z}^n \enspace.
$$
\end{mdframed}

\noindent Our main result is the following.

\begin{theorem}\label{thm:feasibilityMod3}
There is a strongly polynomial-time randomized algorithm for \CCTUF{} problems with $m=3$.\footnote{In this context, we consider a randomized algorithm to be one that always correctly detects infeasibility of a problem, and finds a solution of a feasible problem with high probability $1-\sfrac1n$, where $n$ is the number of variables.}
\end{theorem}

As we show in \cref{app:unboundedness}, being able to solve feasibility problems is also enough to detect unboundedness of \CCTU{} problems.\footnote{Analogous to linear and integer programming, we call a \CCTU{} \emph{unbounded} if it is possible to achieve arbitrarily good objective values. Hence, having an unbounded feasible region does not imply unboundedness of the problem.}
One of the key ideas in the proof of \cref{thm:feasibilityMod3} is to reduce a \CCTUF{} problem to a hierarchy of slightly relaxed congruency-constrained problems with totally unimodular constraint matrices that we call \RCCTUF{} problems, and which we define as follows.

\begin{mdframed}[innerleftmargin=0.5em, innertopmargin=0.5em, innerrightmargin=0.5em, innerbottommargin=0.5em, userdefinedwidth=0.95\linewidth, align=center]
{\textbf{\boldmath$R$-Congruency-Constrained TU Feasibility (\hypertarget{prb:R-CCTUF}{$R$-CCTUF})\unboldmath:}}
Let $T\in\{-1, 0, 1\}^{k\times n}$ be a totally unimodular matrix and let $b\in\mathbb{Z}^k$. Additionally, let $m\in\mathbb{Z}_{> 0}$, $\gamma\in\mathbb{Z}^n$, and $R\subseteq \{0, \ldots, m-1\}$. The task is to show infeasibility or find a feasible solution of the system
\begin{equation*}
Tx \leq b,\ \gamma^\top x \in R \pmod*{m},\ x\in\mathbb{Z}^n \enspace.
\end{equation*}
\end{mdframed}

\noindent Here, the constraint $\gamma^\top x\in R\pmod*{m}$ is satisfied if and only if there exists an $r\in R$ such that $\gamma^\top x\equiv r\pmod*{m}$. We call $R$ the set of \emph{target residues}.
Clearly, every \CCTUF{} problem is an \RCCTUF{} problem with $R=\{r\}$.
Intuitively, the larger the set $R$ of target residues is, the easier the corresponding problem gets---in the extreme case of $|R|=m$, the congruency constraint is trivially fulfilled by any solution, and simply finding a solution of the TU problem without congruency constraint is enough.
Additionally, \RCCTUF{} problems can always be reduced to several problems of the same type with a smaller set of target residues. In particular, any \RCCTUF{} problem can be reduced to $|R|$ many \CCTUF{} problems, namely one for each $r\in R$.
Our new progress for \RCCTUF{} problems is going two steps into the hierarchy if the modulus $m$ is a prime number, i.e., we can solve feasibility problems with $|R|\geq m-2$.

\begin{theorem}\label{thm:R=m-2}
There is a strongly polynomial-time randomized algorithm for \RCCTUF{} problems with constant prime modulus $m$ and $|R|\geq m-2$.
\end{theorem}

Observing that for $m=3$, an \RCCTUF{} problem with $|R|=m-2$ is in fact a \CCTUF{} problem, \cref{thm:feasibilityMod3} immediately follows from \cref{thm:R=m-2}. Our proof of \cref{thm:R=m-2} is inspired by methods developed in~\cite{artmann_2017_strongly} for bimodular integer programs, but goes significantly beyond the strategy and techniques employed there.
In particular, we also decompose \RCCTUF{} problems into smaller ones following Seymour's decomposition of TU matrices, but we need methods that allow for progressing in the hierarchy of \RCCTUF{} problems introduced above. This step requires us to have prime modulus due to an application of the Cauchy-Davenport Inequality.
The decomposition approach deterministically reduces general \RCCTUF{} problems to problems with so-called \emph{base-block} constraint matrices.
While parity-constraints are fairly common in Combinatorial Optimization and known techniques could be leveraged in~\cite{artmann_2017_strongly} to solve parity-constrained base block problems, we present new approaches for $m>2$.
In particular, we create new links to recent advances on congruency-constrained submodular optimization and exact weight flow problems. The only known algorithm for exact weight flow problems is randomized, which is why we obtain a randomized algorithm as stated in \cref{thm:R=m-2} (and thus also in \cref{thm:feasibilityMod3}).

Interestingly, focusing on the case of $|R|=m-1$ only, our techniques lead to a substantially simpler approach for \RCCTUF{} problems that does not need to rely on decomposition methods and can therefore avoid both randomization and the prime modulus requirement, resulting in the following theorem.
\begin{theorem}\label{thm:R=m-1}
There is a strongly polynomial-time algorithm for \RCCTUF{} problems with $|R|=m-1$.
\end{theorem}

For $m=2$, \cref{thm:R=m-1} states that feasibility of parity-constrained TU problems can be decided efficiently.
This is a special case of bimodular IP feasibility, which has been known to admit polynomial time algorithms since the work of Veselov and Chirkov~\cite{veselov_2009_integer}.
Let us also remark that for general $m$, the congruency constraint in \RCCTUF{} problems with $|R|=m-1$ can be rewritten in the form $\gamma^\top x\not\equiv r\pmod*{m}$ for some residue $r$. Such constraint types and generalizations thereof have been studied in different settings already, in particular in the context of minimizing submodular functions (see Goemans and Ramakrishnan~\cite{goemans_1995_minimizing}, and Gr\"otschel, Lov\'asz, and Schrijver~\cite{groetschel_1993_geometric}).

Our approach for \cref{thm:R=m-1} is derived from interesting structural properties of \RCCTUF{} problems that are likely to be of independent interest, and two of which we want to highlight here.
One is concerned with \emph{flat directions} of the underlying polyhedron, i.e., vectors $d\in\mathbb{Z}^n\setminus \{0\}$ for which the \emph{width} $\max \{d^\top x \colon x \in \mathbb{Z}^n, Tx\leq b\} - \min  \{d^\top x\colon x\in \mathbb{Z}^n, Tx\leq b\}$ is small.
Prior to our work, results of this type have only been known for very restricted cases.
In particular, it is proved in Artmann's PhD thesis~\cite[Theorem 3.4]{artmann2020optimization} that for \CCTUF{} problems restricted to modulus $m=3$ and to base block constraint matrices, it holds that if the problem is infeasible, then a row of the constraint matrix is a flat direction of width $1$.
Our techniques show, through an arguably much simpler approach, that analogous results hold for arbitrary moduli $m$ and \CCTUF{} problems without any further restriction on the constraint matrix. Moreover, our result also generalizes to \RCCTUF{} problems, providing the following bound on the width, which can easily be seen to be tight.
\begin{theorem}\label{thm:flatDirections}
For every \RCCTUF{} problem, either there is a constraint matrix row that is a flat direction of the underlying polyhedron of width at most $m-|R|-1$, or a feasible solution of the \RCCTUF{} problem can be found in strongly polynomial time.
\end{theorem}

Finally, our techniques also lead to proximity results.
We call the problem obtained from \CCTU{}, \CCTUF{}, or \RCCTUF{} problems after dropping the congruency constraint the \emph{relaxation} of the respective problem.
Note that this relaxation is not a linear relaxation in the usual sense as we still require integral solutions, but is nonetheless closely related to it due to the totally unimodular constraint matrices.
Prior knowledge of proximity results in this context have been very limited.
In particular, it was known~\cite[Lemma 3.3]{artmann2020optimization} that given a feasible \CCTU{} problem with $m=3$, then for any vertex $y\in\mathbb{Z}^n$ of the underlying polyhedron $\{x\in\mathbb{R}^n\colon Tx\leq b\}$, there exists a feasible solution $x$ of the \CCTU{} problem such that $\|y-x\|_{\infty}\leq 2$.
While the method used in~\cite{artmann2020optimization} is specific for the $m=3$ case, our techniques lead to the following more general result for arbitrary modulus $m$ and, again, the more general congruency-constraint type. Here, $R$-CCTU denotes the optimization versions of \RCCTUF{} problems, analogous to the relation between \CCTU{} and \CCTUF{} problems. In other words, an $R$-CCTU problem is a \CCTU{} problem where the congruency-constraint $\gamma^\top x\equiv r\pmod*{m}$ for a single residue $r$ is replaced by $\gamma^\top x\in R\pmod*{m}$ for a set $R$ of residues.

\begin{theorem}\label{thm:proximity} Consider a feasible $R$-CCTU problem with modulus $m$.
\begin{enumerate}
\item\label{thmitem:Fproximity} For any $x_0$ feasible for the relaxation, there is an $x$ feasible for the problem with $\|x-x_0\|_{\infty}\leq m-|R|$.
\item\label{thmitem:proximity} For any $x_0$ optimal for the relaxation, there is an $x$ optimal for the problem with $\|x-x_0\|_{\infty}\leq m-|R|$, and vice versa.
\end{enumerate}
Moreover, in~\ref{thmitem:Fproximity} and~\ref{thmitem:proximity}, given $x_0$ and any feasible or optimal solution of the $R$-CCTU problem, respectively, a solution $x$ with the stated properties can be found in strongly polynomial time.
Also, in~\ref{thmitem:proximity}, given $x$, a solution $x_0$ with the stated properties can be found in strongly polynomial time.
\end{theorem}

\subsection{Related work\MOORdot}

The maximum absolute value $\Delta$ of a subdeterminant of the constraint matrix is a parameter that has received significant attention in integer programming recently.
The closely related problem class of congruency-constrained combinatorial optimization problems has been investigated already in the early 80's for the parity-constrained case, and several further advances have been made since.
We briefly recap prior work linked to these areas.

A problem that can be cast as a bounded subdeterminant integer program, has gained substantial interest recently~\cite{bock_2014_solving,conforti_2020_stableset,conforti_2021_extended}, and was resolved in~\cite{fiorini_2021_integerPrograms}, is the stable set problem in graphs $G$ with bounded odd cycle packing number $\mathrm{ocp}(G)$, i.e., graphs for which the maximum number of disjoint odd cycles is bounded.
The incidence matrix of such a graph has maximum subdeterminant $2^{\mathrm{ocp}(G)}$ (see, e.g., \cite{grossman_1995_minors}).
Several further interesting results link the parameter $\Delta$ to properties of integer programs, their relaxations, and underlying polyhedra (see, e.g.,~\cite{bonifas_2012_subdetDiameter,eisenbrand_2017_geometric,lee_2020_improvingProximity,lee_2021_polynomial,paat_2021_integralitynumber,tardos_1986_strongly} and references therein).
Furthermore, there has been interesting recent progress on the problem of approximating the largest subdeterminant of a matrix (see \textcite{summa_2015_largest}, and \textcite{nikolov_2015_randomized}).
Also, IPs with more constrained subdeterminant structures that admit efficient algorithms for integer programming were considered \cite{veselov_2009_integer,artmann_2016_nondegenerate,glanzer_2021_abcRecognition}.

One of the most classical congruency-constrained combinatorial optimization problems is the minimum odd cut problem, which asks to find a minimum cut among all cuts with an odd number of vertices. \textcite{padberg_1982_odd} presented a first efficient method for the minimum odd cut problem. Subsequently, \textcite{barahona_1987_construction} showed that efficient minimization is also possible over all cuts with an even number of vertices.
Later works by \textcite{groetschel_1984_corrigendum}, and by \textcite{goemans_1995_minimizing} generalized these results to the minimization of submodular functions. More precisely, the approach of~\cite{goemans_1995_minimizing} allows for minimizing over so-called triple families, which includes the case of cuts $C\subseteq V$ of cardinality \emph{not} congruent to $r$ modulo $m$, for any integers $r$ and $m$. \textcite{nagele_2018_submodular} showed that a submodular function can also be efficiently minimized over sets of cardinality $r \pmod*{m}$, for any integer $m$ that is a constant prime power. For the special case of minimum cuts, \textcite{nagele_2020_newContraction} presented a randomized PTAS for finding a minimum cut among all cuts containing~$r\pmod*{m}$ many vertices, for any constant $m$.

\subsection{Organization of the paper\MOORdot}

In \cref{sec:overview}, we present the key ideas and techniques that lead to our new results.
In particular, \cref{sec:overviewDecompAndFlat} presents a decomposition lemma, a crucial ingredient that is central to all our results, and we showcase its strength by readily deducing from it our flatness and proximity results (\cref{thm:flatDirections,thm:proximity}).
Subsequently, \cref{sec:overviewRCCTUF} gives an overview of our approach to \CCTUF{} problems and the proof of \cref{thm:R=m-2}.

A proof of the decomposition lemma as well as more applications thereof (in particular, \cref{thm:R=m-1}), are given in \cref{sec:decompLemma}, while \cref{sec:baseBlocks,sec:patterns} fill in details and present the missing proofs from \cref{sec:overviewRCCTUF}.

\section{\boldmath Overview of our approach\MOORdot \unboldmath}\label{sec:overview}

\subsection{Decomposition, flat directions, and proximity\MOORdot}\label{sec:overviewDecompAndFlat}

One technique that we employ repeatedly is a careful decomposition of vectors into well-structured ones. In particular, we often apply such decomposition to solutions of \CCTUF{} or \RCCTUF{} problems, to obtain a structured sum of other vectors. A key role in this decomposition is taken by \emph{elementary vectors}, which we define as follows.

\begin{definition}
Let $T\in\mathbb{Z}^{k\times n}$ be a totally unimodular matrix.
\begin{enumerate}
\item A vector $d\in\mathbb{Z}^n$ is \emph{TU-appendable} to $T$ if the matrix $\begin{psmallmatrix}
T\\ d^\top
\end{psmallmatrix}$ is totally unimodular.
\item A vector $x\in\mathbb{Z}^n$ is \emph{elementary} w.r.t.~$T$ if $d^\top x\in\{-1,0,1\}$ for all $d$ that are TU-appendable to $T$.
\end{enumerate}
\end{definition}

Concretely, we obtain the following decomposition lemma.
We remark that here and throughout this paper, we use the shorthand notation $[n]\coloneqq \{1,\ldots,n\}$ for $n\in\mathbb{Z}_{\geq 1}$.
\begin{lemma}[Decomposition lemma]\label{lem:decompositionLemma}
Let $T\in\{-1, 0, 1\}^{k\times n}$ be a totally unimodular matrix, let $b\in\mathbb{Z}^k$, and let $x_0, y\in \mathbb{Z}^n$ be two solutions of the system $Tx\leq b$. Then, we can determine in strongly polynomial time $y^1,\ldots,y^n\in\mathbb{Z}^n$ and $\lambda_1,\ldots,\lambda_n\in\mathbb{Z}_{\geq 0}$ such that $y-x_0=\sum_{i=1}^n \lambda_i y^i$ with the following properties:
\begin{enumerate}
\item\label{lemitem:unitScalarProduct}
$y^1,\ldots,y^n$ are elementary with respect to $T$.
\item\label{lemitem:feasibleSubsum}
For $\mu_1,\ldots,\mu_n\in\mathbb{Z}_{\geq 0}$ with $\mu_i\leq\lambda_i$ for all $i\in[n]$, the vector $\tilde{y}\coloneqq x_0 + \sum_{i=1}^n \mu_i y^i$ satisfies $T\tilde y\leq b$.
\end{enumerate}
\end{lemma}
In words, the above decomposition lemma allows for efficiently writing a solution $y$ to the relaxation of a \CCTUF{} (or, more generally, also \RCCTUF{}) as a sum of another solution $x_0$ and a combination of elementary vectors $y^i$ that can moreover be freely combined to obtain other solutions to the relaxation. A formal proof of this decomposition lemma is given in \cref{sec:proofDecompLemma}.

One of our applications of the decomposition lemma is to bound the search space in which we need to look for solutions of \RCCTUF{} problems.
Note that given a solution $x_0$ of the relaxation of an \RCCTUF{} problem and any feasible \RCCTUF{} solution $y$, i.e., one that also satisfies the congruency constraint, as well as a TU-appendable row $d^\top$,
\cref{lem:decompositionLemma} allows for efficiently decomposing $y-x_0$ into a sum of the form $\sum_{i=1}^n \lambda_i y^i$ with $\sum_{i=1}^n\lambda_i \geq |d^\top (y-x_0)|$.
Hence, if $|d^\top (y-x_0)|$ is large, the sum $\sum_{i=1}^n \lambda_i y^i$ has many terms, and due to point~\ref{lemitem:feasibleSubsum}, there are many options to build new solutions $x_0+\sum_{i=1}^n\mu_i y^i$ of the relaxation of the \RCCTUF{} problem by removing an arbitrary subset of the terms (i.e., choosing $\mu_i\in\{0,\ldots,\lambda_i\}$).
Thus, in order to obtain a new feasible solution for the \RCCTUF{} problem, we have to make sure that $\gamma^\top x_0 + \sum_{i=1}^n \mu_i \gamma^\top y^i \in R \pmod*{m}$, i.e., that we hit a feasible residue again.
The following lemma shows that there always exists such a choice with $\sum_{i=1}^n \mu_i \leq m-|R|$.

\begin{lemma}\label{lem:feasibleResidueSum}
Let $m\in\mathbb{Z}_{>0}$, $R\subseteq\{0,\ldots, m-1\}$, and $r_1,\ldots,r_\ell\in\mathbb{Z}$ with $\sum_{i\in[\ell]} r_i \in R\pmod*{m}$. If there is no interval $I = \{i_1,\ldots,i_2\}$ with $i_1,i_2\in [\ell]$ and $i_1<i_2$ such that $\sum_{i\in [\ell]\setminus I} r_i\in R$, then $\ell\leq m-|R|$.
\end{lemma}
\begin{proof}
Assume for the sake of deriving a contradiction that there is no interval $I\subseteq [\ell]$ such that $\sum_{i\in [\ell]\setminus I} r_i\in R$, but $\ell\geq m-|R|+1$.
Consider the $\ell$ integers $s_0 = 0$, $s_1=r_1$, \ldots, $s_{\ell-1} = r_1 + \ldots + r_{\ell-1}$. Observe that $s_j\notin R \pmod*{m}$ for all $j\in [\ell-1]$; for otherwise, there is an interval $I=\{j+1,\ldots, \ell\}$ for some $j\in [\ell-1]$ such that $\sum_{i\in[\ell]\setminus I} r_i = s_j\in R\pmod*{m}$, contradicting the assumption. Thus, $s_j\in \{0,\ldots,m-1\}\setminus R\pmod*{m}$ for $j\in [\ell-1]$. Hence, because $\ell\geq m-|R|+1$, we have by the pigeonhole principle that there exist distinct $j_1,j_2\in [\ell-1]$ such that $s_{j_1}\equiv s_{j_2}\pmod*{m}$. Thus, $I=\{j_1+1, \ldots, j_2\}$ is an interval with $\sum_{i\in[\ell]\setminus I} r_i = \sum_{i\in[\ell]} r_i - (s_{j_2} - s_{j_1}) \equiv \sum_{i\in[\ell]} r_i \in R\pmod*{m}$, again contradicting the assumption and hence completing the proof.
\end{proof}

Indeed, \cref{lem:feasibleResidueSum} shows that as long as the sum
$$ \underbrace{\gamma^\top y^1 + \ldots + \gamma^\top y^1}_{\text{$\lambda_1$ many terms}} + \ldots + \underbrace{\gamma^\top y^n + \ldots + \gamma^\top y^n}_{\text{$\lambda_n$ many terms}} \in R - \gamma^\top x_0 \pmod{m}
$$
has at least $m-|R|+1$ many terms, there is a subset of consecutive terms that can be removed while keeping the total residue inside the set $R - \gamma^\top x_0$.
Iterating the procedure eventually leaves us with terms corresponding to a solution of the form $\tilde y \coloneqq x_0+\sum_{i=1}^n\mu_i y^i$ with $\sum_{i=1}^n \mu_i \leq m-|R|$.
Observe that this solution $\tilde y$ is close to the solution $x_0$ of the relaxation of the initial problem in the sense that $|d^\top (\tilde y - x_0)| \leq m-|R|$, which can be used as a bound for the search space when looking for feasible solutions.
Beyond that, the idea described above is also at the heart of our flatness and proximity results (\cref{thm:flatDirections,thm:proximity}).

One caveat in the above construction is that a direct realization of the approach suggested by \cref{lem:feasibleResidueSum} may have a worst-case running time polynomial in $m$, which is not polynomial in the input size of the \RCCTUF{} problem when $m$ is part of the input.
Interestingly, given a sum $\sum r_i$ that lies in $R\pmod*{m}$ for residues $r_i\in\mathbb{Z}$ and a set $R\subseteq\{0,\ldots, m-1\}$, it is generally NP-hard to find a smallest possible number of terms $r_i$ that also sum to a residue in $R$ modulo $m$, as can be seen by a reduction from the Subset Sum problem, for example.
Nonetheless, we are able to get the following constructive result by exploiting that the sum $\sum_{i=1}^n\lambda_i y^i$ contains no more than $n$ distinct vectors $y^i$, and the fact that we do not need to find a shortest partial sum with residue in $R-\gamma^{\top} x_0$ but only one with at most $m-|R|$ terms. Its formal proof is postponed to \cref{sec:proofDecompLemma}.

\begin{lemma}\label{lem:transformSolutionEfficiently}
Consider an \RCCTUF{} problem with modulus $m$, constraint matrix $T$, a feasible solution $y$, and a solution $x_0$ of its relaxation.
We can obtain in strongly polynomial time a feasible solution $\tilde y$ such that $x_0+y-\tilde y$ is feasible for the relaxation, as well, and
\begin{enumerate}
\item\label{lemitem:smallProducts} for any $d\in\mathbb{Z}^n$ that is TU-appendable to $T$, we have $d^\top (\tilde y-x_0) \leq m-|R|$, and
\item\label{lemitem:objective} for any $c\in\mathbb{Z}^n$ such that $x_0$ minimizes $c^\top x$ over the relaxation of the
\RCCTUF{} problem, $c^\top \tilde y \leq c^\top y$.
\end{enumerate}
\end{lemma}

Note that point~\ref{lemitem:objective} adds an additional property on the relation of the costs of the three vectors $x_0$, $y$, and $\tilde y$ that is useful in optimization settings.
To showcase two concrete applications of \cref{lem:transformSolutionEfficiently} in this overview, we show how \cref{lem:transformSolutionEfficiently} readily implies our flatness and proximity results, i.e., \cref{thm:flatDirections,thm:proximity}.
We start by showing \cref{thm:flatDirections}, which is a consequence of the following statement.
\begin{lemma}\label{lem:flatOrIrrelevant}
Consider an \RCCTUF{} problem, and let $d^\top x\leq \beta$ be one of its constraints. Either
\begin{enumerate}
\item $d$ is a flat direction of width at most $m-|R|-1$ for the underlying polyhedron, or
\item\label{lemitem:dropConstraint} the problem is feasible if and only if the \RCCTUF{} problem without the constraint $d^\top x\leq \beta$ is feasible.
\end{enumerate}
In case~\ref{lemitem:dropConstraint}, a solution of the initial problem can be obtained in strongly polynomial time from any solution of the initial problem without the constraint $d^\top x\leq \beta$.
\end{lemma}

\begin{proof}
Assume that $d$ is a direction of width at least $m-|R|$, and let $x_0$ be feasible for the relaxation of the \RCCTUF{} problem such that $d^\top x_0 \leq \beta - m + |R|$.
It is enough to show that we can in strongly polynomial time obtain a feasible solution of the initial problem, assuming that we are given a feasible solution $y$ of the problem without the constraint $d^\top x\leq \beta$.
Applying \cref{lem:transformSolutionEfficiently} in this setting, we get that given $y$, we can in strongly polynomial time obtain another feasible solution $\tilde y$ such that $d^\top \tilde y \leq d^\top x_0 + m-|R| \leq \beta$, i.e., a solution that also satisfies the constraint $d^\top x\leq \beta$.
This proves the desired statement.
\end{proof}

\begin{proof}[Proof of \cref{thm:flatDirections}]
Consider an \RCCTUF{} problem and one of its constraints $d^\top x\leq \beta$.
Using a result of Tardos~\cite{tardos_1986_strongly}, we can in strongly polynomial time determine whether this constraint identifies a direction of width at most $m-|R|-1$ of the underlying polyhedron (namely, by optimizing the objectives $d^\top x$ and $-d^\top x$ over the polyhedron).
If not, by \cref{lem:flatOrIrrelevant}, the constraint can be dropped without changing the feasibility status.
Iterating over all constraints, we either find a flat direction, or we end up with a problem without inequality constraints that is trivially feasible, thus implying that the initial problem was feasible as well.
In that case, a solution of the initial problem can be constructed within the desired running time from a solution of the final problem through \cref{lem:flatOrIrrelevant}.
\end{proof}
Let us remark that the width $m-|R|-1$ of flat directions in infeasible problems is best possible for any size of $R$, as can be seen from the infeasible problems given by $\{x\in\mathbb{Z}\colon 0\leq x \leq m-\ell-1, x \in R_\ell\pmod{m}\}$ with $R_\ell=\{m-\ell,\ldots,m-1\}$ for $\ell\in[m-1]$.

\smallskip

Finally, we also show how \cref{lem:transformSolutionEfficiently} implies \cref{thm:proximity}. More precisely, we prove the following generalization, from which \cref{thm:proximity} follows immediately.

\begin{theorem}\label{thm:proximityGeneral}
Consider a feasible $R$-CCTU problem with modulus $m$ and constraint matrix $T$.
\begin{enumerate}
\item\label{thmitem:FproximityGeneral} For any feasible solution $x_0$ of the relaxation, there is a feasible solution $x$ of the $R$-CCTU problem such that for every vector $d$ that is TU-appendable to T, we have $d^\top(x-x_0)\leq m-|R|$.
\item\label{thmitem:proximityGeneral} For any optimal solution $x_0$ of the relaxation, there is an optimal solution $x$ of the $R$-CCTU problem such that for every vector $d$ that is TU-appendable to T, we have $d^\top(x-x_0)\leq m-|R|$, and vice versa.
\end{enumerate}
Moreover, in~\ref{thmitem:Fproximity} and~\ref{thmitem:proximityGeneral}, given $x_0$ and any feasible or optimal solution of the $R$-CCTU problem, respectively, a solution $x$ with the stated properties can be found in strongly polynomial time.
Also, in~\ref{thmitem:proximityGeneral}, given $x$, a solution $x_0$ with the stated properties can be found in strongly polynomial time.
\end{theorem}

\begin{proof}
For part~\ref{thmitem:FproximityGeneral}, apply \cref{lem:transformSolutionEfficiently} to the given problem with feasible solutions $y$ and $x_0$ of the problem and its relaxation, respectively, to obtain a feasible solution $\tilde y$.
Property~\ref{lemitem:smallProducts} in \cref{lem:transformSolutionEfficiently} states that $d^\top (\tilde y-x_0)\leq m-|R|$ for any $d\in\mathbb{Z}^n$ that is TU-appendable to the constraint matrix.
Moreover, if $y$ is given, we can also obtain $\tilde y$ in strongly polynomial time by \cref{lem:transformSolutionEfficiently}, hence $\tilde y$ has the properties of the solution $x$ claimed by \cref{thm:proximityGeneral}.

To also deduce the first part of~\ref{thmitem:proximityGeneral}, we proceed identically, but take $x_0$ to be an optimal solution of the relaxation with respect to the minimization objective $c^\top x$, and $y$ an optimal solution to the problem. In that case, on top of what we derived before, $\tilde y$ satisfies $c^\top \tilde y\leq c^\top y$ by property~\ref{lemitem:objective} in \cref{lem:transformSolutionEfficiently}.
Thus, because $y$ is optimal, this must be an equality and $\tilde y$ is optimal, as well.

For the other direction of~\ref{thmitem:proximityGeneral}, where we are given an optimal solution $x$ of the $R$-CCTU problem, we first determine any optimal solution $x_0$ of the relaxation.
This can be done in strongly polynomial time using the framework of Tardos~\cite{tardos_1986_strongly}.
Next, by applying \cref{lem:transformSolutionEfficiently} to $x$ and $x_0$, we can in strongly polynomial time obtain a feasible solution $\tilde x$ of the $R$-CCTU problem with $c^\top x\geq c^\top \tilde x$ such that $d^\top (\tilde x-x_0)\leq m-|R|$ for any $d\in\mathbb{Z}^n$ that is TU-appendable to the constraint matrix.
We claim that $\bar x_0\coloneqq x_0 + x - \tilde x$ has the desired properties.
First, $\bar x_0$ is feasible for the relaxation by \cref{lem:transformSolutionEfficiently}; additionally, because $x$ is an optimal solution of the $R$-CCTU problem, we must have $c^\top x = c^\top \tilde x$, hence $c^\top \bar x_0 = c^\top x_0$, and hence $\bar x_0$ must in fact be an optimal solution of the relaxation.
Moreover, for any $d\in\mathbb{Z}^n$ that is TU-appendable to the constraint matrix, we have $d^\top (x-\bar x_0) = d^\top (\tilde x-x_0) \leq m-|R|$, as desired.
\end{proof}

\begin{proof}[Proof of \cref{thm:proximity}]
Note that for every $i\in [n]$, the unit vector $e_i$ and its negative $-e_i$ are TU-appendable to every totally unimodular matrix. Thus, the solutions guaranteed by \cref{thm:proximityGeneral} satisfy
\begin{equation*}
\|x - x_0 \|_\infty = \max_{i\in[n]}\, \max\{e_i^\top (x - x_0), -e_i^\top (x - x_0)\} \leq m-|R|\enspace.\qedhere
\end{equation*}
\end{proof}

We postpone further applications of the decomposition lemma to \cref{sec:decompLemma}, and continue with an overview of our approach to deal with \RCCTUF\ problems.
The above discussion aimed at exemplifying how the decomposition lemma can be employed, and should help to better understand further implications, including settings that we state in the following overview of how to deal with \RCCTUF\ problems.

\subsection[Overview of our approach to \texorpdfstring{$R$}{R}-CCTUF problems and \texorpdfstring{\cref{thm:R=m-2}}{Theorem~\ref{thm:R=m-2}}]{\boldmath Overview of our approach to \RCCTUF{} problems and \cref{thm:R=m-2}\MOORdot\unboldmath}\label{sec:overviewRCCTUF}

When approaching \RCCTUF{} problems of the form
\begin{equation*}
Tx \leq b,\enspace \gamma^\top x \in R \pmod*{m},\enspace x\in\mathbb{Z}^n
\end{equation*}
with constant prime modulus $m$, we follow the general idea of decomposing the problem into smaller ones by applying Seymour's TU decomposition to the constraint matrix $T$. Exploiting Seymour's decomposition to approach problems that involve TU matrices is a standard approach that has been successfully used in a variety of contexts (see for example~\cite{dinitz_2014_matroid,artmann_2017_strongly,aprile_2021_regular}).
In particular, this includes the solution to parity-constrained TU problems presented in~\cite{artmann_2017_strongly}. However, going to congruency-constraints with modulus $3$ or larger creates substantial extra hurdles beyond prior techniques.
For completeness and clear references, we repeat Seymour's TU decomposition framework here, which breaks a TU matrix into smaller ones using so-called $1$-, $2$-, and $3$-sums, and pivoting operations, which are defined as follows.

\begin{definition}[$1$-, $2$-, and $3$-sums]
Let $A\in\mathbb{Z}^{k_A\times n_A}$, $B\in\mathbb{Z}^{k_B\times n_B}$, $e\in\mathbb{Z}^{k_A}$, $f\in\mathbb{Z}^{n_B}$, $g\in\mathbb{Z}^{k_B}$,  $h\in\mathbb{Z}^{n_A}$.
\begin{enumerate}
\item The \emph{$1$-sum} of $A$ and $B$ is
$
A \ksum[1] B \coloneqq \begin{psmallmatrix}
A & 0 \\ 0 & B
\end{psmallmatrix}
$.

\item The \emph{$2$-sum} of $\begin{pmatrix}
A & e
\end{pmatrix}$ and $\begin{psmallmatrix}
f^\top \\ B
\end{psmallmatrix}$ is
$
\begin{pmatrix}
A & e
\end{pmatrix} \ksum[2] \begin{psmallmatrix}
f^\top \\ B
\end{psmallmatrix} \coloneqq \begin{psmallmatrix}
A & ef^\top \\ 0 & B
\end{psmallmatrix}
$.
\item The \emph{$3$-sum} of $\begin{psmallmatrix}
A & e & e \\ h^\top & 0 & 1
\end{psmallmatrix}$ and $\begin{psmallmatrix}
0 & 1 & f^\top \\ g & g & B
\end{psmallmatrix}$ is
$
\begin{psmallmatrix}
A & e & e \\ h^\top & 0 & 1
\end{psmallmatrix} \ksum[3] \begin{psmallmatrix}
0 & 1 & f^\top \\ g & g & B
\end{psmallmatrix} \coloneqq \begin{psmallmatrix}
A & ef^\top \\ gh^\top & B
\end{psmallmatrix}
$.
\end{enumerate}
\end{definition}

\begin{definition}[Pivoting]
\label{def:pivoting}
Let $C\in\mathbb{Z}^{k\times n}$, $p\in\mathbb{Z}^n$, $q\in\mathbb{Z}^k$, and $\varepsilon\in\{-1,1\}$. The matrix obtained from pivoting on $\varepsilon$ in $T \coloneqq \begin{psmallmatrix}
\varepsilon & p^\top \\ q & C
\end{psmallmatrix}$, i.e., pivoting on the element $T_{11}$ of $T$, is
$
\pivot[11](T) \coloneqq \begin{psmallmatrix}
-\varepsilon & \varepsilon p^\top \\ \varepsilon q & C - \varepsilon q p^\top
\end{psmallmatrix}$.
More generally, $\pivot[ij](T)$ for indices $i$ and $j$ such that $T_{ij}\in\{-1,1\}$ is obtained from $T$ by first permuting rows and columns such that the element $T_{ij}$ is permuted to the first row and first column, then performing the above pivoting operation on the permuted matrix, and finally reversing the row and column permutations.
\end{definition}

It is well-known that a $1$-, $2$-, and $3$-sum is totally unimodular if and only if the two summands it is obtained from are, and a pivoted matrix is totally unimodular if and only if the original matrix is. Seymour's TU decomposition theorem states that a TU matrix is either very structured, or it can be decomposed using $1$-, $2$-, and $3$-sums, or pivoting steps. We use the following variation of the decomposition theorem, which provides some extra guarantees on the dimensions of the matrices appearing in the decomposition. It readily follows from classical statements of Seymour's decomposition for TU matrices (see~\cref{sec:seymour} for details).

\begin{theorem}[Seymour's TU decomposition]\label{thm:TUdecomp}
Let $T\in\mathbb{Z}^{k\times n}$ be a totally unimodular matrix. Then, one of the following cases holds.
\begin{enumerate}
\item\label{thmitem:TUdecomp_netw} $T$ or $T^\top$ is a network matrix.

\item\label{thmitem:TUdecomp_const} $T$ is, possibly after iteratively applying the operations of
\begin{itemize}
\item deleting a row or column with at most one non-zero entry,
\item deleting a row or column that appears twice or whose negation also appears in the matrix, and
\item changing the sign of a row or column,
\end{itemize}
equal to one of
\begin{equation*}
\begin{psmallmatrix*}[r]
 1 & -1 &  0 &  0 & -1 \\
-1 &  1 & -1 &  0 &  0 \\
 0 & -1 &  1 & -1 &  0 \\
 0 &  0 & -1 &  1 & -1 \\
-1 &  0 &  0 & -1 &  1
\end{psmallmatrix*}
\quad\text{and}\quad
\begin{psmallmatrix}
 1 &  1 &  1 &  1 &  1 \\
 1 &  1 &  1 &  0 &  0 \\
 1 &  0 &  1 &  1 &  0 \\
 1 &  0 &  0 &  1 &  1 \\
 1 &  1 &  0 &  0 &  1
\end{psmallmatrix}\enspace.
\end{equation*}

\item\label{thmitem:TUdecomp_sum} $T$ can, possibly after row and column permutations, be decomposed into a $1$-, $2$-, or $3$-sum of totally unimodular matrices with $n_A, n_B\geq 2$.

\item\label{thmitem:TUdecomp_pivot} $T$ can, after pivoting once and possibly performing row and column permutations, be decomposed into a $3$-sum of totally unimodular matrices with $n_A, n_B\geq 2$.
\end{enumerate}
Additionally, we can in time $\mathrm{poly}(n)$ decide which of the cases holds and determine the involved matrices.
\end{theorem}

Cases~\ref{thmitem:TUdecomp_netw} and~\ref{thmitem:TUdecomp_const} are the cases where $T$ is a so-called \emph{base block} matrix.
We exploit the structure of those matrices to reduce \CCTUF{} problems with such a constraint matrix $T$ to certain combinatorial optimization problems with congruency constraints.
In particular, if $T$ is a network matrix, the corresponding problem can be interpreted as a congruency-constrained circulation problem. Here, we exploit a connection to exact weight matching problems \cite{camerini_1992_rpp} that results in an efficient randomized procedure.
For $T$ being the transpose of a network matrix, we present a reduction to a congruency-constrained submodular minimization problem, which can be solved (whenever $m$ is a prime power) by a recent algorithm by \textcite{nagele_2018_submodular}.
We expand on these connections in \cref{sec:baseBlocks}, thereby obtaining the following statement on the corresponding feasibility problems.

\begin{theorem}\label{thm:baseBlocks}
Let $T$ be a TU matrix for which case~\ref{thmitem:TUdecomp_netw} or~\ref{thmitem:TUdecomp_const} in \cref{thm:TUdecomp} holds. There is a strongly polynomial time randomized algorithm for \CCTUF{} problems with constraint matrix $T$ and constant prime power modulus.
\end{theorem}

In the cases where the constraint matrix $T$ admits a decomposition as a $1$-, $2$-, or $3$-sum, i.e., case~\ref{thmitem:TUdecomp_sum} of \cref{thm:TUdecomp}, we can write $T=\begin{psmallmatrix}A & ef^\top \\ gh^\top & B\end{psmallmatrix}$. If $T$ is a $2$-sum, $g$ and $h$ will be zero vectors; if $T$ is a $1$-sum, also $e$ and $f$ will be zero vectors.
This matrix decomposition splits the variables $x$, the right-hand sides $b$, and the residue vector $\gamma$ into two parts accordingly. The \RCCTUF{} problem can then be rewritten as the problem of finding a feasible solution of the system
\begin{equation}\label{eq:structured-problem}
\begin{aligned}
\begin{pmatrix}
A & ef^\top \\ gh^\top & B
\end{pmatrix} \cdot \begin{pmatrix}
x_A \\ x_B
\end{pmatrix} &\leq \begin{pmatrix}
b_A \\ b_B
\end{pmatrix} \\
\gamma_A^\top x_A + \gamma_B^\top x_B & \in R \pmod{m} \\
x_A \in \mathbb{Z}^{n_A},\
x_B & \in \mathbb{Z}^{n_B}\enspace.
\end{aligned}
\end{equation}
For any fixed values of $\alpha\coloneqq f^\top x_B$ and $\beta\coloneqq h^\top x_A$, the above problem can be split into the two almost independent \CCTUF{} problems
\begin{equation}\label{eq:A-B-problem}\begin{minipage}{4cm}
\setlength{\abovedisplayskip}{0pt}$$\begin{aligned}
A x_A & \leq b_A - \alpha e \\
h^\top x_A & = \beta \\
\gamma_A^\top x_A & \equiv r_A \pmod{m} \\
x_A & \in \mathbb{Z}^{n_A}
\end{aligned}$$
\end{minipage}
\qquad\text{and}\qquad
\begin{minipage}{4cm}
\setlength{\abovedisplayskip}{0pt}$$\begin{aligned}
B x_B & \leq b_B - \beta g \\
f^\top x_B & = \alpha \\
\gamma_B^\top x_B & \equiv r_B \pmod{m} \\
x_B & \in \mathbb{Z}^{n_B}
\end{aligned}$$
\end{minipage}
\enspace,
\end{equation}
where we would like to find solutions $x_A$ and $x_B$ for residues $r_A$ and $r_B$ such that $r_A + r_B \in R \pmod{m}$.
Hence, this desired relation between the target residues $r_A$ and $r_B$ is the only dependence between the two problems once $\alpha$ and $\beta$ are fixed.
We refer to the problem on the left as the \emph{A-problem} and the problem on the right as the \emph{B-problem}.

A solution of the initial \RCCTUF{} problem can only exist for pairs $(\alpha,\beta)\in\mathbb{Z}^2$ for which both the $A$- and the $B$-problem are feasible. We denote this set by $\Pi\subseteq\mathbb{Z}^2$.
In \cref{sec:patterns}, we will see that $\Pi$ is a polyhedron that can be obtained by essentially projecting feasible solutions of the relaxation of our \RCCTUF{} problem down to the $(\alpha,\beta)$-space.
This will allow us to deduce structural properties of $\Pi$.
For now, we aim at narrowing down the values of $(\alpha, \beta)\in\Pi$ that we have to consider for finding a feasible solution. To this end, we use the following Lemma.
\begin{lemma}\label{lem:patternsBounds}
Consider an \RCCTUF{} problem of the form given in~\eqref{eq:structured-problem}. We can in strongly polynomial time obtain $\ell_i, u_i\in\mathbb{Z}$ with $u_i-\ell_i\leq m-|R|$ for $i\in\{0,1,2\}$ such that if the \RCCTUF{} problem has a solution, then it has one with $\ell_0 \leq \alpha+\beta \leq u_0$, $\ell_1 \leq \alpha \leq u_1$, and $\ell_2\leq \beta \leq u_2$, where $\alpha=f^\top x_B$ and $\beta=h^\top x_A$.
\end{lemma}
Note that $\alpha$, $\beta$, and $\alpha+\beta$ are scalar products of a solution of~\eqref{eq:structured-problem} with suitably chosen row vectors.
We show in \cref{sec:boundedProducts} that those rows are all TU-appendable to the constraint matrix, thus enabling the application of techniques from the previous section to prove existence of solutions with those scalar products bounded to the desired range.
Here, as a consequence of \cref{lem:patternsBounds}, we can restrict our attention to $O(m^2)$ many pairs $(\alpha, \beta)$ in the \emph{narrowed} set
\begin{equation*}
\Pi_{\text{narrowed}}\coloneqq \Pi\cap\left\{(\alpha, \beta)\in\mathbb{Z}^2 \colon \ell_0 \leq \alpha+\beta \leq u_0, \ell_1 \leq \alpha \leq u_1, \ell_2\leq \beta \leq u_2\right\}\enspace.
\end{equation*}
We will later see that properties of $\Pi$ imply that we can choose $\ell_i, u_i$ such that we even have
\begin{equation*}
\Pi_{\text{narrowed}}=\left\{(\alpha, \beta)\in\mathbb{Z}^2 \colon \ell_0 \leq \alpha+\beta \leq u_0, \ell_1 \leq \alpha \leq u_1, \ell_2\leq \beta \leq u_2\right\}\enspace.
\end{equation*}

One natural attempt at this point would be to explicitly try all $O(m^4)$ remaining combinations of $r_A$, $r_B$, and $(\alpha,\beta)\in\Pi_{\text{narrowed}}$, and recurse on the corresponding (now independent) $A$- and $B$-problems in~\eqref{eq:A-B-problem}.
If we could guarantee that both problems had about the same number of variables in each such step (more precisely, at least a constant fraction of the original variables), this would lead to a polynomial time procedure at least for constant moduli $m$:
The number of variables would go down by roughly a factor of two in every step; hence we would fall back to cases~\ref{thmitem:TUdecomp_netw} or~\ref{thmitem:TUdecomp_const} of \cref{thm:TUdecomp} after $O(\log n)$ many iterations at the latest, each increasing the number of subproblems by a factor of $O(m^4)$, giving a total running time bound of $m^{O(\log n)}$.\footnote{More generally, this enumerative approach is efficient whenever the depth of Seymour's decomposition is at most logarithmic in the input size.}
Unfortunately, the guarantees of \cref{thm:TUdecomp} are much weaker: We can only guarantee that both $A$ and $B$ have at least two columns, and if their sizes happen to be imbalanced in most decomposition steps, the above argument fails.

Still, we can always solve the relaxations of both problems for all $(\alpha,\beta)\in\Pi_{\text{narrowed}}$.
Without loss of generality, let us assume that the $B$-problem is the smaller among the $A$- and the $B$-problem (with respect to the number of columns in its constraint matrix, i.e., the number of variables).
Because $B$ has at most half the number of columns compared to $T$, it turns out that we can afford (in terms of running time) to recursively call \cref{thm:R=m-2} on an \RCCTUF{} version of the $B$-problem (i.e., the $B$-problem in~\eqref{eq:A-B-problem} with the congruency constraint replaced by $\gamma_B^\top x_B\in R_B\pmod*{m}$, for sets $R_B$ of the same size as the set $R$ in the original problem). Concretely, for any fixed $(\alpha, \beta)\in\Pi_{\text{narrowed}}$, at most $m-|R|+1$ such recursive calls suffice to determine up to $m-|R|+1$ different feasible residues of the $B$-problem (or fewer, if there are less than that many). We elaborate on why this is enough in what follows.

Let $\pi\colon \Pi_{\text{narrowed}}\to 2^{\{0, \ldots, m-1\}}$ be the function assigning to any $(\alpha, \beta)\in \Pi_{\text{narrowed}}$ the set $\pi(\alpha,\beta)\subseteq \{0,\ldots, m-1\}$ of residues $r_B\in\{0, 1, \ldots, m-1\}$ for which the $B$-problem is feasible.
We call $\pi$ a \emph{narrowed pattern} associated to the problem given in~\eqref{eq:structured-problem}.
Note that this pattern depends on the $3$-sum decomposition and the choice of $\ell_i$ and $u_i$ in \cref{lem:patternsBounds}, and hence may not be unique.
Also, we remark that a narrowed pattern can be seen as a restriction (to the narrowed domain $\Pi_{\text{narrowed}}$) of a \emph{global} pattern that maps any $(\alpha,\beta)\in\Pi$ to the corresponding set of feasible residues of the $B$-problem.

We can easily obtain a feasible solution for~\eqref{eq:structured-problem} if, among the solutions of the $B$-problem that we compute, we find a solution $x_B$ that fulfills $\gamma_A^{\top} x_A + \gamma_B^{\top} x_B \in R \pmod*{m}$, where $x_A$ is the computed solution to the relaxation of the $A$-problem.
Indeed, in this case, the concatenation of the two solutions $x_A$ and $x_B$ is feasible for the relaxation of \eqref{eq:structured-problem}.
In particular, if $|\pi(\alpha, \beta)| \geq m - |R| + 1$ for some $(\alpha, \beta)\in\Pi_{\text{narrowed}}$, we are guaranteed that there is such a feasible combination.
As explained above, through recursive calls to our procedure on the $B$-problem, we can decide whether we are in this case, and if so also compute $m - |R| + 1$ different feasible residues (and corresponding solutions).
Concretely, if we start from a problem with $|R| = m-2$, whenever we find a pair $(\alpha,\beta)$ of scalar products in $\Pi_{\text{narrowed}}$ with $|\pi(\alpha, \beta)|\geq 3$, we can find a feasible solution.
If $|\pi(\alpha, \beta)|\leq 2$ for all $(\alpha, \beta)\in\Pi_{\text{narrowed}}$, we study the pattern $\pi$ more closely.

One interesting special case is when $|\pi(\alpha, \beta)| = 1$ for all $(\alpha, \beta)\in\Pi_{\text{narrowed}}$, i.e., each of the $B$-problems is feasible for precisely one residue $r_B$. It turns out that in this case, $\pi$ is $\emph{linear}$ in the following sense.

\begin{definition}\label{def:linearity}
Let $\Pi\subseteq \mathbb{Z}^2$, and let $\pi\colon \Pi\to 2^{\{0,\ldots,m-1\}}$ for some $m\in\mathbb{Z}_{>0}$.
We say that $\pi$ is \emph{linear} if $|\pi(\alpha,\beta)|=1$ for all $(\alpha,\beta)\in\Pi$, and there exist $r_0,r_1,r_2\in\mathbb{Z}$ such that the mapping $r\colon \Pi\to\mathbb{Z}$ fulfilling $\pi(\alpha,\beta)=\{r(\alpha,\beta)\}$ satisfies $r(\alpha,\beta)\equiv r_0+r_1\alpha+r_2\beta\pmod*{m}$ for all $(\alpha,\beta)\in\Pi$.
\end{definition}

Linearity of $\pi$ and the shape of the domain $\Pi_{\text{narrowed}}$ makes it possible to encode the feasibility structure of the $B$-problem in only two variables $y_1$ and $y_2$ that represent the scalar products $\alpha$ and $\beta$, which allows for replacing $x_B$ with those new variables.

\begin{theorem}\label{thm:integration}
Consider an \RCCTUF{} problem of the form given in~\eqref{eq:structured-problem} and let $\pi$ an associated narrowed pattern. If $\pi$ is linear, then~\eqref{eq:structured-problem} can be reduced to the \RCCTUF{} problem
\begin{equation}\label{eq:integration}
\begin{array}{rcrcrcrcl}
       &      &             A x_A & + &   e y_1 &   &         & \leq   & b_A \\
       &      &        h^\top x_A &   &         & - &     y_2 &   =    & 0 \\
\ell_0 & \leq &                   &   &     y_1 & + &     y_2 & \leq   & u_0 \\
\ell_1 & \leq &                   &   &     y_1 &   &         & \leq   & u_1 \\
\ell_2 & \leq &                   &   &         &   &     y_2 & \leq   & u_2 \\
       &      & \gamma_A^\top x_A & + & r_1 y_1 & + & r_2 y_2 & \in    & r_0 + R \pmod{m}\\
       &      &               x_A &   &         &   &         & \in    & \mathbb{Z}^{n_A}\\
       &      &                   &   &     y_1 & , &     y_2 & \in    & \mathbb{Z}\\
\end{array}
\end{equation}
for suitable $\ell_0, u_0, \ell_1, u_1, \ell_2, u_2\in\mathbb{Z}$ with $u_i-\ell_i \leq m-|R|$ and $r_0, r_1, r_2\in \{0,1,\ldots, m-1\}$ that can be determined in strongly polynomial time. That is, the initial \RCCTUF{} problem is feasible if and only if~\eqref{eq:integration} is, and a solution of one problem can be transformed into one for the other in strongly polynomial time.
\end{theorem}

Hence, when $\pi$ is linear, we aim at applying \cref{thm:integration} and continuing our procedure with the \RCCTUF{} problem~\eqref{eq:integration}. To make progress, we aim at obtaining a smaller problem, which, as before, we measure in terms of the number of variables.
Note that the number of variables of~\eqref{eq:integration} is the number of columns of $A$ plus $2$, which is the same as the number of columns of the original problem plus $2$ minus the number of columns of $B$.
However, recall that by \cref{thm:TUdecomp}, we are only guaranteed that the matrix $B$ has at least two columns---which, in the extreme case, is not enough to reduce the number of columns through \cref{thm:integration}. Nevertheless, the equality constraint in~\eqref{eq:integration} allows for eliminating a variable while keeping the TU structure of the constraint matrix, thus guaranteeing that we can make progress. The following theorem formalizes this result.

\begin{theorem}\label{thm:projection}
Let $\begin{psmallmatrix}
A & a_1 \\ a_2^\top & \alpha
\end{psmallmatrix}$ be a TU matrix with $\alpha\neq 0$. Then, the matrix $A-\alpha a_1a_2^\top$ is TU, and the two systems
$
\quad\left\{\begin{array}{rcrcl}
Ax & + & a_1y & \leq & b \\
a_2^\top x & + & \alpha y & = & \beta
\end{array}\right.$\quad and
$\quad\left\{\begin{array}{rcl}
\left(A-\alpha a_1a_2^\top\right) x & \leq & b-\alpha\beta a_1 \\
y & = & \alpha\beta - \alpha a_2^\top x
\end{array}\right.
$\quad are equivalent.
\end{theorem}

Combining \cref{thm:integration,thm:projection}, we can thus make progress in case of a linear narrowed pattern $\pi$.
For non-linear narrowed patterns, like the one exemplified in \cref{fig:pattern}, there are pairs $(\alpha,\beta)$ for which there is more than one residue available, i.e., $|\pi(\alpha,\beta)|$, which is an additional flexibility we can exploit as follows.

\begin{figure}
\centering
\newcommand{\figscale}{0.9}
\begin{tikzpicture}[scale=\figscale]
\begin{scope}[every node/.style={draw, thick, rectangle, minimum width=\figscale cm, minimum height=\figscale cm, inner sep=0pt, text width=\figscale cm, align=center}]
\foreach \x/\y/\fill/\text in {-1/ 1/none/{\{0,1\}},  0/ 1/none/{\{1\}},   -1/ 0/none/{\{0\}},    0/ 0/none/{\{0,1\}}, 0/-1/none/{\{0\}}  }
\node[fill=\fill, font=\scriptsize] (\x\y) at (\x, \y) {$\text$};
\end{scope}

\begin{scope}[every path/.style={gray}]
\draw[->] (-2.25, -2) to node[pos=1, anchor=north west, inner sep=1pt] {$\alpha$} (1, -2);
\draw[->] (-2, -2.25) to node[pos=1, anchor=south east, inner sep=1pt] {$\beta$} (-2, 2);
\foreach \x in {-1, ..., 0}
\draw (\x, -2.1) to node[pos=0, anchor=north, font=\footnotesize] {$\x$} (\x, -1.9);
\foreach \y in {-1, ..., 1}
\draw (-2.1, \y) to node[pos=0, anchor=east, font=\footnotesize] {$\y$} (-1.9, \y);
\end{scope}

\clip (-3, -2.5) rectangle (2, 2.5);

\end{tikzpicture} \caption{A non-linear pattern $\pi$ with support defined by $-1\leq\alpha\leq 0$, $-1\leq\beta\leq1$, and $-1\leq \alpha+\beta\leq 1$.}
\label{fig:pattern}
\end{figure}
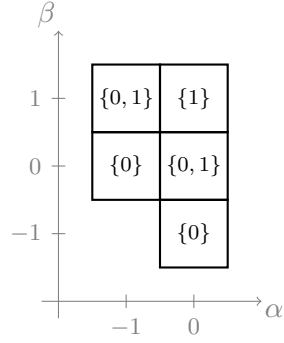

Concretely, consider a pair $(\alpha, \beta)\in\Pi_{\text{narrowed}}$ of scalar products with $\pi(\alpha, \beta) = \{r_B^1, \ldots, r_B^\ell\}$ for some $\ell>1$ and pairwise different $r_B^i\in\{0,\ldots,m-1\}$, and let $x_B^{1},\ldots, x_B^{\ell}$ be solutions of the relaxation of the $B$-problem with residues $\gamma_B^\top x_B^i=r_B^i$.
Observe that we can combine any feasible solution $x_A$ of the corresponding $A$-problem with any of the solutions $x_B^i$ to obtain feasible solutions $(x_A, x_B^i)$ of the relaxation of the initial \RCCTUF{} problem.
Thus, there is a solution with scalar products $(\alpha, \beta)$ if and only if the following variation of the $A$-problem is feasible, where $R'\coloneqq R-\pi(\alpha, \beta)=\{(r - r_B \bmod{m})\colon r\in R, r_B\in \pi(\alpha, \beta)\}$:
\begin{equation}\label{eq:AProbRedMoreThanOneRes}
\begin{array}{rcl}
A x_A             & \leq & b_A - \alpha e            \\
h^\top x_A        & =    & \beta                     \\
\gamma_A^\top x_A & \in  & R'   \pmod{m}             \\
x_A               & \in  & \mathbb{Z}^{n_A}\enspace.
\end{array}
\end{equation}
We will create a subproblem of the above form for each pair $(\alpha,\beta)\in \Pi_{\text{narrowed}}$ with $|\pi(\alpha,\beta)|\geq 2$, and recurse on these problems.
Doing so for all such scalar product pairs $(\alpha, \beta)\in\Pi_{\text{narrowed}}$, we create $O(m^2)$ many \RCCTUF{} problems to recurse on, each having at most $n-2$ many variables.
A key observation that allows for bounding the number of times we construct a problem of type~\eqref{eq:AProbRedMoreThanOneRes} and recurse on it is that problem~\eqref{eq:AProbRedMoreThanOneRes} is simpler than the problem we started with, because the set of target residues $R'$ strictly increased in size compared to $R$, whenever $m$ is a prime number.
This is a consequence of the Cauchy-Davenport Inequality stated below.
\begin{lemma}[Cauchy-Davenport Inequality] \label{lem:cauchyDavenport}
Let $m$ be a prime number and let $R_1, R_2\subseteq \{0, \ldots, m-1\}$. Then
$$
|\{(r_1 + r_2 \bmod{m})\colon r_1\in R_1, r_2\in R_2\}| \geq \min\{m, |R_1| + |R_2| - 1\}\enspace.
$$
\end{lemma}

Consequently, after at most $m-|R|$ reduction steps, the target residues comprise all possible residues and the corresponding problem gets trivial.
Therefore, the total number of subproblems that are spawned can be bounded by $O(m^{2(m-|R|)})$.
It thus remains to discuss scalar products $(\alpha, \beta)\in\Pi_{\text{narrowed}}$ with $|\pi(\alpha, \beta)|=1$ that are not covered by the previous arguments.
\cref{fig:pattern} shows an example of a narrowed pattern that contains three scalar product pairs $(\alpha, \beta)$ with $|\pi(\alpha,\beta)|=1$ together with two pairs with $|\pi(\alpha,\beta)|=2$.
Again, explicitly solving the corresponding $A$-problems is not an option because we lack the necessary progress either in terms of the number of variables or the number of target residues.

Also, it is not possible to apply \cref{thm:integration} only to the pairs $(\alpha,\beta)\in\Pi_{\text{narrowed}}$ with $|\pi(\alpha,\beta)|=1$, because \cref{thm:integration} crucially relies on the shape of the full domain of $\pi$, which can be described by inequalities of the form $\ell_0 \leq \alpha+\beta \leq u_0$, $\ell_1 \leq \alpha \leq u_1$, and $ \ell_2\leq \beta \leq u_2$.
Therefore, we focus in this case on identifying a well-chosen linear \emph{sub-pattern} $\widetilde{\pi}$ of $\pi$, i.e., a mapping $\widetilde{\pi}$ with the properties that
\begin{enumerate*}
\item its domain is a subset of the domain of $\pi$ and can be described by inequalities of the above type,
\item $\widetilde{\pi}(\alpha, \beta)=\{r_{\alpha,\beta}\}$ for some $r_{\alpha,\beta}\in\pi(\alpha, \beta)$, and
\item $\tilde{\pi}$ is linear according to \cref{def:linearity}.
\end{enumerate*}
Loosely speaking, a sub-pattern covers some of the available residues in the $B$-problem, and it is structured enough so that we can apply a variation of \cref{thm:integration} to cover these options through a smaller problem.
If $|R|\geq m-2$, it turns out that one such sub-pattern is enough in the following sense.

\begin{lemma}\label{lem:coveringPattern}
Let $\pi\colon\Pi_{\text{narrowed}}\to 2^{\{0,\ldots,m-1\}}$ be a narrowed pattern associated to a feasible \RCCTUF{} problem of the form in~\eqref{eq:structured-problem} with prime modulus $m$ and $|R|\geq m-2$.
Then, we can in strongly polynomial time determine a linear sub-pattern $\widetilde\pi$ of $\pi$ such that one of the following holds:
\begin{enumerate}
\item\label{lemitem:goodCombo} There are $(\alpha, \beta)\in\Pi_{\text{narrowed}}$ with $|\pi(\alpha,\beta)|=1$ so that for any $x_A$ solving the relaxation of the $A$-problem with respect to $(\alpha, \beta)$, there is an $x_B$ solving the relaxation of the $B$-problem such that the combination $(x_A,x_B)$ is feasible for the \RCCTUF{} problem.

\item\label{lemitem:type1}
There is a feasible solution for some pair $(\alpha,\beta)\in \Pi_{\text{narrowed}}$ with $|\pi(\alpha,\beta)|\geq 2$.

\item\label{lemitem:mainBranch}
There is a feasible solution $(x_A,x_B)$ for some pair $(\alpha,\beta)\in\operatorname{dom}(\widetilde{\pi})$ such that $\widetilde\pi(\alpha,\beta)=\{\gamma_B^\top x_B\}$.
\end{enumerate}
\end{lemma}

Thus, to check feasibility for an \RCCTUF{} problem of the form~\eqref{eq:structured-problem}, we can first compute, for each pair $(\alpha,\beta)\in \Pi_{\text{narrowed}}$, a solution $x_A$ to the relaxation of the $A$-problem with respect to scalar products $(\alpha,\beta)$ and check whether there is a solution $x_B$ to the $B$-problem that, combined with $x_A$, leads to a feasible solution to the initial problem. If this is the case, we are done. Otherwise, we know that~\ref{lemitem:goodCombo} of \cref{lem:coveringPattern} does not hold, and therefore either~\ref{lemitem:type1} or~\ref{lemitem:mainBranch} must hold.
Moreover, as previously explained, we call our procedure recursively for pairs $(\alpha,\beta)\in \Pi_{\text{narrowed}}$ with $|\pi(\alpha,\beta)|\geq 2$, spawning independent and simpler (because we increase the size of the allowed target residues $R$) subproblems for the $A$-problem.  Hence, if~\ref{lemitem:type1} of \cref{lem:coveringPattern} applies, then one of these simpler subproblems will lead to a feasible solution to the original problem.
Finally, we apply (a slight extension of) \cref{thm:integration} using the linear sub-pattern $\widetilde{\pi}$ and \cref{thm:projection}, thereby reducing to problems with fewer variables.
By \cref{lem:coveringPattern}, we know that if there is a feasible solution, we will find one in the described procedure.
Altogether, we get to the following theorem.

\begin{theorem}\label{thm:decompProgress}
Consider an \RCCTUF{} problem with prime modulus $m$, $n$ variables, $\ell\in\{m-1, m-2\}$ many target residues, and a constraint matrix $T$ falling into case~\ref{thmitem:TUdecomp_sum} of \cref{thm:TUdecomp}. Let $p=\min\{n_A, n_B\}$ be the number of columns of the smaller matrix in the decomposition.
After solving less than $3(m-\ell+1)^2$ many \RCCTUF{} problems with $p$ variables, modulus $m$, and at most $\ell$ target residues, we can in strongly polynomial time determine either a solution of the problem, or a family $\mathcal{F}$ of at most
\begin{itemize}
\item one \RCCTUF{} problem with at most $n-p+1$ variables, modulus $m$, and $\ell$ target residues, and
\item $(m-\ell+1)^2$ \RCCTUF{} problems with $n-p$ variables, modulus $m$, and at least $\ell+1$ target residues
\end{itemize}
such that the initial \RCCTUF{} problem is feasible if and only if at least one problem in $\mathcal{F}$ is feasible. Also, a feasible solution to any problem in $\mathcal{F}$ can in strongly polynomial time be transformed to one of the initial problem.
\end{theorem}

Finally, it remains to cover the case where the constraint matrix $T$ falls into case~\ref{thmitem:TUdecomp_pivot} of \cref{thm:TUdecomp}, i.e., only after pivoting, a decomposition step is possible.
It turns out that such \RCCTUF{} problems can be rewritten as a problem of the same type with the pivoted constraint matrix and one extra constraint that is a variable bound, thus subsequently allowing for a decomposition step without interfering with the progress that was made before (the number of variables and the number of target residues stay the same in the described transformation). The following theorem formalizes this.

\begin{theorem}\label{thm:pivoting}
Consider an \RCCTUF{} problem with constraint matrix $T$ for which case~\ref{thmitem:TUdecomp_pivot} of \cref{thm:TUdecomp} applies, i.e., $\pivot[ij](T)$ admits a decomposition as a $3$-sum according to \cref{thm:TUdecomp}. Then we can in strongly polynomial time determine an \RCCTUF{} problem of the form
\begin{equation}\label{eq:pivotedProblem}
\overline{T} y \leq \overline{b},\enspace y_j \leq \delta,\enspace \overline{\gamma}^\top y\in R \pmod*{m},\enspace y\in\mathbb{Z}^n\enspace,
\end{equation}
where $\overline{T}$ is, up to changing the sign in some rows and columns, the matrix $\pivot[ij](T)$, and solutions of the initial problem can be transformed to solutions of~\eqref{eq:pivotedProblem} in strongly polynomial time, and vice versa.
\end{theorem}

Leveraging \cref{thm:TUdecomp,thm:baseBlocks,thm:decompProgress,thm:pivoting}, we can conclude our main result, \cref{thm:R=m-2}.
\begin{proof}[Proof of \cref{thm:R=m-2}]
Consider an \RCCTUF{} problem with modulus $m$ and $\ell\geq m-2$ target residues.
If $\ell = m$, a solution can be found in strongly polynomial time by solving the relaxation of the problem using the framework of Tardos~\cite{tardos_1986_strongly}.
Else, we apply \cref{thm:TUdecomp} to the constraint matrix $T$.
If case~\ref{thmitem:TUdecomp_netw} or~\ref{thmitem:TUdecomp_const} of \cref{thm:TUdecomp} applies, \cref{thm:baseBlocks} guarantees that we can efficiently solve the corresponding problem.
If case~\ref{thmitem:TUdecomp_pivot} applies, we can reduce the problem to one where case~\ref{thmitem:TUdecomp_sum} applies through \cref{thm:pivoting}.
Finally, if case~\ref{thmitem:TUdecomp_sum} of \cref{thm:TUdecomp} applies, we apply \cref{thm:decompProgress} to reduce the problem to several smaller problems on which we recursively call our procedure.
Through these recursive calls, the initial \RCCTUF{} problem is reduced to several simpler \RCCTUF{} problems, where each of them has either $m$ many target residues or its constraint matrix is a base block matrix.

We first bound the number of such simpler \RCCTUF{} problems that we obtain.
Let $f(n,\ell)$ be the smallest upper bound on the number of such problems that we obtain through our reductions when starting from an instance with $n$ variables and $\ell$ target residues. We claim that
\begin{equation*}
f(n,\ell)\leq m^{2(m-\ell)}\cdot n^{m-\ell+3\log_2 m+2}\enspace.
\end{equation*}
Indeed, note that $f(n, \ell)=1$ for $n\leq 3$ and all $\ell\leq m$, and $f(n, m)=1$ for all $n$, and assume that the inequality holds for all instances of up to $n-1$ variables. By \cref{thm:decompProgress} and this assumption, we get, for some $p\in\{2,\ldots,\lfloor\sfrac n2\rfloor\}$, the desired
\begin{align*}
f(n, \ell) &\leq 3m^2 f(p, \ell) + f(n-p+1, \ell) + m^2f(n-p, \ell+1)\\
&\leq m^{2(m-\ell)}n^{m-\ell+3\log_2 m+2}\Bigg(\underbrace{\left(\frac pn\right)^{2} + \left(\frac{n-p+1}{n}\right)^{2} + \frac{n-p}{n^2}}_{\leq 1}\Bigg) \leq  m^{2(m-\ell)}n^{m-\ell+3\log_2 m+2}\enspace.
\end{align*}

Now observe that each of the at most $f(n,\ell)$ many subproblems can either be solved directly in strongly polynomial time as stated earlier (if it is a problem with $m$ target residues), or we can apply the strongly polynomial randomized algorithm provided by \cref{thm:baseBlocks} to each of them $\log_n(nf(n,\ell))=O(1)$ many times to correctly solve each problem with error probability at most $\sfrac{1}{nf(n,\ell)}$. Thus, by a union bound, we can solve all these problems (and thus the initial problem) correctly with probability $1-\sfrac1n$. To finish the proof, it remains to observe that the time for solving the discussed problems clearly dominates the time needed for transformations and solution propagation.
\end{proof}

\section{Proof and further implications of the decomposition lemma\MOORdot}\label{sec:decompLemma}

For the sake of presentation, we postpone the proof of the decomposition lemma (\cref{lem:decompositionLemma}) and \cref{lem:transformSolutionEfficiently} to the end of this section and start by showing additional implications, namely \cref{thm:R=m-1} and \cref{lem:patternsBounds}.

\subsection[An alternative approach to \texorpdfstring{$R$}{R}-CCTUF problems with \texorpdfstring{$|R|=m-1$}{|R|=m-1}: Proving \texorpdfstring{\cref{thm:R=m-1}}{Theorem~\ref{thm:R=m-1}}]{\boldmath An alternative approach to \RCCTUF{} problems with $|R|=m-1$: Proving \cref{thm:R=m-1}\MOORdot\unboldmath}

In this section, we prove that \RCCTUF{} problems with $|R|=m-1$ can be solved deterministically and in strongly polynomial time, as stated by \cref{thm:R=m-1}.
This result is closely linked to our flatness statement, \cref{thm:flatDirections}, which already guarantees that if none of the constraint matrix rows of the \RCCTUF{} problem is a flat direction of the underlying polyhedron with width $m-|R|-1$, then the problem can be solved efficiently.
For $|R|=m-1$, the width in this statement is $0$, i.e., the corresponding constraint is a tight constraint for the full underlying polyhedron.
Using \cref{thm:projection}, we can see that in this case of non-full-dimensional underlying polyhedra, we can easily project to a lower-dimensional space.

\begin{lemma}\label{lem:dimensionReduction}
Consider an \RCCTUF{} problem in $n\geq 2$ variables with a constraint that is tight for all points in the underlying polyhedron. We can in strongly polynomial time determine an \RCCTUF{} problem in $n-1$ variables such that solutions of the first problem can be transformed to solutions of the second problem in strongly polynomial time, and vice versa.
\end{lemma}

\begin{proof}
After permuting variables and constraints, we may assume that the inequality system in the given \RCCTUF{} problem has the form
\begin{equation*}
\begin{pmatrix}
\overline T & a_1 \\ a_2^\top & \alpha
\end{pmatrix} \begin{pmatrix}
\overline x \\ x_n
\end{pmatrix}
\leq
\begin{pmatrix}
\overline b \\ b_n
\end{pmatrix},
\quad\text{where}\quad
T = \begin{pmatrix}
\overline T & a_1 \\ a_2^\top & \alpha
\end{pmatrix},\quad
x = \begin{pmatrix}
\overline x \\ x_n
\end{pmatrix},\quad\text{and}\quad
b = \begin{pmatrix}
\overline b \\ b_n
\end{pmatrix},
\end{equation*}
such that $a_2^\top \overline x + \alpha x_n = b_n$ is a constraint that is tight for any solution to the relaxation of the \RCCTUF{} problem and $\alpha\neq 0$.
By \cref{thm:projection}, $(\overline y, y_n)$ is a solution of the above system if and only if $\overline y$ solves the TU system $\left(\overline T-\alpha a_1a_2^\top\right)\overline x \leq \overline b$, and $y_n=\alpha b_n- \alpha a_2^\top \overline y$.
Therefore, the original \RCCTUF{} problem can be reduced in strongly polynomial time to the following \RCCTUF{} problem with only $n-1$ variables:
\begin{equation*}
 \overline{T} \overline x \leq \overline b,\enspace
 (\overline \gamma-\alpha\gamma_na_2)^\top \overline x \not\equiv r-\alpha\gamma_nb_n\pmod*{m},\enspace
\overline x\in\mathbb{Z}^{n-1}\enspace. \qedhere
\end{equation*}
\end{proof}

Although not exploited here, we remark that the above reduction of non-full-dimensional problems also applies to the optimization version of the considered problem. Now, combining \cref{lem:dimensionReduction,thm:flatDirections}, we immediately obtain a proof of \cref{thm:R=m-1}.

\begin{proof}[Proof of \cref{thm:R=m-1}]
We start by observing that using a result of Tardos~\cite{tardos_1986_strongly}, we can solve linear programs over the underlying polyhedron of a given \RCCTUF{} problem in strongly polynomial time, and hence, we can also detect in strongly polynomial time whether there is a tight constraint.
If there is no tight constraint, then the problem can be solved by \cref{thm:flatDirections}.
Otherwise, the problem is trivial when $n=1$, and if $n\geq 2$, we can repeatedly apply \cref{lem:dimensionReduction} until we obtain a problem with $n=1$, or one that does not have tight constraints.
Note that the number of variables reduces by $1$ in each application of \cref{lem:dimensionReduction}, hence there are less than $n$ iterations.
We conclude the proof by observing that solutions of a problem without tight constraints that stem from \cref{lem:dimensionReduction} can be transformed back to solutions of the initial problem in strongly polynomial time by the very same lemma.
\end{proof}

\subsection{Bounded scalar products\MOORdot}\label{sec:boundedProducts}

The goal of this subsection is to deduce \cref{lem:patternsBounds}, which we use to restrict the search space for solutions of \RCCTUF{} problems.
It turns out that this lemma is an implication of a more general result that we expand on below.

\begin{lemma}\label{lem:productInterval}
Consider a feasible \RCCTUF{} problem with constraint matrix $T$ and modulus $m$, and let $d$ be TU-appendable to $T$.
We can determine in strongly polynomial time $\ell, u\in\mathbb{Z}$ with $u-\ell \leq m-|R|$ such that the \RCCTUF{} problem has a feasible solution $x_0$ if and only if it has one with $\ell \leq d^\top x_0\leq u$.
\end{lemma}

\begin{proof}
Let $Tx\leq b$ be the inequality system in the \RCCTUF{} problem. To start with, we can in strongly polynomial time determine $\eta_{\max}\coloneqq \max\{d^\top x\colon Tx\leq b,\ x\in\mathbb{R}^n \}$ and $\eta_{\min}\coloneqq \min\{d^\top x\colon Tx\leq b,\ x\in\mathbb{R}^n\}$.
If $\eta_{\max}-\eta_{\min}\leq m-|R|$, we can choose $u=\eta_{\max}$ and $\ell=\eta_{\min}$, and there is nothing to show.
Otherwise, we claim that the statement holds for any choice of $\ell,u\in\{\eta_{\min},\ldots,\eta_{\max}\}$ with $u-\ell \leq m-|R|$.
To see this, consider any such choice of $\ell$ and $u$ and consider the given \RCCTUF{} problem with the constraints $\ell\leq d^\top x \leq u$ added to the inequality system.
Because by construction, $d$ is a flat direction of width exactly $m-|R|$ for that problem, applying twice \cref{lem:flatOrIrrelevant} (once for each of the two constraints that we added) gives that the problem with the constraints added is feasible if and only if the original one is feasible.
\end{proof}

Note that if we are given vectors $d_1,\ldots,d_p$ that are all simultaneously TU-appendable to the constraint matrix of the problem, we can apply \cref{lem:productInterval} iteratively with the TU-appendable vectors $d_i$, adding the obtained constraints $\ell_i\leq d_i^\top x\leq u_i$ to the system in each step. This immediately implies the following corollary.

\begin{corollary}\label{cor:simultaneousBound}
Consider a feasible \RCCTUF{} problem with constraint matrix $T$ and modulus $m$, and let $d_1,\ldots,d_p$ be simultaneously TU-appendable to $T$.
We can determine in strongly polynomial time $\ell_i, u_i\in\mathbb{Z}$ with $u_i-\ell_i \leq m-|R|$ for $i\in[p]$ such that the \RCCTUF{} problem has a feasible solution $x_0$ if and only if it has one with $\ell_i \leq d^\top x_0\leq u_i$ for all $i\in[p]$.
\end{corollary}

Now \cref{lem:patternsBounds} follows immediately from \cref{cor:simultaneousBound} after observing the following.

\begin{observation}\label{obs:simultTUappend}
Consider a matrix $T$ that is a $3$-sum of the form $T=\begin{psmallmatrix}
A & ef^\top \\ gh^\top & B
\end{psmallmatrix}$.
Then, the rows $\begin{pmatrix} 0 & f^\top \end{pmatrix}$, $\begin{pmatrix} h^\top & 0 \end{pmatrix}$, and $\begin{pmatrix} h^\top & f^\top\end{pmatrix}$ are simultaneously TU-appendable to $T$.
\end{observation}

\begin{proof}
Observe that
\begin{equation}\label{eq:extendedSum}
\begin{pmatrix}
A & ef^\top \\ 0 & f^\top \\ h^\top & f^\top \\ h^\top & 0 \\ gh^\top & B
\end{pmatrix}  = \begin{pmatrix}
A & e & e \\ 0 & 1 & 1 \\ h^\top & 0 & 1
\end{pmatrix} \ksum[3] \begin{pmatrix}
0 & 1 & f^\top \\ 1 & 1 & f^\top \\ 1 & 1 & 0 \\ g & g & B
\end{pmatrix}.
\end{equation}
Recall that because the totally unimodular matrix $T$ decomposes into a $3$-sum of the two matrices
$\begin{psmallmatrix}
A & e & e \\ h^\top & 0 & 1
\end{psmallmatrix}$ and $\begin{psmallmatrix}
0 & 1 & f^\top \\ g & g & B
\end{psmallmatrix}$, we know that these matrices are totally unimodular, as well.
It can be easily seen that this implies total unimodularity of the two summands in~\eqref{eq:extendedSum}, and hence also of the $3$-sum of the two matrices.
\end{proof}

\begin{proof}[Proof of \cref{lem:patternsBounds}]
By \cref{cor:simultaneousBound} above, it is enough to show that the vectors $\begin{pmatrix} 0 & f^\top \end{pmatrix}$, $\begin{pmatrix} h^\top & 0 \end{pmatrix}$, and $\begin{pmatrix} h^\top & f^\top\end{pmatrix}$ are simultaneously TU-appendable to $T$.
The latter is true, as seen in \cref{obs:simultTUappend} above.
\end{proof}

Finally, we note that the assumption of simultaneous TU-appendability in \cref{cor:simultaneousBound} is necessary to obtain ranges of width $m-|R|$ for each scalar product.
More generally, if we want to obtain bounds simultaneously for all TU-appendable vectors, our general proximity result, \cref{thm:proximityGeneral}, only implies ranges of width $2(m-|R|)+1$.

\subsection[Proof of the decomposition lemma (\texorpdfstring{\cref{lem:decompositionLemma}}{Lemma~\ref{lem:decompositionLemma}}) and \texorpdfstring{\cref{lem:transformSolutionEfficiently}}{Lemma~\ref{lem:transformSolutionEfficiently}}]{Proof of the decomposition lemma (\cref{lem:decompositionLemma}) and \cref{lem:transformSolutionEfficiently}\MOORdot}\label{sec:proofDecompLemma}

In order to prove \cref{lem:decompositionLemma} we first show a key property of pointed polyhedral cones defined by TU matrices (which we also call \emph{TU cones}), from which will later derive \cref{lem:decompositionLemma}.
To this end, we recall that, for a polyhedral cone $C\coloneqq \{x\in \mathbb{R}^n \colon Ax \leq 0\}$ (where $A\in \mathbb{Q}^{k\times n}$), an \emph{extremal ray} of $C$ is a non-zero vector $r\in C$ that lies on a $1$-dimensional face of $C$.
Moreover, we use the following notion of \emph{elementary extremal ray}.
\begin{definition}[Elementary extremal ray]
An extremal ray $r$ of a polyhedral cone $C\subseteq \mathbb{R}^n$ is $\emph{elementary}$ if $r\in \mathbb{Z}^n$ and the greatest common divisor of the coordinates of $r$ is one.
\end{definition}
Hence, for every rational cone $C$ and every extremal ray $r$ of the cone, there is some unique $\lambda > 0$ such that $\lambda r$ is an elementary extremal ray of $C$.

\cref{lem:decompPointedTUCone} below shows that any point in a pointed cone $C$ that is defined by a TU matrix can be integrally decomposed into few elementary extremal rays in strongly polynomial time. We highlight that the crucial part of \cref{lem:decompPointedTUCone} is that the coefficients $\lambda_i$ can be chosen to be integral. Note that, despite the cone being defined by a TU matrix, the elementary extremal rays in \cref{lem:decompPointedTUCone} have to be well-chosen because the set of elementary extremal rays of $C$ does not form a totally unimodular matrix.\footnote{Indeed, cones defined by TU matrices can have exponentially many elementary extremal rays. This follows for example by the well-known fact that the bipartite matching polytope $P$, which can be described by a TU matrix, has vertices $v\in \vertices(P)$ with exponentially many edges incident to them. Hence, the set of constraints of $P$ that are tight at $v$ define a TU cone (when shifted such that $v$ becomes the origin) with exponentially many elementary extremal rays.}
Hence, even if a set of $n$ elementary extremal rays of $C$ spans $y$, it may be that the decomposition of $y$ into a conic combination of these elementary extremal rays requires non-integral coefficients. (This is arguably the case to be expected without choosing the rays carefully.)
\begin{lemma}\label{lem:decompPointedTUCone}
Let $T\in \{-1,0,1\}^{k\times n}$ be a totally unimodular matrix such that the cone $C\coloneqq \{x\in \mathbb{R}^n \colon Tx \leq 0\}$ is pointed, and let $y\in C \cap \mathbb{Z}^n$. Then one can determine in strongly polynomial time elementary extremal rays $y^1,\ldots, y^n\in \mathbb{Z}^n$ of $C$ and coefficients $\lambda_1,\ldots, \lambda_n\in \mathbb{Z}_{\geq 0}$ such that $y=\sum_{i=1}^n \lambda_i y^i$.
\end{lemma}

\begin{proof}
We prove the statement by determining successively pairs $(\lambda_i, y^i)$ of the desired decomposition of $y$.
We start by explaining how we compute $\lambda_1$ and $y^1$, and then highlight how to iterate the procedure to obtain the full decomposition of $y$.
To obtain a first coefficient $\lambda_1$ and vector $y^1$ of the desired decomposition of $y$, we define an auxiliary polytope $P_1$ by
\begin{align*}
P_1 \coloneqq C\cap C_1\enspace,\quad\text{where}\quad
C_1 \coloneqq \left\{x\in \mathbb{R}^n \colon -Tx \leq -Ty\right\}\enspace.
\end{align*}
Hence,
\begin{align*}
P_1 \coloneqq \left\{x\in \mathbb{R}^n \colon \begin{pmatrix} T\\ -T \end{pmatrix}x \leq \begin{pmatrix} 0\\ -Ty\end{pmatrix}\right\}\enspace.
\end{align*}
Note that $C_1$ can be interpreted as a reversed version of $C$ with apex at $y$.
Also note that $P_1$ is a polytope because $C$ is pointed.
Indeed, if $P_1$ were unbounded, there would need to be a non-zero vector $r\in \mathbb{R}^n$ with $Tr\leq 0$ and $-T r\leq 0$, which implies $Tr=0$ and contradicts that $C$ is pointed.
Moreover, as highlighted above, observe that $P_1$ can be described by the constraint matrix $\begin{psmallmatrix} T \\ -T\end{psmallmatrix}$, which is TU.

Let $T^{\scriptscriptstyle =}$ be the set of constraints of $C$ that are tight at $y$. Hence, $T^{\scriptscriptstyle =} y = 0$. Similarly, let $T^{\scriptscriptstyle <}$ denote the remaining constraints of $C$, which are the ones not tight at $y$. Hence, $T^{\scriptscriptstyle <} y < 0$.
	In addition, without loss of generality, we may assume that the rows in $T^{\scriptscriptstyle <}$ are linearly independent of those of $T^{\scriptscriptstyle =}$; for otherwise they are redundant and we can drop them.
	Let $y^1$ be any extremal ray of
\begin{equation*}
Q_1 \coloneqq \left\{x\in \mathbb{R}^n \colon T^{\scriptscriptstyle{=}}x = 0, T^{\scriptscriptstyle{<}}x \leq 0 \right\}\enspace.
\end{equation*}
Note that $Q_1$ is pointed because $Q_1 \subseteq C$ and $C$ is pointed; thus, it has extremal rays.
Such an extremal ray $y^1$ can be computed efficiently via standard techniques.\footnote{Any vertex $u \in \mathbb{R}^n_{\geq 0}$ of the polytope $P'\coloneqq Q_1 \cap \{x\in \mathbb{R}^n \colon 1^\top x \leq 1\}$, with $u \neq 0$, induces an extremal ray of $Q_1$. Hence, it is enough to compute an optimal vertex solution of the linear program $\max \{1^\top x\colon  x \in P'\}$, which can be done in polynomial time via standard methods. Note that all numbers/coefficients involved in this linear program are small (actually they are all within $\{-1,0,1\}$). Hence, the running time is thus trivially strongly polynomial in the original input size.}
By rescaling $y^1$, we can assume without loss of generality that $y^1 \in \mathbb{Z}^n$ is an elementary extremal ray of $Q_1$. Let
\begin{equation*}
\lambda_1 \coloneqq\max\left\{ \lambda \in \mathbb{R}_{\geq 0} \colon -T^{\scriptscriptstyle <}(\lambda y^1) \leq -T^{\scriptscriptstyle <} y\right\}\enspace,
\end{equation*}
that is, $\lambda_1$ captures how far in the direction of the elementary extremal ray $y^1$ we can go, when starting from the origin, while staying within $P_1$.
The constraints of the above optimization problem are of the form $\lambda a_i \leq b_i$ for $i\in [\ell]$, with $a_i \coloneqq -(T^{\scriptscriptstyle <} y^1)_i$ and $b_i \coloneqq -(T^{\scriptscriptstyle <} y)_i$.
By definition of $T^{\scriptscriptstyle <}$, we have $T^{\scriptscriptstyle <} y < 0$, and thus $b_i > 0$ for all $i\in [\ell]$. Hence,
\begin{equation*}
\lambda_1 = \min\left\{\frac{b_i}{a_i}\colon i \in [\ell]\text{ with } a_i > 0\right\}\enspace,
\end{equation*}
which shows that $\lambda_1$ can be computed in strongly polynomial time by first computing $a_i$ and $b_i$ for $i\in [\ell]$ and then determining the minimizing ratio $\sfrac{b_i}{a_i}$.

Note that $\lambda_1 y^1$ must be a vertex of $P_1$. This follows because $\lambda_1 y^1 \in P_1$ by construction, and $y^1$ is an extremal ray of $Q_1$ (it thus lies on a face of $Q_1$ of dimension $1$), and therefore $y^1$ is also an extremal ray of $P_1$ because $Q_1$ is a face of $P_1$.\footnote{Here we use the basic polyhedral fact that a face of a face of a polyhedron is a face of the polyhedron.}
Hence, $\lambda_1 y^1$ is a face of $P_1$ of dimension $0$, i.e., $\lambda_1 y^1 \in \vertices(P_1)$.
Moreover, because $P_1$ is described by a TU system, its set of vertices must be all integral, and hence $\lambda_1 y^1 \in \mathbb{Z}^n$.
Furthermore, we must also have that $\lambda_1 \in \mathbb{Z}_{\geq 0}$.
If not, then we can write $\lambda_1 = \sfrac{p}{q}$ with $p,q \in \mathbb{Z}_{\geq 0}$ such that their greatest common divisor $\gcd(p,q)$ equals $1$ and $q \geq 2$.
As $\lambda_1 y^1 \in \mathbb{Z}^n$, we must have that $q$ divides $p y^1_i$ for all $i \in [n]$.
However, this implies that $q$ divides $y^1_i$ for all $i\in [n]$, which follows from $\gcd(p,q)=1$ and a well-known basic number theory result.\footnote{More precisely, we use that for any $a,b,c \in \mathbb{Z}$ with $\gcd(a,b)=1$, if $a$ divides $bc$ then $a$ divides $c$.}
But this contradicts with $y^1$ being elementary.

We now proceed inductively on the vector $y'\coloneqq y-\lambda_1 y^1$.
Note that by construction we have $T y'\leq 0$, and can thus reiterate the above-explained approach with the vector $y'$ instead of $y$.
Let $T_1^{\scriptscriptstyle{=}}$ be the rows of $T$ that correspond to constraints of $Tx \leq 0$ that are tight at $y'$; hence, $T_1^{\scriptscriptstyle{=}} y' = 0$.
Analogously as before, let $T_1^{\scriptscriptstyle{<}}$ be the other rows, which correspond to constraints of $Tx \leq 0$ that are not tight at $y'$.
As before, we then define
\begin{equation*}
Q_2 \coloneqq \left\{x\in \mathbb{R}^n \colon T_1^{\scriptscriptstyle{=}}x =0, T_1^{\scriptscriptstyle{<}}x \leq 0\right\}\enspace,
\end{equation*}
compute an elementary extremal ray of $Q_2$ and continue as above.
Note that $\dim(Q_2)< \dim(Q_1)$, because $y' \coloneqq y-\lambda_1 y^1$ was chosen such that a new constraint of $Tx \leq 0$ that was not tight at $y$ became tight at $y'$.
Hence, this procedure will terminate after at most $\dim(Q_1)\leq n$ many iterations.
If the procedure terminates in less than $n$ iterations, in which case we get a decomposition with fewer than $n$ terms, we can add arbitrary extremal rays with zero coefficients to the decomposition to obtain the claimed $n$ many terms.
\end{proof}

The following statement shows that elementary extremal rays of a TU cone are elementary with respect to the TU matrix defining the cone.
This property links the notions of elementary extremal ray and of being elementary with respect to a TU matrix.
\begin{lemma}\label{lem:extremalRayIsExtremalT}
Let $T\in \{-1,0,1\}^{k\times n}$ be a totally unimodular matrix and $r\in \mathbb{Z}^n$ be an elementary extremal ray of $C\coloneqq \{x\in \mathbb{R}^n \colon Tx \leq 0\}$. Then $r$ is elementary with respect to $T$.
\end{lemma}
\begin{proof}
With the goal of deriving a contradiction, assume that there is a vector $d\in \{-1,0,1\}^n$ that is TU-appendable to $T$ and such that $\eta\coloneqq d^{\top} r\not\in \{-1,0,1\}$.
Without loss of generality, we assume $\eta > 0$, which can be achieved by replacing $d$ by $-d$ if necessary.
We denote by
\begin{equation*}
L\coloneqq \left\{\lambda r \colon \lambda \geq 0\right\}
\end{equation*}
the $1$-dimensional face of $C$ on which $r$ lies.
Note that $(\sfrac{1}{\eta})\cdot r$ lies in the polyhedron $Z$ defined by
\begin{align*}
Z\coloneqq \left\{x\in \mathbb{R}^n \colon Tx\leq 0, d^\top x = 1\right\}\enspace.
\end{align*}
Hence, because $\begin{psmallmatrix}
T \\ d^\top
\end{psmallmatrix}$ is TU (recall that $d$ is TU-appendable to $T$), $(\sfrac{1}{\eta})\cdot r$ can be written as a convex combination of integer points in $Z$, say
\begin{equation}\label{eq:decomopOfScaledDownYi}
\frac{1}{\eta}\cdot r = \sum_{j=1}^{q} \mu_j z_j\enspace,
\end{equation}
with $\mu_j \geq 0, z_j\in Z\cap \mathbb{Z}^n$ for $j\in [q]$, and $\sum_{j=1}^q \mu_j =1$.
Observe that $(\sfrac{1}{\eta})\cdot r$ is the only point on $L$ that is also in $Z$ because $d^{\top} r \neq 0$, i.e.,
\begin{equation*}
L\cap Z = \left\{(\sfrac{1}{\eta})\cdot r\right\}\enspace.
\end{equation*}
As $(\sfrac{1}{\eta})\cdot r \not\in \mathbb{Z}^n$, because $r$ is elementary and $\eta>1$, we have
\begin{equation*}
z_j\not\in L \quad \forall j\in [q]\enspace.
\end{equation*}
However, this leads to a contradiction because it implies that the decomposition~\eqref{eq:decomopOfScaledDownYi} expresses a point on the $1$-dimensional face $L$ of $C$ as a convex combination of points in $C$, none of which lies on $L$.
This is impossible because any convex combination that describes a point on a $1$-face of a polyhedron needs to use terms on the same face.
\end{proof}

We are now ready to prove \cref{lem:decompositionLemma}.

\begin{proof}[Proof of \cref{lem:decompositionLemma}]
Because the statement is invariant under a shift of the coordinate system, we can assume $x_0=0$ for convenience.
(Formally, instead of considering $Tx\leq b$ and $x_0,y$, we consider the system $Tx \leq b-Tx_0$ and replace $x_0$ and $y$ by the origin and $y-x_0$, respectively.)
Moreover, we observe that we can assume that the system $Tx \leq b$ contains, for each $i\in [n]$, the constraint
\begin{equation}\label{eq:makeTConePointed}
\begin{cases}
x_i \geq 0 &\text{if } y_i \geq 0\enspace,\\
x_i \leq 0 &\text{if } y_i <0\enspace.
\end{cases}
\end{equation}
Indeed, by adding these constraints, the thus obtained system $\widetilde{T}x \leq \widetilde{b}$ is still a TU system for which both the origin and $y$ are feasible.
Moreover, a decomposition of $y$ with respect to this new system $\widetilde{T} x \leq \widetilde{b}$ has the desired properties because a vector is TU-appendable to $T$ if and only if it is TU-appendable to $\widetilde{T}$, which implies that a vector is elementary w.r.t.~$T$ if and only if it is elementary w.r.t.~$\widetilde{T}$.\footnote{The fact that TU-appendability to $T$ is the same as TU-appendability to $\widetilde{T}$ is an immediate consequence of the fact that adding rows that are all-zero except for a single $1$ or $-1$ entry to any TU matrix preserves TU-ness.}
Hence, we assume from now on that $Tx \leq b$ contains the constraints~\eqref{eq:makeTConePointed}, which implies that $T$ has full column rank.

We now define a TU matrix $\overline{T}\in \{-1,0,1\}^{k\times n}$ which is obtained from $T$ by changing the sign of some of its rows.
More precisely for each row $w^{\top}$ of $T$, the matrix $\overline{T}$ contains a row
\begin{equation*}
\begin{cases}
\phantom{-}w^{\top} &\text{if } w^{\top} y \leq 0\enspace,\\
-w^{\top}           &\text{if } w^{\top} y > 0\enspace.
\end{cases}
\end{equation*}
We define
\begin{equation*}
C\coloneqq \left\{ x\in \mathbb{R}^n \colon \overline{T}x \leq 0\right\}\enspace.
\end{equation*}
Note that $C$ is pointed because $\overline{T}$ has full column rank, which follows from $T$ having full column rank.
We now apply \cref{lem:decompPointedTUCone} to the TU matrix $\overline{T}$ and point $y$.
This leads to a decomposition of $y$ as $y=\sum_{i=1}^n \lambda_i y^i$ such that, for $i\in [n]$, we have $\lambda_i \in \mathbb{Z}_{\geq 0}$ and $y^i$ is an elementary extremal ray of $C$.
We claim that this decomposition has the desired properties.

Note that by \cref{lem:extremalRayIsExtremalT}, each vector $y^i$ for $i\in [n]$ is elementary with respect to $\overline{T}$.
It is therefore also elementary with respect to $T$, because $\overline{T}$ and $T$ have the same set of TU-appendable rows as they are the same matrices up to sign changes of some of the rows.

It remains to show that for any coefficients $\mu_1,\ldots, \mu_n\in \mathbb{Z}_{\geq 0}$  with $\mu_i\leq \lambda_i$ for $i\in [n]$, we have that the vector
\begin{equation*}
\widetilde{y} \coloneqq x_0 + \sum_{i=1}^n \mu_i y^i = \sum_{i=1}^n \mu_i y^i
\end{equation*}
satisfies $T \widetilde{y} \leq b$.
To this end consider a constraint $w^{\top} x \leq \beta$ of the system $Tx \leq b$.
We distinguish between whether $w^{\top}$ or $-w^{\top}$ is a row of $\overline{T}$.
If $w^{\top}$ is a row of $\overline{T}$, then
\begin{equation*}
w^\top \widetilde{y} = \sum_{i=1}^n \mu_i w^{\top} y^i \leq 0 \leq \beta\enspace,
\end{equation*}
where the first inequality follows from $w^\top y^i \leq 0$ because $y^i$ is a ray of $C$, and the second inequality follows from the fact that the origin is feasible for the system $Tx \leq b$, which implies that all right-hand sides are non-negative.

Consider now the case where $-w^{\top}$ is a row of $\overline{T}$. Then we have
\begin{equation*}
w^{\top} \widetilde{y} = w^{\top} y - \sum_{i=1}^n \left(\lambda_i -\mu_i\right) w^{\top} y^i \leq w^{\top} y \leq \beta\enspace,
\end{equation*}
where the first inequality follows from $\lambda_i \geq \mu_i$ together with $w^{\top} y^i \geq 0$, which holds because $\overline{T} y^i \leq 0$ and $\overline{T}$ contains the row $-w^{\top}$, and the last inequality follows from $Ty \leq b$, which contains the constraint $w^{\top} y \leq \beta$.
Hence, $\widetilde{y}$ fulfills all constraints of the system $Tx\leq b$, as desired, which finishes the proof.
\end{proof}

\begin{proof}[Proof of \cref{lem:transformSolutionEfficiently}]
By applying \cref{lem:decompositionLemma} to the solutions $y$ and $x_0$ of the system $Tx\leq b$ of the given \RCCTUF{} problem, we obtain in strongly polynomial time $y^1,\ldots,y^n\in\mathbb{Z}^n$ and $\lambda_1,\ldots,\lambda_n\in\mathbb{Z}_{\geq 0}$ such that $y = x_0 + \sum_{i=1}^n \lambda_i y^i$ and
\begin{enumerate*}
\item $d^\top y^i\in\{-1,0,1\}$ for all $i\in[n]$ and all $d$ that are TU-appendable to $T$, and
\item $\tilde y = x_0 + \sum_{i=1}^n \mu_iy^i$ is feasible for $Tx\leq b$ for any choice of $\mu_i\in\{0,\ldots,\lambda_i\}$.
\end{enumerate*}
By these properties, in order to prove \cref{lem:transformSolutionEfficiently}, it is enough to identify in strongly polynomial time $\mu_i\in\{0,\ldots,\lambda_i\}$ with $\sum_{i=1}^n \mu_i\leq m-|R|$ such that $\gamma^\top \tilde y = \gamma^\top x_0 + \sum_{i=1}^n\mu_i \gamma^\top y^i\in R\pmod*{m}$. Denoting $\Lambda=\sum_{i=1}^n\lambda_i$ and
\begin{align}\label{eq:defineResidues}
R'=\{(r -\gamma^\top x_0 \bmod{m})\colon r\in R\}\enspace, \quad\text{as well as}\qquad
\begin{aligned}
r_1=\ldots=r_{\lambda_1}&=\gamma^\top y^1\enspace,\\ r_{\lambda_1+1}=\ldots=r_{\lambda_1+\lambda_2}&=\gamma^\top y^2\enspace,\\
 &\ \;\vdots \\
r_{\lambda_1+\ldots+\lambda_{n-1}+1} = \ldots = r_{\Lambda}&=\gamma^\top y^n\enspace,
\end{aligned}
\end{align}
we can formulate this problem as follows: We start from the sum $\sum_{i\in S_0} r_i\in R'\pmod*{m}$ with $S_0=[\Lambda]$, and our goal is to identify a subset $S\subseteq S_0$ of size at most $m-|R|=m-|R'|$ such that $\sum_{i\in S}r_i\in R'\pmod*{m}$, as well.
By \cref{lem:feasibleResidueSum}, we know that if $|S_0| > m-|R'|$, there exists an interval $I_1=\{i_1^1,\ldots, i_2^1\}$ with $i_1^1, i_2^1\in S_0$ and $i_1^1 < i_2^1$ such that for $S_1 = S_0\setminus I_1$, we have $\sum_{i\in S_1} r_i\in R'\pmod*{m}$.
Iterating this argument, we obtain that for $j=1,2,\ldots$ and while $|S_{j-1}|>m-|R'|$, there exists an interval $I_j=\{i_1^j,\ldots, i_2^j\}$ with $i_1^j, i_2^j\in S_0$ and $i_1^j < i_2^j$ such that $I_j \cap S_{j-1}\neq \emptyset$, and for $S_j = S_{j-1}\setminus I_j$, we have $\sum_{i\in S_j} r_j\in R'\pmod*{m}$.
For clarity, we remark that in step $j$, we are removing the terms with indices in $S_{j-1}\cap I_j$ from the sum.
Moreover, while these indices are consecutive in the sum that we consider in step $j$, they may not be so in the original sum $\sum_{i=1}^n r_i$, as indices in $I_j\setminus S_{j-1}$ correspond to terms that were removed in earlier steps.
For this reason, an index $i\in[\Lambda]$ may well be contained in several intervals $I_j$.

Because $I_j\cap S_{j-1}\neq\emptyset$, the number of terms in the sum strictly decreases in every step, so the procedure terminates, which shows existence of the desired solution $\tilde y$, as already pointed out in \cref{sec:overviewDecompAndFlat}.
To arrive at a suitably short sum in strongly polynomial time, we split the deletion process into two phases:
\begin{description}
\item[Phase 1:]
Steps $j$ such that $|S_{j-1}|>m-1$, i.e., the sum has more than $m-1$ terms. \\
Hence, the above arguments can be applied with $R'$ replaced by the singleton set $\{(\sum_{i\in S_0}r_i \bmod{m})\}$ such that the sums $\sum_{i\in S_j} r_i$ obtained in this phase all have the same residue.
Equivalently, terms that sum to $0\pmod*{m}$ are removed in every step, i.e., $\sum_{i\in  S_{j-1}\cap I_j}r_i\equiv 0\pmod*{m}$.
\item[Phase 2:]
Steps $j$ such that $|S_{j-1}|\leq m-1$, i.e., the sum has at most $m-1$ terms. \\
In this case, at most $|R|-1$ further deletion steps suffice to reduce to at most $m-|R|$ many terms.
\end{description}
A way to perform the steps in strongly polynomial time both in phase~1 and phase~2, as well as a strongly polynomial bound on the number of steps in phase~1 is provided by the following two claims:
\begin{enumerate}[label=(\alph*), ref=(\alph*)]
\item\label{claimitem:maxInterval} We can, in every step of the described procedure and in strongly polynomial time, determine an interval to delete of maximum possible size, i.e., determine $I_j$ such that $|S_{j-1}\cap I_j|$ is maximized.
\item\label{claimitem:2nsteps} If in every step, $I_j$ is chosen according to point~\ref{claimitem:maxInterval}, the procedure ends after at most $n$ steps.
\end{enumerate}
Together,~\ref{claimitem:maxInterval} and~\ref{claimitem:2nsteps} immediately prove \cref{lem:transformSolutionEfficiently}.
To proof the two claims, let us start with focusing on claim~\ref{claimitem:2nsteps}.
First, we observe that in phase~1, choosing $I_j$ to maximize $|S_{j-1}\cap I_j|$ implies that no two intervals will overlap, i.e., $I_j\cap I_k=\emptyset$ for all intervals $I_j$ and $I_k$ that we construct in this phase.
To see this, assume for the sake of deriving a contradiction that $I_\ell$ is an interval that overlaps with some earlier intervals $I_{j_1},\ldots,I_{j_t}$ with $j_1 < \ldots < j_t < \ell$, and choose the minimum $\ell$ with this property.
In particular, we thus know that the intervals $I_{j_1},\ldots,I_{j_t}$ do not overlap with each other and with any other intervals $I_j$ with $j<\ell$.
This implies that in step $j_1$, $I'\coloneqq I_\ell\cup I_{j_1}\cup\ldots\cup I_{j_t}$ is a candidate interval: Indeed, taking $I'$ would remove the terms
$$
\smashoperator[r]{\sum_{i\in S_{j_1-1}\cap I'}}\ r_i = \sum_{i\in I'} r_i = \sum_{p=1}^t \sum_{i\in I_{j_p}} r_i + \smashoperator[r]{\sum_{i\in I_\ell\setminus\bigcup_{p=1}^t I_{j_p}}}\ r_i = \sum_{p=1}^t \smashoperator[r]{\sum_{i\in S_{j_p-1}\cap I_{j_p}}}\ r_i + \smashoperator[r]{\sum_{i\in S_{\ell-1}\cap I_\ell}}\ r_i \equiv 0 \pmod{m}\enspace,
$$
where we use that $I_\ell$ is the first interval that overlaps with other intervals, and that because we are in phase~1, each individual sum in the last expression is $0\pmod*{m}$.
Moreover, note that $S_{j_1-1}\cap I_{j_1} \subsetneq S_{j_1-1}\cap I'$, hence $|S_{j_1-1}\cap I_{j_1}| < |S_{j_1-1}\cap I'|$, contradicting the choice of $I_{j_1}$ to maximize $|S_{j_1-1}\cap I_{j_1}|$.
Thus, the intervals $I_j$ obtained in phase~1 are all disjoint, hence in particular, we always have $S_{j-1}\cap I_j = I_j$, i.e., in step~$j$, we remove precisely the terms with indices in $I_j$ from the sum.

Next, recall the way that residues $r_i$ were defined in~\eqref{eq:defineResidues}: They come in $n$ chunks of equal residues, namely with indices in $C_1=\{1, \ldots, \lambda_1\}$, $C_2=\{\lambda_1+1, \ldots, \lambda_1+\lambda_2\}$, \ldots, $C_n=\{\lambda_1+\ldots+\lambda_{n-1}+1, \ldots, \Lambda\}$.
We observe that each of those chunks $C_i$ can contain at most $2$ endpoints of intervals $I_j$ that are constructed during phase~1. To see this, assume for the sake of deriving a contradiction that one $C_\ell$ contains at least three interval endpoints. We distinguish two cases:
\begin{itemize}
\item $C_\ell$ contains both endpoints of an interval $I_j=\{i_1^j, \ldots, i_2^j\}$, and (at least) one endpoint of $I_k=\{i_1^k, \ldots, i_2^k\}$. Intervals do not overlap, so assume without loss of generality that $i_2^j < i_1^k$ and choose $k$ such that $i_1^k$ is smallest possible. We claim that instead of $I_j$ or $I_k$ (whichever was deleted first), we could also have chosen the larger interval $I'=\{i_1^k - i_2^j + i_1^j, \ldots, i_2^k\}$: Indeed,
$$
\sum_{i\in I'}r_i = \sum_{i=i_1^k - i_2^j + i_1^j-1}^{i_1^k - 1} r_i + \sum_{i = i_1^k}^{i_2^k} r_i = \sum_{i\in I_j} r_i + \sum_{i\in I_k}r_i \equiv 0 \pmod{m}\enspace,
$$
where we use that $r_i = r_{i'}$ for all $i, i'\in C_\ell$, and that because we are in phase~1, each individual sum in the last expression is $0\pmod*{m}$.
Because $|I_j|, |I_k|<|I'|$, this contradicts the choice of intervals $I_j$ such that $|S_{j-1}\cap I_j|=|I_j|$ is maximized.

\item $C_\ell$ does not contain both endpoints of any interval $I_j$. This implies that every interval that has one endpoint in $C_\ell$ contains at least one of the minimum or maximum indices in $C_\ell$. Consequently, if $C_\ell$ contains at least three endpoints, one of these two indices is covered by at least two intervals, contradicting that intervals are disjoint in phase~1.
\end{itemize}
This proves that every $C_i$ can contain at most $2$ endpoints of intervals constructed in phase~1, hence there can be at most $n$ such intervals, and phase~1 ends after at most $n$ steps. This proves claim~\ref{claimitem:2nsteps}.\footnote{We remark that a slightly more careful analysis, in particular of endpoints in $C_1$ and $C_n$, immediately improves this bound to $n-1$, but this is not needed for our purpose.}

Finally, and to complete the proof of \cref{lem:transformSolutionEfficiently}, we focus on claim~\ref{claimitem:maxInterval} above, i.e., on how to efficiently find intervals $I_j$ maximizing $|S_{j-1}\cap I_j|$.
To this end, let us recall what the situation is:
We are given $r_1,\ldots,r_\Lambda$ as defined in~\eqref{eq:defineResidues}, and a set $S$ of target residues (in phase~1, $S$ will contain a single residue; in phase~2, it will be equal to $R'$ from~\eqref{eq:defineResidues}) such that $\sum_{i\in [\Lambda]} r_i\in S\pmod*{m}$, and the goal is to identify an interval $I=\{i_1,\ldots,i_2\}\subseteq[\Lambda]$ such that $\sum_{i\in[\Lambda]\setminus I} r_i\in S\pmod*{m}$, and $|I|$ has maximum possible size.
Observe that if we update the values $\lambda_i$, $\Lambda$, and $r_i$ accordingly to reflect the remaining sum after each step of the procedure, this is the precise setup that we are faced with in each step.
In what follows, we show that an optimal interval $I=\{i_1,\ldots,i_2\}$ can be identified after solving $O(n^2 |S|)$ many IPs with a constant number of variables and a constant number of constraints.

To see this, let $C_j$ and $C_k$ (as defined earlier), with $1\leq j\leq k\leq n$, be such that $i_1\in C_j$ and $i_2\in C_k$.
If $j<k$, then $i_1 = \sum_{i=1}^{j-1}\lambda_i + \tau_1$ for some $\tau_1\in[\lambda_j]$, and $i_2 = \sum_{i=1}^{k-1}\lambda_i + \tau_2$ for some $\tau_2\in[\lambda_k]$,
\begin{equation*}
\sum_{i\in[\Lambda]\setminus I} r_i = \sum_{i\in[\Lambda]} r_i - \left(
(\lambda_j-\tau_1+1)r_{\lambda_1+\ldots+\lambda_j+1} + \sum_{i=j+1}^{k-1}\lambda_i r_{\lambda_i+\ldots+\lambda_i+1} + \tau_2 r_{\lambda_1+\ldots+\lambda_{k}+1}
\right)\enspace,
\end{equation*}
and thus
\begin{equation*}
\sum_{i\in[\Lambda]\setminus I} r_i \in S\pmod{m}
\quad\iff\quad
-\tau_1 r_{\lambda_1+\ldots+\lambda_j+1} + \tau_2 r_{\lambda_1+\ldots+\lambda_k+1} \in S' \pmod*{m}\enspace,
\end{equation*}
where $S'$ is a shifted version of $S$.
Moreover, observe that $|I|=\lambda_j-\tau_1+1 + \sum_{i=j+1}^{k-1}\lambda_i + \tau_2$, hence $|I|$ is of maximum size if $\tau_2-\tau_1$ is maximized.
Altogether, we obtain that $(\tau_1,\tau_2)$ is an optimal solution of
\begin{equation}\label{eq:IP_j<k}
\begin{array}{rrrcl}
\max_{s\in S'} & \max & \tau_2-\tau_1 \hfill & \\
&& -\tau_1 r_{\lambda_1+\ldots+\lambda_j+1} + \tau_2 r_{\lambda_1+\ldots+\lambda_k+1} & = & zm + s \\
&& \tau_1 & \in & [\lambda_j]\\
&& \tau_2 & \in & [\lambda_k]\\
&& z & \in & \mathbb{Z}\enspace.
\end{array}
\end{equation}
Similarly, if $j=k$, then $i_1=\sum_{i=1}^{j-1}\lambda_i + \tau_1$ and $i_2 = \sum_{i=1}^{j-1}\lambda_i + \tau_2$ for some $\tau_1,\tau_2\in[\lambda_j]$ with $\tau_1\leq \tau_2$, and we have
\begin{multline*}
\sum_{i\in[\Lambda]\setminus I} r_i = \sum_{i\in[\Lambda]} r_i - (\tau_2-\tau_1+1)r_{\lambda_1+\ldots+\lambda_j+1} \in S\pmod{m}\\
\quad\iff\quad
(\tau_2-\tau_1) r_{\lambda_1+\ldots+\lambda_j+1} \in S' \pmod*{m}\enspace,
\end{multline*}
where again, $S'$ is a shifted version of $S$.
Moreover, $|I|=\tau_2-\tau_1+1$, hence $|I|$ is of maximum size if $\tau_2-\tau_1$ is maximized.
Thus, we obtain that $(\tau_1,\tau_2)$ is an optimal solution of
\begin{equation}\label{eq:IP_j=k}
\begin{array}{rrrcl}
\max_{s\in S'} & \max & \tau_2-\tau_1 \hfill & \\
&& (\tau_2-\tau_1) r_{\lambda_1+\ldots+\lambda_j+1} & = & zm + s \\
&& \tau_1 & \leq & \tau_2 \\
&& \tau_1, \tau_2 & \in & [\lambda_j]\\
&& z & \in & \mathbb{Z}\enspace.
\end{array}
\end{equation}
Finally, observe that to solve the problems in~\eqref{eq:IP_j<k} and~\eqref{eq:IP_j=k}, it is enough to solve the inner maximization problem for every $s\in S'$.
Given $s$, these maximization problems are integer programs with $3$ variables and a constant number of constraints, and can thus be solved in time polynomial in the encoding size of the IP using Lenstra's algorithm~\cite{lenstra_1983_IPfixed}, which is strongly polynomial in the size of the \RCCTUF{} problem.
Moreover, for fixed $j$ and $k$, it is immediate that a solution of~\eqref{eq:IP_j<k} or~\eqref{eq:IP_j=k} (if it exists) corresponds to a largest possible interval $I$ with endpoints in $C_j$ and $C_k$.
Altogether, by going through the $O(n^2)$ many options for $j,k\in[n]$, we can in strongly polynomial time determine the optimal interval $I$. This proves claim~\ref{claimitem:maxInterval}, and thus concludes the proof of \cref{lem:transformSolutionEfficiently}.
\end{proof}

\section{Solving base block problems\MOORdot}\label{sec:baseBlocks}

In this section, we discuss how to solve \CCTU{} problems---and thus also \CCTUF{} and \RCCTUF{} problems---whose constraint matrices are base block matrices, i.e., matrices falling into case~\ref{thmitem:TUdecomp_netw} or~\ref{thmitem:TUdecomp_const} of \cref{thm:TUdecomp}. Note that we can always assume to start with a \CCTU{} problem whose relaxation is feasible, which we can check in strongly polynomial time; for otherwise, the \CCTU{} problem is clearly infeasible. Hence, we assume feasibility of the relaxation throughout this section.
To start with, let us recall the definition of a \emph{network matrix}.
\begin{definition}\label{def:networkMatrix}
A matrix $T$ is a \emph{network matrix} if the rows of $T$ can be indexed by the edges of a directed spanning tree $(V, U)$, and the columns can be indexed by the edges of a directed graph $(V, A)$ on the same vertex set, such that for every arc $a=(v, w)\in A$ and every arc $u\in U$,
$$
T_{u,a} = \begin{cases}
1 & \text{if the unique $v$-$w$ path in $U$ passes through $u$ forwardly,}\\
0 & \text{if the unique $v$-$w$ path in $U$ does not pass through $u$,}\\
-1 & \text{if the unique $v$-$w$ path in $U$ passes through $u$ backwardly.}
\end{cases}
$$
\end{definition}

Note that here, a directed graph is called a spanning tree if it is a spanning tree when ignoring edge directions. Moreover, we remark that we allow graphs to have several parallel edges connecting the same two vertices. In particular, the graph $(V,A)$ in the above definition may have parallel edges, which correspond to identical columns of $T$. An important fact for our purposes is the following.

\begin{lemma}[see, for example, \cite{schrijver1998theory}]\label{lem:networkMatrixGraphs}
Given a matrix $T$, one can in strongly polynomial time recognize whether it is a network matrix. If so, a directed graph $(V,A)$ and a directed tree $(V, U)$ as in \cref{def:networkMatrix} can be found efficiently.
\end{lemma}

In the subsequent three sections, we distinguish three cases, namely whether the constraint matrix $T$ of the \CCTU{} problem that we consider is a network matrix, the transpose of a network matrix, or a matrix stemming from the constant-size matrices given in case~\ref{thmitem:TUdecomp_const} of \cref{thm:TUdecomp}. As indicated above, we show in each case that the corresponding \CCTU{} problem can be  solved efficiently under some assumptions, thus implying \cref{thm:baseBlocks}, which covers the corresponding feasibility problems.

In the case of network matrices and their transposes, we perform reductions to combinatorial problems.
In this context, it is convenient to transform the \CCTU{} problems into a more structured class of \CCTU{} problems, which we call \emph{normalized \CCTU{}} problems and are defined as follows.
\begin{definition}[Normalized \CCTU{} problem]
A problem of the form
\begin{equation}\label{eq:normalizedCCTU}
\min \left\{c^\top x \colon Tx \leq b,\ \gamma^{\top} x \equiv r \pmod*{m}, x\in \mathbb{Z}^n_{\geq 0}\right\}
\end{equation}
fulfilling that the origin is an optimal solution to the relaxation of~\eqref{eq:normalizedCCTU}, is called a \emph{normalized \CCTU{}} problem.
\end{definition}
Note that the right-hand side $b$ of a normalized \CCTU{} problem is non-negative because the origin is feasible.
As we briefly discuss in the following, it is not hard to see that one can assume to deal with normalized \CCTU{} problems, as formalized in the following observation.
\begin{observation}\label{obs:normalization}
Every \CCTU{} problem can be reduced in strongly polynomial time to a normalized \CCTU{} problem. Furthermore, if the constraint matrix of the first problem is a base block matrix, the constraint matrix of the latter problem is a base block matrix of the same type.
\end{observation}

\begin{proof}
Indeed, consider an arbitrary \CCTU{} problem (with feasible relaxation)
\begin{equation}\label{eq:cctuWlogNormalized}
\min\left\{c^\top x \colon Tx \leq b,\ \gamma^\top x \equiv r \pmod*{m},\ x\in\mathbb{Z}^n \right\}\enspace.
\end{equation}
An equivalent \CCTU{} problem where the origin is an optimal solution to its relaxation can simply be obtained by a standard shifting argument.
To this end, assume first that the relaxation has a finite optimal solution.
In this case we compute such a finite optimal solution $x_0$, and then substitute $x$ by $x'+x_0$ to obtain the equivalent \CCTU{} problem
\begin{equation*}
\min\left\{c^\top x' \colon Tx' \leq b',\ \gamma^\top x' \equiv r' \pmod*{m},\ x'\in\mathbb{Z}^n \right\}\enspace,
\end{equation*}
where $b'=b-Tx_0$ and $r'=r-\gamma^\top x_0$.
Clearly, the origin is an optimal solution to the relaxation of this transformed problem.
In case the relaxation is unbounded, we know by \cref{lem:unboundedness} that~\eqref{eq:cctuWlogNormalized} is either infeasible or unbounded.
Hence, it is unbounded if and only if it is feasible.
Moreover, \cref{lem:unboundedness} allows for obtaining efficiently a description of a set of unbounded solutions from any solution to \eqref{eq:cctuWlogNormalized}.
Hence, in this case, the optimization problem for \eqref{eq:cctuWlogNormalized} is equivalent to its feasibility version, and we can therefore replace the objective $c$ by an all-zeros objective.
This brings us back to the first case where the relaxation has a finite optimum.

Furthermore, to reduce to non-negative variables we can use another standard transformation that replaces every variable $x\in\mathbb{Z}$ by the difference $x^+ - x^-$ of two non-negative variables $x^+, x^-\in\mathbb{Z}_{\geq 0}$.
Notably, these substitutions change the constraint matrix, but it can be observed that base block matrices remain base block matrices of the same type.\footnote{Indeed, if we start with a constraint matrix $T$, this transformation to non-negative variables will lead to constraints described by the constraint matrix $[T\ -T]$ together with non-negativity constraints.
Moreover, each of the base block matrix types is closed under copying columns, changing the signs of columns, and adding rows with a single non-zero entry.
}
Applying this reduction on top of the previous one, we maintain that the origin is an optimal solution to the relaxation, thus obtaining \cref{obs:normalization}.
\end{proof}

Moreover, note that by our proximity result, \cref{thm:proximity}, we obtain that a normalized \CCTU{} problem has an optimal solution $x^*$ with $\|x^*\|_{\infty}\leq m-1$.
Due to the non-negativity of the variables in a normalized \CCTU{} problem, we thus obtain that there is an optimal solution $x^*$ with $x^*_i \in \{0,\ldots, m-1\}$ for each entry $i\in [n]$.
This is a property we repeatedly exploit in our reductions developed in the following.

\subsection{Network matrices\MOORdot}

In this section, we show that \CCTU{} problems with unary encoded objectives and constraint matrices that are network matrices can be solved efficiently using a randomized algorithm.

\begin{theorem}\label{thm:solveNetwRPP}
There is a strongly polynomial time randomized algorithm for \CCTU{} problems with unary encoded objectives, constant modulus and constraint matrices that are network matrices.
\end{theorem}

Our approach in this case is to exploit the graph structure that comes with network matrices to interpret \CCTU{} problems (or, more precisely, normalized \CCTU{} problems) with network constraint matrices as minimum-cost congruency-constrained circulation problems in certain directed graphs. To get started, let us recall that a circulation $f$ in a directed graph $G=(V,A)$ with capacities $u\colon A\to\mathbb{Z}_{\geq 0}$ is a mapping $f\colon A\to\mathbb{Z}_{\geq 0}$ such that $f(a)\leq u(a)$ for every arc $a\in A$, and $f(\delta^+(v)) = f(\delta^-(v))$ for every vertex $v\in V$. Given arc lengths $\ell \colon A\to\mathbb{Z}$, the length of a circulation $f$ is $\ell(f)\coloneqq\sum_{a\in A}\ell(a)f(a)$. Note that here, arc lengths are allowed to be negative.

A congruency-constrained circulation problem is formally defined as follows.
\begin{mdframed}[innerleftmargin=0.5em, innertopmargin=0.5em, innerrightmargin=0.5em, innerbottommargin=0.5em, userdefinedwidth=0.95\linewidth, align=center]
{\textbf{Congruency-Constrained Circulation (\hypertarget{prb:CCC}{CCC}):}}
Let $G=(V,A)$ be a directed graph with capacities $u\colon A\to\mathbb{Z}_{\geq 0}$, arc lengths $\ell\colon A\to\mathbb{Z}$, and let $\eta\colon A\to\mathbb{Z}$, $r\in\mathbb{Z}$, and $m\in\mathbb{Z}_{>0}$. Find a minimum-length circulation $f\colon A\to\mathbb{Z}_{\geq 0}$ in the given network such that $\sum_{a\in A}\eta(a)f(a) \equiv r\pmod*{m}$.
\end{mdframed}

\noindent The lemma below reduces \CCTU{} problems with constraint matrices that are network matrices to \CCC{} problems.
\begin{lemma}\label{lem:CCTUtoCCC}
\CCTU{} problems with modulus $m$, objective vector $c$, and constraint matrices that are network matrices can be reduced in strongly polynomial time to \CCC{} problems with modulus $m$, capacities $u$ within $\{0, \ldots, m-1\}$, and arc lengths $\ell$ with $\|\ell\|_\infty\leq \|c\|_{\infty}$.
\end{lemma}
\begin{proof}
First of all, we know by \cref{obs:normalization} that any \CCTU{} problem with a constraint matrix that is a network matrix can be efficiently reduced to a normalized \CCTU{} problem with a constraint matrix of the same type.
Thus, assume we are given a normalized problem of the form
\begin{equation*}
\min\left\{c^\top x \colon Tx \leq b,\ \gamma^\top x \equiv r \pmod*{m},\ x\in\mathbb{Z}^n_{\geq 0} \right\}
\end{equation*}
with a network matrix $T$.
By \cref{thm:proximityGeneral}, we have that there is an optimal solution $x$ to the above problem with $|d^\top x|\leq m-1$ for all $d$ that are TU-appendable to $T$.

We now define a \CCC{} problem to which the above \CCTU{} problem reduces.
To this end, let $(V, U)$ be the directed tree whose edges index the rows of the network matrix $T$, and let $(V, E)$ be the digraph whose edges index the columns of $T$, as described in \cref{def:networkMatrix}.
Let $G$ be the directed graph with vertex set $V$ and edge set $A\coloneqq U\cup \cev{U} \cup \cev{E}$, where $\cev{U}\coloneqq\{(w, v)\colon (v, w)\in U\}$ and analogously $\cev{E}\coloneqq \{(w,v) \colon (v,w) \in E\}$.
Moreover, for an arc $u=(v, w)$, denote by $\cev{u}=(w,v)$ the corresponding reverse arc.
We define the capacities $u\colon A \to\mathbb{Z}_{\geq 0}$, lengths $\ell\colon A\to\mathbb{Z}$, and values $\eta\colon A\to \mathbb{Z}$ of the \CCC{} problem as follows. For all $a\in A$,
\begin{align*}
u(a) &\coloneqq \begin{cases}
\min\{b_a, m-1\} & \text{if $a\in U$}\\ m-1 & \text{if $a\in\cev{U}\cup \cev{E}$}
\end{cases}\enspace,\\
\ell(a) &\coloneqq \begin{cases}
c_{\cev{a}} & \text{if $a\in \cev{E}$} \\ 0 & \text{if $a\in U\cup \cev{U}$}
\end{cases}\enspace,\quad\text{and}\\
\eta(a) &\coloneqq \begin{cases}
\gamma(\cev{a}) & \text{if } a\in \cev{E} \\
0               & \text{if } a\in U\cup \cev{U}
\end{cases}\enspace.
\end{align*}
Moreover, the modulus and target residue of the \CCC{} problem are the same as of the \CCTU{} problem, i.e., $m$ and $r$, respectively.
This concludes the definition of the \CCC{} problem to which we reduce.

Finally, the desired statement follows directly from the following claim, which relates solutions of the \CCTU{} problem to feasible circulations of the above-defined \CCC{} problem.
\begin{claim}\label{claim:relRelaxedSolsCirculations}
There is a solution of the \CCC{} problem of length no larger than the optimal value of the \CCTU{} problem.
Conversely, given a circulation $f$ for the \CCC{} problem, one can compute in strongly polynomial time a solution $x$ of the \CCTU{} problem with $c^{\top} x = \ell(f)$.
\end{claim}

To see the forward direction of the claim, we start with an optimal solution $x$ to the \CCTU{} problem.
By \cref{thm:proximityGeneral}, we can assume that $|d^{\top} x| \leq m-1$ for all $d$ that are TU-appendable to $T$.
In particular, this implies $x\in \{0,\ldots, m-1\}^E$.

We now start by defining a circulation $g \colon A\to\mathbb{Z}_{\geq 0}$ (that may violate the capacity constraints given by $u$) by
\begin{equation}\label{eq:fTildeDef}
g(a) \coloneqq \sum_{e\in E} x(e)\left(\chi^{\cev{e}}(a) + \chi^{P_e}(a)\right)\enspace,
\end{equation}
where, for every $e=(v, w)\in E$, the set $P_e\subseteq U \cup \cev{U}$ is the unique path from $v$ to $w$ in $U\cup \cev U$ that has all edges directed from $v$ to $w$.
Finally, the circulation $f\colon A\to \mathbb{Z}_{\geq 0}$ that corresponds to $x$ is obtained from $g$ by canceling out flows on arcs in opposite directions.
Formally, we set
\begin{equation*}
f(a) \coloneqq \begin{cases}
g(a) &\text{if } a\in \cev{E}\enspace,\\
g(a) - \min\{g(a),g(\cev{a})\}     &\text{if }a\in U \cup \cev{U}\enspace.\\
\end{cases}
\end{equation*}
Hence, one can interpret $f$ as being obtained from $g$ by canceling flow on $2$-cycles.
By the definition of the lengths $\ell$, one immediately obtains $\ell(f) = c^{\top} x$ as desired.
Moreover, because $x$ is integral, we have that $g$ is integral and therefore also $f$.
Also, $\sum_{a\in A} \eta(a) f(a) = \gamma^\top x \equiv r \pmod*{m}$.
It remains to observe that $f$ is a circulation, i.e., each vertex has the same in-flow as out-flow with respect to $f$ and $f$ fulfills the capacity constraints given by $u$.

Note that each vertex has the same in- and out-flow with respect to $g$, because every term in~\eqref{eq:fTildeDef} corresponds to sending a flow of $x(e)$ along the cycle $\cev{e}\cup P_e$.
Because $f$ is obtained from $g$ by canceling flow on $2$-cycles, also $f$ has the same in- and out-flow at every vertex.

It remains to verify that the capacities given by $u$ are respected by $f$.
The capacities of arcs $a\in \cev{E}$, which are $u(a)=m-1$, are fulfilled by $f$ because $x(e)\leq m-1$.
Consider now an arc $a\in U$ and denote by $C_a\subseteq V$ the unique cut in $(V,U)$ that satisfies $\delta^+(C_a)=\{a\}$ and $\delta^-(C_a)=\emptyset$.
Such a cut exists as $(V,U)$ is a tree.
Because $f$ is a circulation, we have
\begin{multline}\label{eq:linkFlowToSolution}
0 = f(\delta^+(C_a)) - f(\delta^-(C_a)) = f(a) - f(\cev{a}) + f(\delta^+(C_a)\cap {\cev{E}}) - f(\delta^-(C_a)\cap{\cev{E}})\\
\iff\qquad f(a) - f(\cev{a}) = \sum_{e\in E\colon a\in P_e} x(e) - \sum_{e\in E\colon \cev{a}\in P_e} x(e)\enspace.
\end{multline}
Observe that the difference of the last two sums is precisely $d^\top x$, where $d$ is the row vector of $T$ indexed by $u$.
Because $d^{\top} x\leq b_a$ is a constraint of the original normalized \CCTU{} problem, we have
$f(a) - f(\cev{a}) = d^{\top} x \leq b_a$.
Moreover, because both $d^{\top}$ and $-d^{\top}$ are TU-appendable to the constraint matrix $T$, we obtain by \cref{thm:proximityGeneral} that $-m+1 \leq f(a) - f(\cev{a}) \leq m-1$.
Hence,
\begin{equation*}
-(m-1) \leq f(a)-f(\cev{a}) \leq \min\{b_a, m-1\}\enspace.
\end{equation*}
The above inequality implies $f(a) \leq \min\{b_a, m-1\} + f(\cev{a}) = u(a) + f(\cev{a})$ and $f(\cev{a})\leq m-1 + f(a) = u(\cev{a}) + f(a)$.
Note that because $f$ has by definition a value of zero on either $a$ or $\cev{a}$, and $u(a)\geq 0$, it follows from these inequalities that both $f(a) \leq u(a)$ and $f(\cev{a})\leq u(\cev{a})$ hold.
Thus, $f$ also fulfills the capacity constraints for all arcs in $U\cup \cev{U}$.

\medskip

For the backward direction of \cref{claim:relRelaxedSolsCirculations}, assume that we are given an integral circulation $f$ in $G$ respecting the capacity constraints $u$, and define $x(e) \coloneqq f(\cev{e})$ for all $e\in E$. Note that we thus obtain $x$ in strongly polynomial time.
Again,~\eqref{eq:linkFlowToSolution} holds and the right-hand side is $d^\top x$, where $d$ is the row indexed by $a$ in $T$.
Non-negativity of $f$ and the capacity constraints then imply for all $a\in A$ that
$$d^\top x = f(a) - f(\cev{a}) \leq f(a) \leq b_a\enspace.$$
Hence, $x$ satisfies all constraints $Tx\leq b$ and is non-negative due to non-negativity of $f$.
Moreover, we again have
\begin{equation*}
\gamma^{\top} x = \sum_{e\in E} \gamma(e) x(e) = \sum_{e\in \cev{E}} \eta(\cev{e}) f(\cev{e}) = \sum_{a\in A} \eta(a) f(a) \equiv r \pmod*{m}\enspace.
\end{equation*}
Hence, the vector $x$ is feasible for the \CCTU{} problem.
This proves the claim, which in turn implies the statement of \cref{lem:CCTUtoCCC}, as desired.
\end{proof}

We remark that for modulus $m=2$, an analogous reduction to the one we used in the proof of \cref{lem:CCTUtoCCC} was already done in~\cite{artmann_2017_strongly}. Our reduction is a generalization of that one.
For the special case with modulus $m=2$, the resulting \CCC{} problems are non-trivial only if $r=1$, i.e., when the goal is to find an odd circulation. This can easily be reduced to finding a shortest odd cycle in a suitable auxiliary graph, which can be solved via standard techniques.
For general $m$, however, the solution structure can be significantly more complex.
We observe and exploit a connection to the so-called \emph{exact length circulation problem}, where the goal is to find a circulation whose length is equal to a given value.

\begin{mdframed}[innerleftmargin=0.5em, innertopmargin=0.5em, innerrightmargin=0.5em, innerbottommargin=0.5em, userdefinedwidth=0.95\linewidth, align=center]
{\textbf{Exact Length Circulation (\hypertarget{prb:XLC}{XLC}):}}
Let $G=(V,A)$ be a digraph with capacities $u\colon A\to\mathbb{Z}_{>0}$ and arc lengths $\ell \colon A\to\mathbb{Z}$. Given $L\in\mathbb{Z}$, find a circulation $f$ in the given network such that $\ell(f)=L$.
\end{mdframed}

\noindent Exact length circulation problems can be solved using a randomized pseudopolynomial algorithm, as shown by \textcite{camerini_1992_rpp}.
They reduce the problem to an exact cost perfect matching problem, which can then be reduced to computing the coefficients of a well-defined polynomial.
The following theorem summarizes the result of \textcite{camerini_1992_rpp} for \XLC{}.
\begin{theorem}[\cite{camerini_1992_rpp}]
There is a randomized algorithm for \XLC{} problems in a directed graph $G=(V,E)$ with capacities $u\colon A\to\mathbb{Z}_{\geq 0}$ in time $\mathrm{poly}(|V|, \max_{a\in A}u(a), \max_{a\in A}|\ell(a)|)$.
\end{theorem}

Thus, it remains to build the connection between \CCC{} and \XLC{} problems.
We achieve this by integrating the contributions $\eta(a)$ of every arc towards the congruency constraint into its length, and searching for the minimum length of a suitable circulation using binary search, thereby obtaining the following lemma.
\begin{lemma}\label{lem:CCCtoXLC}
A \CCC{} problem in a graph $G=(V,A)$ with arc lengths $\ell\colon A\to\mathbb{Z}$, capacities $u\colon A\to\{0, 1, \ldots, m-1\}$, and modulus $m$ can be polynomially reduced to $\mathrm{poly}(m, |V|, |A|, \max_{a\in A}|\ell(a)|)$ many \XLC{} problems in $G$ with the same capacities.
\end{lemma}

\begin{proof}
Note that in any \CCC{} problem, we may assume without loss of generality that $\eta(a)\in\{0,\ldots,m-1\}$ by reducing the values modulo $m$.
Now, for every arc $a$ in a given \CCC{} problem, define a new length function $\widetilde{\ell}(a)=\ell(a)\cdot m^2 |A| + \eta(a)$.
We thus have $\widetilde{\ell}(f) = \ell(f) \cdot m^2 |A| + \sum_{a\in A}\eta(a)f(a)$, and because $\sum_{a\in A}\eta(a)f(a) < m^2 |A|$, we can retrieve both $\ell(f)$ and $\sum_{a\in A}\eta(a)f(a)$ from $\widetilde{\ell}(f)$.
Consequently, finding a circulation of length $L$ with $\sum_{a\in A}\eta(a)f(a)\equiv r\pmod{m}$ is equivalent to solving \XLC{} problems in $G$ with respect to lengths $\widetilde{\ell}$ and with target length $\widetilde{L}=L\cdot m^2|A| + km + r$ for all $k\in\{0,\ldots, m|A|-1\}$.
We can find the smallest $L$ for which there is a \CCC{} solution of length $L$ by binary search in $O(\log (m|A|\cdot\max_{a\in A}|\ell(a)|))$ iterations, because $|\ell(f)| = \left|\sum_{a\in A}\ell(a)f(a)\right| \leq m|A|\cdot\max_{a\in A}|\ell(a)|$.
Altogether, this gives the desired result.
\end{proof}

Combining the above findings, we conclude this section with a proof of \cref{thm:solveNetwRPP}.

\begin{proof}[Proof of \cref{thm:solveNetwRPP}]
By \cref{lem:CCTUtoCCC}, a \CCTU{} problem whose constraint matrix is a network matrix can be reduced in strongly polynomial time to a \CCC{} problem with $u(a)\leq m-1$ for all $a\in A$.
By \cref{lem:CCCtoXLC}, this problem further reduces to $\mathrm{poly}(m, |V|, |A|, \max_{a\in A}|c(a)|)$ many \XLC{} problems, where each of them can be solved in $\mathrm{poly}(|V|, \max_{a\in A}u(a), \max_{a\in A}|c(a)|)=\mathrm{poly}(|V|, m, \max_{a\in A}|c(a)|)$ time using a randomized algorithm.
Thus, overall, we obtain that there is a randomized algorithm to solve a \CCTU{} problem whose constraint matrix is a network matrix in time $\mathrm{poly}(m, |V|, |A|, \max_{a\in A}|c(a)|)$, i.e., a strongly polynomial algorithm if the objective $c$ is given in unary encoding and $m$ is a constant.
\end{proof}

\subsection{Transposes of network matrices\MOORdot}

The purpose of this section is to prove the following theorem.

\begin{theorem}\label{thm:solveTranspNetwPP}
There is a strongly polynomial time algorithm for \CCTU{} problems with constant prime power modulus and constraint matrices that are transposed network matrices.
\end{theorem}

To achieve this result, we again exploit the graph structure coming with network matrices. This time, we reduce \CCTU{} problems (or, more precisely and equivalently, normalized \CCTU{} problems) to certain directed cut problems of the following form.

\begin{mdframed}[innerleftmargin=0.5em, innertopmargin=0.5em, innerrightmargin=0.5em, innerbottommargin=0.5em, userdefinedwidth=0.95\linewidth, align=center]
{\textbf{Constrained Tree Cuts (\hypertarget{prb:CTC}{CTC}):}}
Let $T=(V, U)$ be a directed tree, $A\subseteq V\times V$ and $b\colon A\to\mathbb{Z}_{\geq0}$. Let $c\colon U\to\mathbb{Z}$ be arc costs, $\alpha\colon V\to\mathbb{Z}$, $r\in\mathbb{Z}$, and $m\in\mathbb{Z}_{>0}$. Find a family of sets $S_1,\ldots,S_\ell\subseteq V$ minimizing the total cost $\sum_{i=1}^{\ell} c(\delta^+(S_i))$ such that
\begin{enumerate}
\item\label{prbitem:CTC_inedges} $\delta^-(S_i)=\emptyset$ for all $i\in[\ell]$,
\item\label{prbitem:CTC_coverage} $|\{i\in[\ell]\colon v\in S_i\}| - |\{i\in[\ell]\colon w \in S_i\}| \leq b_a$ for all $a=(v, w)\in A$, and
\item\label{prbitem:CTC_cong} $\sum_{i=1}^{\ell} \alpha(S_i)\equiv r\pmod*{m}$, where $\alpha(S_i)\coloneqq\sum_{v\in S_i} \alpha(v)$.
\end{enumerate}
\end{mdframed}

\noindent We highlight that in \CTC{} problems, the number $\ell\in\mathbb{Z}_{\geq 0}$ of sets that are returned is not fixed upfront; in the extreme case, we might even return an empty family, i.e., use $\ell=0$.
Moreover, we also allow the sets $S_i$ to be empty or equal to $V$, opposed to the typical setting in cut problems where this is usually excluded.
\CTC{} problems inherit many structural properties from \CCTU{} problems, including structural results on optimal solutions. These will allow us to further reduce \CTC{} problems to directed congruency-constrained minimum cut problems, for which efficient algorithms are known for the case of the modulus $m$ being a constant prime power~\cite{nagele_2018_submodular}. In \CTC{} problems, we call the constraint~\ref{prbitem:CTC_cong} the \emph{congruency constraint}, and we refer to the problem obtained after dropping that constraint as the \emph{relaxation} of the \CTC{} problem.

We start by showing the reduction from normalized \CCTU{} problems to \CTC{} problems.
More concretely, to every normalized \CCTU{} problem $\min\{c^\top x \colon Tx \leq b,\ \gamma^\top x \equiv r \pmod*{m},\ x\in\mathbb{Z}^n_{\geq 0} \}$ with $T$ being the transpose of a network matrix and such that $T$ does not contain identical rows (otherwise, one row of the identical rows corresponds to a redundant constraint and can be deleted), we associate the following \CTC{} problem:
The tree $(V,U)$ and the extra arc set $A\subseteq V\times V$ are those coming with the network constraint matrix through \cref{def:networkMatrix}, $b\colon A\to\mathbb{Z}_{\geq 0}$ is the right-hand side vector of the \CCTU{} problem (which is non-negative because we assume the \CCTU{} problem to be normalized), $\alpha\colon V\to\mathbb{Z}$ is defined by $\alpha(v)\coloneqq \gamma(\delta^+(v)) - \gamma(\delta^-(v))$ for all $v\in V$, and costs $c$ as well as $r$ and $m$ are left unchanged.\footnote{Assuming that $T$ does not contain identical rows implies that no parallel arcs are needed in $A$, which justifies the assumption $A\subseteq V\times V$.}
To relate feasible solutions of \CCTU{} problems and the associated \CTC{} problem, we prove the following result.

\begin{lemma}\label{lem:CCTUtoCTCfeasibility}
Consider a normalized \CCTU{} problem whose constraint matrix has no identical rows and is the transpose of a network matrix, and the associated \CTC{} problem. Let $S_1,\ldots,S_\ell\subseteq V$ with $\delta^{-}(S_i)=\emptyset$ for all $i\in[\ell]$, and define $x=\sum_{i=1}^\ell \chi^{\delta^+(S_i)}$. Then $x$ is a feasible \CCTU{} solution if and only if $S_1,\ldots,S_\ell$ is a feasible \CTC{} solution.
Moreover, if both are feasible, their objective values are the same.
\end{lemma}

The main ingredient in \cref{lem:CCTUtoCTCfeasibility} is to relate inequality constraints of the \CCTU{} problem and the constraints~\ref{prbitem:CTC_coverage} in the associated \CTC{} problems. We use this relation again later, and hence state it independently here before using it to prove \cref{lem:CCTUtoCTCfeasibility}.

\begin{lemma}\label{lem:CCTUtoCTCconstraints}
Let $(V,U)$ be a directed spanning tree, let $S_1,\ldots,S_\ell\subseteq V$ with $\delta^{-}(S_i)=\emptyset$ for all $i\in[\ell]$, and denote $x=\sum_{i=1}^\ell \chi^{\delta^+(S_i)}$. Then for any $v,w\in V$, the vector $t_{vw}\in\{-1,0,1\}^{U}$ defined by
$$
\forall u\in U\colon\quad
t_{vw}(u) = \begin{cases}
1 & \text{if the unique $v$-$w$ path in $U$ passes through $u$ forwardly,}\\
0 & \text{if the unique $v$-$w$ path in $U$ does not pass through $u$,}\\
-1 & \text{if the unique $v$-$w$ path in $U$ passes through $u$ backwardly}
\end{cases}
$$
satisfies $t_{vw}^\top x = |\{i\in[\ell]\colon v\in S_i\}| - |\{i\in[\ell]\colon w \in S_i\}|$.
\end{lemma}

\begin{proof}
By definition of $x$, we have that
\begin{equation*}
t_{vw}^\top x = \sum_{i=1}^\ell \sum_{u\in\delta^+(S_i)} t_{vw}(u)\enspace.
\end{equation*}
For fixed $i\in[\ell]$ and by definition of $t_{vw}$, the non-zero terms in the inner sum correspond to edges $u$ that are oriented from a vertex inside $S_i$ to a vertex outside $S_i$, and that lie on the unique $v$-$w$ path $P$ in $U$.
Recall that $\delta^-(S_i)=\emptyset$, hence the sum in fact has one non-zero term for every time the path $P$ crosses from one side of $S_i$ to the other.
More precisely, there is a term $+1$ for every time the path $P$ crosses from a vertex inside $S_i$ to one outside $S_i$, and a term $-1$ for every time the path $P$ crosses from a vertex outside $S_i$ to one inside $S_i$.
Consequently, the total value of the sum only depends on where the start- and endpoints $v$ and $w$ are located with respect to $S_i$: If $v\in S_i$ and $w\notin S_i$, for example, $P$ will cross from a vertex inside $S_i$ to one outside $S_i$ one more time than the other way round, hence the sum will be $+1$.
Generally, we get that $\sum_{u\in\delta^+(S_i)} t_{vw}(u)=1_{v\in S_i}-1_{w\in S_i}$, and thus
\begin{equation*}
t_{vw}^\top x = \sum_{i=1}^\ell \big(1_{v\in S_i}-1_{w\in S_i}\big) = \sum_{i=1}^\ell 1_{v\in S_i}-\sum_{i=1}^\ell 1_{w\in S_i} = |\{i\in[\ell]\colon v\in S_i\}| - |\{i\in[\ell]\colon w \in S_i\}|\enspace.\qedhere
\end{equation*}
\end{proof}

\begin{proof}[Proof of \cref{lem:CCTUtoCTCfeasibility}]
We start by showing that $x$ is feasible for the inequality system $Tx\leq b$ of the \CCTU{} problem if and only if $S_1,\ldots,S_\ell$ is feasible for constraint~\ref{prbitem:CTC_coverage} of the \CTC{} problem.
To this end, consider a row of the constraint matrix $T$ that is indexed by the arc $a=(v,w)\in A\times A$, and note that this row is precisely the vector $t_{vw}^\top$, with $t_{vw}$ as defined in \cref{lem:CCTUtoCTCconstraints}.
Consequently, the corresponding constraint $t_{vw}^\top x \leq b_a$ of the \CCTU{} problem is, by \cref{lem:CCTUtoCTCconstraints}, equivalent to $|\{i\in[\ell]\colon v\in S_i\}| - |\{i\in[\ell]\colon w \in S_i\}|\leq b_a$, which is one of the constraints in~\ref{prbitem:CTC_coverage} in the \CTC{} problem (namely the one for the arc $a=(v,w)\in A$).
Thus, we conclude that $Tx\leq b$ is equivalent to constraint~\ref{prbitem:CTC_coverage} in the \CTC{} problem.
Next, we observe that
\begin{align*}
\sum_{i=1}^\ell \alpha(S_i) &= \sum_{i=1}^\ell \sum_{v\in S_i}\left(\gamma(\delta^+(v))-\gamma(\delta^-(v))\right) \\
&= \sum_{i=1}^\ell\left(\gamma(\delta^+(S_i))-\gamma(\delta^-(S_i))\right) = \sum_{i=1}^\ell \gamma^\top \chi^{\delta^+(S_i)} = \gamma^\top x \enspace,
\end{align*}
and hence $\sum_{i=1}^\ell \alpha(S_i)\equiv r\pmod*{m}$ if and only if $\gamma^\top x\equiv r\pmod*{m}$.
Together, we obtain that $x$ is a feasible \CCTU{} solution if and only if $S_1,\ldots,S_\ell$ is a feasible \CTC{} solution.
To finish the proof of the lemma, we observe that the objectives of the \CCTU{} solution $x$ and the \CTC{} solution $S_1,\ldots,S_\ell$ are equal because $c^\top x = \sum_{i=1}^\ell c^\top\chi^{\delta^+(S_i)} = \sum_{i=1}^\ell c(\delta^+(S_i))$.
\end{proof}

By showing that for any feasible \CCTU{} solution $x$, there exist sets $S_1,\ldots,S_\ell\subseteq V$ with $\delta^{-}(S_i)=\emptyset$ and $x=\sum_{i=1}^\ell \chi^{\delta^+(S_i)}$, and combining this with \cref{lem:CCTUtoCTCfeasibility}, we thus obtain the following.

\begin{lemma}\label{lem:CCTUtoCTCsolutions}
Consider a normalized \CCTU{} problem whose constraint matrix has no identical rows and is the transpose of a network matrix, and the associated \CTC{} problem as constructed above.
\begin{enumerate}
\item\label{lemitem:feasibility} For every feasible solution $x$ of the \CCTU{} problem, there is a feasible solution $S_1,\ldots,S_\ell$ of the \CTC{} problem with the same objective value such that $x=\sum_{i=1}^\ell \chi^{\delta^+(S_i)}$.
\item\label{lemitem:optimality} For every optimal solution $x$ of the \CCTU{} problem, there is an optimal solution $S_1,\ldots,S_\ell$ of the \CTC{} problem such that $x=\sum_{i=1}^\ell \chi^{\delta^+(S_i)}$.
\end{enumerate}
\end{lemma}

\begin{proof}
\begin{enumerate}
\item Note that because $(V, U)$ is a tree, for every $u\in U$, there is a unique cut $C_u\subseteq V$ with $\delta^+(C_u)=\{u\}$ and $\delta^-(C_u)=\emptyset$. By definition, we have $x = \sum_{u\in U} x(u)\chi^{\delta^+(C_u)}$. Consequently, by \cref{lem:CCTUtoCTCfeasibility}, the collection consisting of $x(u)$ times the set $C_u$ for all $u\in U$ is a feasible \CTC{} solution, and its objective value is the same as the objective value of $x$ in the \CCTU{} problem.
\item By part~\ref{lemitem:feasibility}, it is enough to prove that the associated \CTC{} problem does not have solutions with objective value less than the value $c^\top x$ of $x$.
If there was such a \CTC{} solution, say $S_1',\ldots,S_{\ell'}'$, of value strictly less than $c^\top x$, then by \cref{lem:CCTUtoCTCfeasibility}, we know that $x'=\sum_{i=1}^{\ell'}\chi^{\delta^+(S_i')}$ is a feasible \CCTU{} solution of the same objective value---but this is a contradiction, since we assumed $x$ to be optimal for the \CCTU{} problem.\qedhere
\end{enumerate}
\end{proof}

In other words, the above immediately implies that \CCTU{} problems can be reduced to \CTC{} problems.

\begin{corollary}\label{cor:CCTUtoCTC}
Every normalized \CCTU{} problem whose constraint matrix has no identical rows and is the transpose of a network matrix can be strongly polynomially reduced to the associated \CTC{} problem, i.e., the \CTC{} problem can be obtained in strongly polynomial time, and any optimal \CTC{} solution can in strongly polynomial time be transformed to an optimal \CCTU{} solution.
\end{corollary}

\begin{proof}
The \CTC{} problem associated to a \CCTU{} problem can be constructed in strongly polynomial time, in particular because from the constraint matrix $T$, the tree $T=(V,U)$ and the extra arcs $A\subseteq V\times V$ can be obtained in polynomial time (in the encoding size of $T$) through \cref{lem:networkMatrixGraphs}.
\cref{lem:CCTUtoCTCsolutions}~\ref{lemitem:optimality} shows that optimal solutions of the \CCTU{} problem and the \CTC{} problem have the same values.
Moreover, by \cref{lem:CCTUtoCTCfeasibility}, any solution $S_1,\ldots,S_\ell$ of the \CTC{} problem immediately gives a feasible solution $x=\sum_{i=1}^\ell \chi^{\delta^+(S_i)}$ of the \CCTU{} problem with the same value (and note that $x$ can be computed in strongly polynomial time).
Thus if $S_1,\ldots,S_\ell$ is optimal for the \CTC{} problem, then so is $x$ for the \CCTU{} problem.
\end{proof}

We remark that the above reduction gives \CTC{} instances with $\alpha(V)=0$.
It turns out that because the underlying graph $(V,U)$ is a tree, this condition is enough to uniquely determine corresponding values $\gamma\colon U\to\mathbb{Z}$ such that $\alpha(v) = \gamma(\delta^+(v)) - \gamma(\delta^-(v))$ for all $v\in V$, which allows us to also reduce \CTC{} problems to \CCTU{} problems in that case.
As for our purposes, the direction covered by \cref{cor:CCTUtoCTC} is enough, we leave the details of this argument to the reader.
To be able to exploit the reduction given in \cref{cor:CCTUtoCTC}, we continue with studying the structure of \CTC{} solutions in more detail, with the goal to identify patterns that help for finding optimal \CTC{} solutions efficiently.

\begin{lemma}\label{lem:chainSolution}
Consider a \CTC{} problem and let $S_1,\ldots,S_\ell$ be a feasible solution. Then there exists a feasible solution $T_1,\ldots,T_{\ell}$ such that $T_{\ell}\subseteq T_{\ell-1}\subseteq\ldots\subseteq T_{1}$ and $\sum_{i=1}^\ell \chi^{\delta^+(S_i)} = \sum_{i=1}^{\ell}\chi^{\delta^+(T_i)}$.
\end{lemma}

\begin{proof}
If for all $j,k\in[\ell]$, we have $S_j\subseteq S_k$ or $S_k\subseteq S_j$, there is nothing to prove, because relabeling the sets to satisfy $S_\ell\subseteq S_{\ell-1}\subseteq\ldots\subseteq S_1$ will give the desired solution.
Thus, assume that there are two sets $S_j$ and $S_k$ for $j,k\in[\ell]$ such that $S_j\not\subseteq S_k$ and $S_k\not\subseteq S_j$.
We claim that removing the sets $S_j,S_k$ from the solution and adding the sets in $S_j\cup S_k$ and $S_j\cap S_k$ instead gives another feasible solution for the \CTC{} problem such that the sum $\sum_{i=1}^\ell \chi^{\delta^+(S_i)}$ is unchanged.
To see this, observe the following:
\begin{itemize}
\item $\delta^-(S_j\cup S_k) = \delta^-(S_j\cap S_k) = \emptyset$ because an arc entering the union or intersection of the two sets would enter at least one of the sets, but we know that $\delta^-(S_j) = \delta^-(S_k) = \emptyset$.
Thus, $\delta^-(S_i)=\emptyset$ holds for all $S_i$ in the new solution.
\item For any vertex $v\in V$, the number of sets in the solution that contain $v$ is invariant under replacing two sets with their union and intersection, hence the left-hand side of any constraint in condition~\ref{prbitem:CTC_coverage} of \CTC{} problems remains the same, and thus the constraints in condition~\ref{prbitem:CTC_coverage} of \CTC{} problems holds for the new solution, as well.
\item We have $\alpha(S_j) + \alpha(S_k) = \alpha(S_j\cup S_k) + \alpha(S_j\cap S_k)$, so the congruency-constraint is fulfilled by the new solution if and only if the initial solution fulfilled it.
\item Finally, it generally holds that
\begin{equation*}
 \chi^{\delta^+(S_j)} + \chi^{\delta^+(S_k)}
 = \chi^{\delta^+(S_j\cup S_k)} + \chi^{\delta^+(S_j\cap S_k)} + \chi^{U(S_j\setminus S_k, S_k\setminus S_j)} + \chi^{U(S_k\setminus S_j, S_j\setminus S_k)}\enspace,
\end{equation*}
where, for vertex sets $V_1,V_2\subseteq V$, we denote by $U(V_1,V_2)\subseteq U$ all arcs of $U$ with tail in $V_1$ and head in $V_2$.
Because $\delta^-(S_j) = \delta^-(S_k)=\emptyset$, we have $U(S_j\setminus S_k, S_k\setminus S_j)=U(S_k\setminus S_j, S_j\setminus S_k)=\emptyset$, which implies that the last two terms of the right-hand side above are zero. Consequently, $\chi^{\delta^+(S_j)} + \chi^{\delta^+(S_k)} = \chi^{\delta^+(S_j\cup S_k)} + \chi^{\delta^+(S_j\cap S_k)}$, and thus the sum $\sum_{i=1}^\ell \chi^{\delta^+(S_i)}$ is unchanged under the replacement step, as well.
\end{itemize}
Thus, as long as there are two sets $S_j$ and $S_k$ such that $S_j\not\subseteq S_k$ and $S_k\not\subseteq S_j$, we can replace them by $S_j\cup S_k$ and $S_j\cap S_k$ while maintaining feasibility for the \CTC{} problem and not changing the sum $\sum_{i=1}^\ell \chi^{\delta^+(S_i)}$.
To see that this procedure ends, note that in any step, the potential function $\Phi(S_1,\ldots,S_\ell)\coloneqq\sum_{i=1}^{\ell}|S_i|^2\in\mathbb{Z}$ strictly increases.
The latter follows from the fact that, for any two sets $A$ and $B$ with $A\nsubseteq B$ and $B\nsubseteq A$, we always have $|A|^2 + |B|^2 < |A\cap B|^2 + |A\cup B|^2$.
Obviously, $\Phi(S_1,\ldots,S_\ell)\leq \ell |V|^2$, so the procedure terminates after less than $\ell |V|^2$ many steps with a solution that has the desired properties.
\end{proof}

In the next lemma, we prove that in \CTC{} problems that are obtained via a reduction from \CCTU{} problems, there even exist optimal solutions that consist of a chain $S_\ell\subseteq\ldots\subseteq S_1$ with a bounded number of sets, namely $\ell\leq m-1$.
This closely links back to our general proximity result, \cref{thm:proximityGeneral}, from which we know that a normalized \CCTU{} problem has an optimal solution $x^*$ such that for any vector $d\in\mathbb{Z}^n$ that is TU-appendable to the constraint matrix $T$, we have $d^\top x^*\leq m-1$.
In the proof of the following lemma, we show that the optimal \CTC{} solution corresponding to such a \CCTU{} solution $x^*$ has the desired properties.

\begin{lemma}\label{lem:CTCsmallChain}
Consider a normalized \CCTU{} problem with modulus $m$ whose constraint matrix has no identical rows and is the transpose of a network matrix.
Then, the associated \CTC{} problem has an optimal solution $S_1,\ldots,S_\ell$ such that $S_\ell\subseteq S_{\ell-1}\subseteq\ldots\subseteq S_1$ and $\ell\leq m-1$.
\end{lemma}

\begin{proof}
Let $x^*$ be an optimal solution of the \CCTU{} problem such that for every vector $d\in\mathbb{Z}^n$ that is TU-appendable to the constraint matrix $T$, we have $d^\top x^*\leq m-1$.
Such a solution exists due to \cref{thm:proximityGeneral} because the \CCTU{} problem is normalized, and hence $x_0=(0 \enspace 0 \enspace \ldots \enspace 0 )^\top\in\mathbb{Z}^n$ is an optimal solution of its relaxation.
By \cref{lem:CCTUtoCTCsolutions}, there exists an optimal solution $S_1,\ldots,S_\ell$ of the associated \CTC{} problem such that $x^* = \sum_{i=1}^\ell \chi^{\delta^+(S_i)}$, and by \cref{lem:chainSolution}, we may even choose the sets $S_i\subseteq V$ such that they form a chain, i.e., $S_\ell\subseteq S_{\ell-1}\subseteq\ldots\subseteq S_1$.
Moreover, we may assume that $S_i\neq\emptyset$ and $S_i\neq V$ for all $i$: Such sets could be removed from the solution family without affecting feasibility of the solution (the left-hand sides of constraints in point~\ref{prbitem:CTC_coverage} of \CTC{} problems will remain the same, and because $\alpha(\emptyset)=0$ and $\alpha(V)=\sum_{v\in V}\alpha(v)=\sum_{v\in V}\gamma(\delta^+(v))-\gamma(\delta^-(v))=0$, the congruency constraint will still be satisfied, as well) and the objective value (which is the same because $\delta^+(V)=\delta^+(\emptyset)=\emptyset$, and thus $c(\delta^+(V))=c(\delta^+(\emptyset))=0$).

We claim that with the above assumptions, we have $\ell\leq m-1$.
To see this, choose $v\in S_1$ and $w\in V\setminus S_\ell$.
Note that such $v$ and $w$ exist by the assumption that $S_i\neq\emptyset$ and $S_i\neq V$ for all $i\in[\ell]$.
Let $t_{vw}\in\{-1,0,1\}^U$ be defined as in \cref{lem:CCTUtoCTCconstraints}, where $(V,U)$ is the directed tree indexing the columns of the constraint matrix $T$ of the \CCTU{} problem according to \cref{def:networkMatrix}.
By definition, $t_{vw}$ is TU-appendable to the matrix $T$, as we can add the arc $(v,w)$ to the arc (multi-)set indexing the rows of $T$ according to \cref{def:networkMatrix} and thereby obtain that the matrix $T$ with extra row $t_{vw}$ is the transpose of a network matrix again, and hence TU.
Consequently, by the choice of the optimal solution $x^*$, we have $t_{vw}^\top x^*\leq m-1$.
On the other hand, \cref{lem:CCTUtoCTCconstraints} implies that
$$
t_{vw}^\top x^* = |\{i\in[\ell]\colon v\in S_i\}| - |\{i\in[\ell]\colon w \in S_i\}| = \ell\enspace,
$$
because by choice of $v$ and $w$, all sets $S_i$ contain $v$, but none of them contain $w$.
Altogether, this gives $\ell\leq m-1$, as desired.
\end{proof}

Thus, by \cref{lem:chainSolution}, it is enough to find an optimal solution of a \CTC{} problem associated to a \CCTU{} problem such that the sets in the solution form a chain of (at most) $m-1$ cuts. This bounded number of cuts allows for a reduction to submodular minimization problems with congruency constraints of the following type.

\begin{mdframed}[innerleftmargin=0.5em, innertopmargin=0.5em, innerrightmargin=0.5em, innerbottommargin=0.5em, userdefinedwidth=0.95\linewidth, align=center]
{\textbf{Congruency-Constrained Submodular Minimization (\hypertarget{prb:CCSM}{CCSM}):}}
Given a submodular function $f\colon\mathcal{L} \rightarrow \mathbb{Z}$ defined on a lattice $\mathcal{L}\subseteq 2^N$, $\gamma\colon N\to\mathbb{Z}$, $m \in \mathbb{Z}_{>0}$, and $r\in \{0,\ldots, m-1\}$, find a minimizer of $\min\{f(C)\colon C\in\mathcal{L}, \gamma(C)\equiv r\pmod*{m}\}$.
\end{mdframed}

\noindent Such problems were studied by \textcite{nagele_2018_submodular}, where an algorithm for solving problems of this kind if $m$ is a constant prime power modulus was presented.
We remark that \cite{nagele_2018_submodular} studies a slightly less general setup than stated above, namely $\gamma \equiv 1$, where the constraint $\gamma(S)\equiv r\pmod*{m}$ translates to $|S|\equiv r\pmod{m}$.
For that case, algorithms with running time $|N|^{2m+O(1)}$ were presented.
However, the setting with general $\gamma$ can be readily reduced to that with $\gamma\equiv 1$ by replacing every element $v\in N$ by $t = (\gamma(v)\bmod{m})$ many elements $v_1,\ldots,v_t$ with $\gamma(v_i)=1$ and updating the lattice and the function correspondingly.
Observing that this reduction blows up the ground set by a factor of at most $m$, we thus get the following immediate generalization of Theorem 1.1 in \cite{nagele_2018_submodular}.

\begin{theorem}\label{thm:CCSM}
For any prime power $m\in\mathbb{Z}_{>0}$, \CCSM{} problems can be solved in $(m|N|)^{2m+O(1)}$ time.
\end{theorem}

It remains to discuss our reduction from \CTC{} problems to \CCSM{} problems.

\begin{lemma}\label{lem:CTCtoCCSM}
Consider a \CTC{} problem with constant modulus $m$. Finding a feasible solution of minimum cost among all solutions that consist of at most $m-1$ sets $S_1,\ldots,S_\ell$ with $S_\ell\subseteq S_{\ell-1}\subseteq\ldots\subseteq S_1$ can be strongly polynomially reduced to a \CCSM{} problem with modulus $m$, i.e., the \CCSM{} problem can be obtained in strongly polynomial time, and an optimal solution of that problem can be transformed to an optimal \CTC{} solution in strongly polynomial time.
\end{lemma}

\begin{proof}
Consider a \CTC{} instance with the usual notation. We construct a \CCSM{} instance on a ground set $N$ with a lattice $\mathcal{L}\subseteq 2^N$ whose sets correspond to feasible solutions of the relaxation of the given \CTC{} problem that have the desired chain structure. Moreover, we show that the function $f\colon \mathcal{L}\to\mathbb{Z}$ assigning to each set in $\mathcal{L}$ the value of the corresponding \CTC{} solution is a modular function. The last step will then be to observe that we can define a congruency constraint of the type appearing in \CCSM{} problems that is equivalent to the congruency constraint in the \CTC{} problem.

Let the ground set $N$ consist of $m-1$ copies of the vertex set $V$ of the tree in the \CTC{} instance, i.e., $N \coloneqq \bigcup_{i=1}^{m-1} V_i$, where $V_i = \{v_i\colon v\in V\}$.
Sets $C\subseteq N$ are in one-to-one correspondence with set families $S_1,\ldots,S_{\ell}\subseteq V$ that satisfy $\ell\leq m-1$ as follows: Given $C$, the corresponding set family is given by $S_i = \{v\in V\colon v_i\in C\}$, and vice versa, given a set family $S_1,\ldots,S_\ell$ with $\ell\leq m-1$, the corresponding subset of $N$ is $C=\bigcup_{i=1}^\ell \{v_i\colon v\in S_i\}$.
Now let us define a set $\mathcal{L}\subseteq 2^N$ such that $C\subseteq N$ is in $\mathcal{L}$ if and only if all three of the following are satisfied:
\begin{enumerate}
\item\label{item:lattice_chain} If $v_i\in C$ for some $v\in V$ and $i\in[m-1]$, then $v_j\in C$ for all $j\leq i$.
\item\label{item:lattice_inarcs} For every $(v,w)\in U$, if $w_i\in C$ for some $i\in[m-1]$, then $v_i\in C$.
\item\label{item:lattice_constraint} For every $(v,w)\in A$, if $v_i\in C$ and $i-b_a\geq 1$, then $w_{i-b_a}\in C$.
\end{enumerate}

\begin{claim}\label{claim:feasibleChain}
Sets $C\in\mathcal{L}$ are precisely those subsets of $N$ that correspond to set families $S_1,\ldots,S_\ell$ with $\ell\leq m-1$ that have chain structure $S_\ell\subseteq\ldots\subseteq S_1$ and are feasible solutions of the relaxation of the given \CTC{} problem.
\end{claim}

To see the claim, we start by observing that a set $C\subseteq N$ satisfies~\ref{item:lattice_chain} if and only if the corresponding sets $S_1,\ldots,S_\ell$ satisfy $S_i\subseteq S_j$ for all $i\geq j$: If $C$ satisfies~\ref{item:lattice_chain}, then $v\in S_i$, we get $v_i\in C$, which implies $v_j\in C$ because $i\geq j$, and thus $v\in S_j$. For the other way round, if $S_i\subseteq S_j$ for all $i\geq j$, then if $v_i\in C$ for some $v\in V$ and $i\in[m-1]$, we have $v\in S_i$, and thus for all $i\geq j$, it follows that $v\in S_j$, and thus $v_j\in C$.

Next,~\ref{item:lattice_inarcs} is satisfied by $C\subseteq N$ if and only if the corresponding sets $S_1,\ldots,S_\ell$ satisfy $\delta^-(S_i)=\emptyset$: $C$ does not satisfy~\ref{item:lattice_inarcs} if and only if there exist $(v,w)\in U$ and $i\in[m-1]$ such that $w_i\in C$, but $v_i\notin C$. The latter is equivalent to $w\in S_i$ and $v\notin S_i$, i.e., $\delta^-(S_i)\neq\emptyset$.

Finally, consider a set $C$ satisfying~\ref{item:lattice_chain} above. We show that the corresponding sets $S_1,\ldots,S_\ell$ then satisfy constraint~\ref{prbitem:CTC_coverage} of \CTC{} problems if and only if $C$ also satisfies~\ref{item:lattice_constraint} above.
To start with, note that by the previous arguments, we know that because $C$ satisfies~\ref{item:lattice_chain}, we have $S_\ell\subseteq\ldots\subseteq S_1$.
Consequently,
\begin{equation*}
|\{i\in[\ell]\colon v \in S_{i}\}|=\max\{i\in[m-1]\colon v_i\in C\} \qquad \forall v\in V\enspace,
\end{equation*}
hence a constraint of the form $|\{i\in[\ell]\colon v\in S_i\}| - |\{i\in[\ell]\colon w \in S_i\}| \leq b_a$ for some $a=(v,w)\in A$ is satisfied if and only if $\max\{i\in[m-1]\colon v_i\in C\} - \max\{i\in[m-1]\colon w_i\in C\} \leq b_a$, which in turn is guaranteed to hold if and only if $C$ satisfies~\ref{item:lattice_constraint} above, as desired. This proves \cref{claim:feasibleChain}.

\begin{claim}\label{claim:lattice}
$\mathcal{L}$ is a lattice.
\end{claim}

To prove this claim, we show that for any $C_1,C_2\in\mathcal{L}$, we also have $C_1\cap C_2\in \mathcal{L}$ and $C_1\cup C_2\in\mathcal{L}$.
We do so by showing that the intersection and union satisfy~\ref{item:lattice_chain} to~\ref{item:lattice_constraint} above.
Note that all three conditions are of the form ``If $a\in C$, then $b\in C$'', for different choices of $a,b\in N$.
It is generally true that if such conditions hold for two sets $C_1$ and $C_2$, then they also hold for $C_1\cap C_2$ and $C_1\cup C_2$:
If $a\in C_1\cap C_2$, then $a\in C_1$ and $a\in C_2$, hence also $b\in C_1$ and $b\in C_2$, and thus $b\in C_1\cap C_2$.
Also, if $a\in C_1\cup C_2$, then there is $\varepsilon\in\{0,1\}$ such that $a\in C_\varepsilon$, hence also $b\in C_\varepsilon$, and thus $b\in C_1\cup C_2$.
This proves \cref{claim:lattice}.

\medskip

As already indicated above, let $f\colon\mathcal{L}\to\mathbb{Z}$ be defined as follows: For $C\in\mathcal{L}$, if $S_1,\ldots,S_\ell$ is the corresponding solution of the relaxation of the \CTC{} problem, then $f(C)=\sum_{i=1}^{m-1} c(\delta^+(S_i))$.
In other words, $f$ assigns to each $C\in\mathcal{L}$ the objective value of the corresponding \CTC{} solution.

We claim that for any two sets $C,D\in \mathcal{L}$, we have $f(C)+f(D)=f(C\cap D)+f(C\cup D)$.
To this end, observe that if $S_1,\ldots, S_\ell\subseteq V$ correspond to $C$ and $T_1,\ldots,T_\ell'\subseteq V$ correspond to $D$, we may introduce $S_{\ell+1}=\ldots=S_{m-1}=\emptyset$ and $T_{\ell'+1}=\ldots=T_{m-1}=\emptyset$ and then obtain
\begin{align*}
f(C)+f(D) &= \sum_{i=1}^{m-1} c(\delta^+(S_i)) + c(\delta^+(T_i)) = \sum_{i=1}^{m-1} c(\delta^+(S_i\cap T_i)) + c(\delta^+(S_i\cup T_i))\\
&= f(C\cap D) + f(C\cup D)\enspace,
\end{align*}
where the middle inequality exploits that $\chi^{\delta^+(S_i)} + \chi^{\delta^+(T_i)} = \chi^{\delta^+(S_i\cap T_i)} + \chi^{\delta^+(S_i\cup T_i)}$, which holds because $\delta^-(S_i)=\delta^-(T_i)=\emptyset$ for $i\in[m-1]$ due to the fact that $S_1,\ldots,S_\ell$ and $T_1,\ldots,T_{\ell'}$ are feasible solutions for the relaxation of the \CTC{} problem and thus satisfy constraint~\ref{prbitem:CTC_inedges} of that problem type.

Finally, define $\gamma\colon C\to\mathbb{Z}$ by $\gamma(v_i) = \alpha(v)$ for all $v\in V$ and $i\in[m-1]$.
This implies that for any $C\in \mathcal{L}$ and a corresponding solution $S_1,\ldots,S_\ell$ of the \CTC{} problem's relaxation,
$$
\gamma(C) = \sum_{i=1}^{m-1} \gamma(C\cap S_i) = \sum_{i=1}^{\ell} \alpha(S_i)\enspace,
$$
and hence $\gamma(C)\equiv r\pmod*{m}$ if and only if $\sum_{i=1}^\ell \alpha(S_i)\equiv r\pmod*{m}$.

Altogether, we obtain that $C$ is an optimal solution of the \CCSM{} problem given by $N$, $\mathcal{L}$, $f$ and $\gamma$ if and only if the corresponding sets $S_1,\ldots,S_\ell$ form an optimal solution of the \CTC{} problem with chain structure $S_\ell\subseteq\ldots\subseteq S_1$ and $\ell\leq m-1$.
Observing that the \CCSM{} problem can be obtained from the \CTC{} problem in strongly polynomial time (recall that $m$ is assumed to be a constant), and that transforming a \CCSM{} solution to a \CTC{} solution is immediate, finishes the proof.
\end{proof}

Finally, combining \cref{cor:CCTUtoCTC,lem:chainSolution,lem:CTCtoCCSM}, we can conclude \cref{thm:solveTranspNetwPP}.

\begin{proof}[Proof of \cref{thm:solveTranspNetwPP}]
Given a \CCTU{} problem, by \cref{cor:CCTUtoCTC} it is enough to solve the associated \CTC{} problem.
By \cref{lem:CTCsmallChain}, this problem has an optimal solution with chain structure and at most $m-1$ cuts---which is precisely the type of problem that can be strongly polynomially reduced to a congruency-constrained submodular minimization problem by \cref{lem:CTCtoCCSM}.
Note that in these reductions, the modulus $m$ of the involved congruency-constraints is invariant, and $m$ is a constant prime power by assumption. Hence, the final congruency-constrained submodular minimization problem is one with constant prime power modulus.
Such problems can be solved in strongly polynomial time by~\cref{thm:CCSM}.
\end{proof}

\subsection{Matrices stemming from particular constant-size matrices\MOORdot}\label{sec:constCore}

To complete the study of base block \CCTU{} problems, we now cover \CCTU{} problems with constraint matrices that fall into case~\ref{thmitem:TUdecomp_const} of \cref{thm:TUdecomp}.
In other words, we study matrices that can be obtained from the two matrices
\begin{equation}\label{eq:specialMatrices}
\begin{pmatrix*}[r]
 1 & -1 &  0 &  0 & -1 \\
-1 &  1 & -1 &  0 &  0 \\
 0 & -1 &  1 & -1 &  0 \\
 0 &  0 & -1 &  1 & -1 \\
-1 &  0 &  0 & -1 &  1
\end{pmatrix*}
\quad\text{and}\quad
\begin{pmatrix}
 1 &  1 &  1 &  1 &  1 \\
 1 &  1 &  1 &  0 &  0 \\
 1 &  0 &  1 &  1 &  0 \\
 1 &  0 &  0 &  1 &  1 \\
 1 &  1 &  0 &  0 &  1
\end{pmatrix}
\end{equation}
by repeatedly appending unit vector rows or columns, appending a copy of a row or column, and inverting the sign of a row or column.
More generally, our arguments apply to any constraint matrices that can be obtained from constant-size matrices by repeatedly applying the aforementioned operations.
More formally, let us introduce the following notion of a \emph{core} of a totally unimodular matrix.

\begin{definition}
Let $T$ be a totally unimodular matrix. A submatrix of $T$ is a \emph{core} of $T$ if it is a smallest possible submatrix of $T$ that can be obtained by iteratively deleting
\begin{enumerate}
\item any row or column with at most one non-zero entry, or
\item any row or column appearing twice or whose negation is also in the matrix.
\end{enumerate}
\end{definition}

It can be observed that up to row and column permutations and sign changes of rows and columns, every totally unimodular matrix has a unique core, which we denote by $\operatorname{core}(T)$.
Still, let us remark that we do not need uniqueness for our arguments and working with \emph{any} core would be enough for us.
In the context of \CCTU{} problems, we show the following theorem.

\begin{theorem}\label{thm:constCore}
\CCTU{} problems with modulus $m$ and a constraint matrix $T$ that has a core of constant size can be solved in strongly polynomial time
\begin{enumerate}
\item by a randomized algorithm if the objective is unary encoded and $m$ is constant, or
\item by a deterministic algorithm if $m$ is a constant prime power.
\end{enumerate}
\end{theorem}

In particular, \cref{thm:constCore} shows that \CCTU{} problems with constant prime power modulus and constraint matrices that fall into case~\ref{thmitem:TUdecomp_const} of \cref{thm:TUdecomp} can be solved in strongly polynomial time.
\cref{thm:constCore} immediately follows from the following more concrete lemma by solving each of the $m^{O(\ell)}$ many \CCTU{} problems using \cref{thm:solveNetwRPP} or \cref{thm:solveTranspNetwPP}.

\begin{lemma}\label{lem:reductionConstCore}
Consider a \CCTU{} problem with modulus $m$ and constraint matrix $T$, and let $\ell$ be the number of columns of $\operatorname{core}(T)$. The \CCTU{} problem can be reduced to $m^{O(\ell)}$ many \CCTU{} problems, with constraint matrices of size linear in the size of $T$, that are network matrices and transposes of network matrices at the same time.
\end{lemma}

\begin{proof}
Assume that we are given a normalized \CCTU{} problem, which has the form
\begin{equation*}
\min\left\{c^\top x \colon Tx \leq b,\ \gamma^\top x \equiv r \pmod*{m},\ x\in\mathbb{Z}^n_{\geq 0} \right\}\enspace,
\end{equation*}
where $T$ is a matrix that is obtained as follows: Start from the matrix $C=\operatorname{core}(T)$ that has $\ell$ many columns, and repeatedly append unit rows or columns, append a copy of a row or column, and invert the sign of a row or column.
In this process, we say that a row or column \emph{stems} from $C$ if it either is a row or column of $C$, or it was obtained by copying a row or column that stems from $C$.
Thus, we may rewrite the inequality system in the form
\begin{equation}\label{eq:partitionedSystem}
\begin{pmatrix}
T^{11} & T^{12} \\ T^{21} & T^{22}
\end{pmatrix}\begin{pmatrix}
x^{1} \\ x^{2}
\end{pmatrix} \leq \begin{pmatrix}
b^{1} \\
b^{2}
\end{pmatrix}\enspace,
\end{equation}
where $T^{11}$ comprises the rows and columns of $T$ that stem from $C$, and the remaining matrix as well as the variables $x$ and the right-hand side $b$ are split accordingly.
Note that while $T$ is achieved as a construction starting from the TU matrix $C$, we could also start from the totally unimodular matrix obtained from $C$ by appending a $\ell\times \ell$ identity matrix, and then perform the same operations to obtain a totally unimodular matrix of the form
\begin{equation}\label{eq:extendedPartitionedMatrix}
\begin{pmatrix}
S & 0 \\ T^{11} & T^{12} \\ T^{21} & T^{22}
\end{pmatrix}\enspace.
\end{equation}
Here $S$ has $\ell$ many rows $s_i^\top$ for $i\in[\ell]$ where, without loss of generality, the support of $s_i^\top$ comprises precisely those columns that stem from column $i$ of $C$ in the construction.
Our approach is to guess the $\ell$ many scalar products $s_i^\top x^1$ of an optimal solution $x^*=(x^1\enspace x^2)$, and thereby reduce the problem to an easier one.

To this end, note that the rows $(s_i^\top\enspace 0)$ are TU-appendable to the constraint matrix $T$ because the matrix in~\eqref{eq:extendedPartitionedMatrix} is TU. Thus, because we work with a normalized problem, we know that there exists an optimal solution $x^* = (x^1\enspace x^2)$ of the \CCTU{} problem such that $s_i^\top x^{1} \in\{-m+1, \ldots, m-1\}$ (see \cref{thm:proximityGeneral}).
Consequently, it is enough to consider $(2m-1)^{\ell}$ many combinations of values that these scalar products may admit.
Once we fix those values, we also know the value of $T^{11}x^{1}$:
Indeed, it is easy to see that every row $t$ of $T^{11}$ is a linear combination of the rows $s_i$, and hence $t^\top x$ is a linear combination of $s_i^\top x$.
Thus, for any guess $\sigma = (\sigma_1,\ldots, \sigma_{\ell})$ of the $\ell$ many scalar products $s_1^{\top} x^1,\dots, s_{\ell}^{\top} x^1$, and after computing $\tau = T^{11}x^1$, we may rewrite the system~\eqref{eq:partitionedSystem} in the form
\begin{align}\label{eq:systemAfterGuess}
\begin{pmatrix}
S & 0 \\ -S & 0 \\
0 & T^{12} \\ T^{21} & T^{22}
\end{pmatrix}\begin{pmatrix}
x^{1} \\ x^{2}
\end{pmatrix} \leq \begin{pmatrix}
\sigma \\
-\sigma \\
b^{1} - \tau \\
b^{2}
\end{pmatrix}\enspace.
\end{align}
We claim that the new constraint matrix is a network matrix and the transpose of a network matrix at the same time. To this end, observe that the matrix
\begin{equation}\label{eq:matReplacedWithZeros}
\begin{pmatrix}
S & 0 \\0 & T^{12} \\ T^{21} & T^{22}
\end{pmatrix}
\end{equation}
can be obtained by performing the same steps as we perform to obtain the matrix in~\eqref{eq:extendedPartitionedMatrix}, but replacing the entries of $C$ with zeros in the starting matrix. This makes the starting matrix being a network matrix and the transpose of a network matrix at the same time, and this property is invariant under the operations that we perform when constructing the matrix. Thus, the matrix in~\eqref{eq:matReplacedWithZeros} is a network matrix and the transpose of a network matrix at the same time, and hence, so is the constraint matrix in~\eqref{eq:systemAfterGuess}.

To sum up, we reduce a \CCTU{} problem with a constraint matrix that has core with $\ell$ columns, to $(2m-1)^\ell$ many \CCTU{} problems with constraint matrices that are a network matrix and the transpose of a network matrix at the same time. Also note that the size of the new constraint matrix is linear in the size of the original constraint matrix. This proves the lemma.
\end{proof}

We remark that instead of guessing all $\ell$ many scalar products in the proof of \cref{lem:reductionConstCore}, we could also guess all but four of them: This would guarantee that the resulting constraint matrix of the reduced problems has a core that consists of at most $4$ rows, and hence does not fall into case~\ref{thmitem:TUdecomp_const} of \cref{thm:TUdecomp}, and we can fall back to another case for solving the reduced problems.
In particular, when applying \cref{lem:reductionConstCore} to a constraint matrix $T$ falling into case~\ref{thmitem:TUdecomp_const} of \cref{thm:TUdecomp}, guessing the scalar product of a single row would be enough.

\section[Further details of our approach to \texorpdfstring{$R$}{R}-CCTUF problems]{\boldmath Further details of our approach to \RCCTUF{} problems\MOORdot\unboldmath}\label{sec:patterns}

In this section, we fill in details and formal proofs supplementing the overview of our approach to \RCCTUF{} problems given in \cref{sec:overview}.

\subsection{Seymour's decomposition of TU matrices\MOORdot}\label{sec:seymour}

\cref{thm:TUdecomp} is, up to the constraints $n_A, n_B\geq 2$, one naturally equivalent way of stating Seymour's decomposition theorem for TU matrices (see, for example, \cite{seymour_1980} or \cite{schrijver1998theory}).
The version presented in \cref{thm:TUdecomp} is a variation thereof that additionally guarantees lower bounds on the number of rows $n_A$ and $n_B$ of the blocks $A$ and $B$, respectively, obtained in $1$-, $2$-, and $3$-sums, namely $n_A, n_B\geq 2$.
Similar bounds were achieved by Artmann, Weismantel, and Zenklusen in~\cite{artmann_2017_strongly}: They lower bound the number of \emph{rows} $k_A$ and $k_B$ of the two blocks $A$ and $B$ by $2$---hence applying their theorem to the transpose of a TU matrix gives the version that we need.

Although not exploited in our results, we remark that the method presented in~\cite{artmann_2017_strongly} in fact allows for obtaining the lower bounds on the number of columns and the number of rows of $A$ and $B$ simultaneously, i.e., it can be guaranteed that in any $1$-, $2$-, and $3$-sum, both matrices are at least $2\times 2$ matrices.

\subsection{Patterns\MOORdot}

Recall that if the constraint matrix of the \RCCTUF{} problem that we consider is a $1$-, $2$-, or $3$-sum, the problem can be written in the form
\begin{equation*}
\begin{aligned}
\begin{pmatrix}
A & ef^\top \\ gh^\top & B
\end{pmatrix} \cdot \begin{pmatrix}
x_A \\ x_B
\end{pmatrix} &\leq \begin{pmatrix}
b_A \\ b_B
\end{pmatrix} \\
\gamma_A^\top x_A + \gamma_B^\top x_B & \in R \pmod{m} \\
x_A \in \mathbb{Z}^{n_A},\
x_B & \in \mathbb{Z}^{n_B}\enspace,
\end{aligned}
\end{equation*}
as also given in~\eqref{eq:structured-problem}. After fixing $\alpha=f^\top x_B$ and $\beta=h^\top x_A$, the above problem splits into an $A$-problem and a $B$-problem as in~\eqref{eq:A-B-problem} (whose only link is through the original congruency constraint, which translates into $r_A+r_B\in R$).
Also recall that we let $\Pi\subseteq\mathbb{Z}^2$ denote all pairs $(\alpha, \beta)$ for which both the $A$-problem and the $B$-problem are feasible, and that by \cref{lem:patternsBounds} we know that if the initial problem is feasible, then it is also feasible for a pair of scalar products $(\alpha,\beta)\in\Pi$ that additionally satisfy
$\ell_0 \leq \alpha+\beta \leq u_0$, $\ell_1 \leq \alpha \leq u_1$, $\ell_2\leq \beta \leq u_2$ for bounds $\ell_i, u_i$ that we can determine in strongly polynomial time, and that satisfy $u_i-\ell_i\leq m-|R|$. For this reason, we defined a narrowed down version of $\Pi$, namely
\begin{equation}\label{eq:PiNarrowed}
\Pi_{\text{narrowed}} \coloneqq \Pi \cap \{(\alpha,\beta)\in\mathbb{Z}^2\colon \ell_0 \leq \alpha+\beta \leq u_0,\ \ell_1 \leq \alpha \leq u_1,\ \ell_2\leq \beta \leq u_2\}\enspace,
\end{equation}
and only look for solutions with scalar products $(\alpha,\beta)\in\Pi_{\text{narrowed}}$.
We also remind the reader that a narrowed pattern associated to the problem is given by $\pi\colon \Pi_{\text{narrowed}}\to 2^{\{0, \ldots, m-1\}}$, where $\pi(\alpha, \beta)$ is the set of residues $r_B\in\{0, \ldots, m-1\}$ for which the $B$-problem is feasible.

\subsubsection*{The shape of pattern supports\MOORdot}
In what follows, we prove the following lemma on the shape of $\Pi_{\text{narrowed}}$.

\begin{lemma}\label{lem:narrowedPiIneqs}
In the above setup, we can in strongly polynomial time determine $\ell_i', u_i'$ for $i\in\{0,1,2\}$ with $u_i'-\ell_i'\leq m-|R|$ such that
$$
\Pi_{\text{narrowed}} = \{(\alpha,\beta)\in\mathbb{Z}^2\colon \ell_0' \leq \alpha+\beta \leq u_0',\ \ell_1' \leq \alpha \leq u_1',\ \ell_2'\leq \beta \leq u_2'\}\enspace.
$$
\end{lemma}

We emphasize that the main contribution of \cref{lem:narrowedPiIneqs} is not to find new bounds $\ell_i', u_i'$ (they will simply be the tightest bounds such that $\Pi_{\text{narrowed}}$ is contained in the resulting set), but that there are no holes within the shape given by the bounds. That is, there are no $(\alpha,\beta)$ satisfying the bounds, but such that there is no feasible solution of our \RCCTUF{} problem with scalar products $(\alpha,\beta)$. It turns out that $\Pi$ has the same property in the following sense, and \cref{lem:narrowedPiIneqs} will follow from that.

\begin{lemma}\label{lem:feasibleResiduePolyhedron}
For $\Pi\subseteq\mathbb{Z}^2$ defined as above, $\operatorname{conv}(\Pi)$ is a polyhedron with $\Pi = \operatorname{conv}(\Pi) \cap \mathbb{Z}^2$ and edge directions in $\mathcal{D}\coloneqq\{\pm\begin{psmallmatrix}
1 \\ 0
\end{psmallmatrix}, \pm\begin{psmallmatrix}
0 \\ 1
\end{psmallmatrix}, \pm\begin{psmallmatrix}
1 \\ -1
\end{psmallmatrix}\}$. Hence, there is an inequality description of $\operatorname{conv}(\Pi)$ only consisting of upper and/or lower bounds on $\alpha$, $\beta$, and $\alpha+\beta$.
\end{lemma}

We remark that when we refer to \emph{edge directions} $v$ of an integral polyhedron (with rational extremal rays in case of unboundedness), then we always choose $v$ to be integral, i.e., $v\in \mathbb{Z}^n$, and such that the greatest common divisor of its coordinates is $1$.
In other words, a vector $v\in\mathbb{Z}^n$ is an \emph{edge direction} of an integral polyhedron if there exist integral points $x_1$ and $x_2$ that lie on the same edge of $P$ such that $x_1=x_2+v$, and the greatest common divisor of all components of $v$ is $1$.

\begin{proof}[Proof of \cref{lem:narrowedPiIneqs}]
From \cref{lem:feasibleResiduePolyhedron} and \eqref{eq:PiNarrowed}, it follows immediately that \cref{lem:narrowedPiIneqs} holds for
\begin{gather*}
\ell_0'=\min\{\alpha+\beta\colon(\alpha,\beta)\in\Pi_{\text{narrowed}}\} \quad\text{and}\quad u_0'=\max\{\alpha+\beta\colon(\alpha,\beta)\in\Pi_{\text{narrowed}}\}\enspace,\\
\ell_1'=\min\{\alpha\colon(\alpha,\beta)\in\Pi_{\text{narrowed}}\} \quad\text{and}\quad u_1'=\max\{\alpha\colon(\alpha,\beta)\in\Pi_{\text{narrowed}}\}\enspace, \quad\text{and}\\
\ell_2'=\min\{\beta\colon(\alpha,\beta)\in\Pi_{\text{narrowed}}\} \quad\text{and}\quad u_2'=\max\{\beta\colon(\alpha,\beta)\in\Pi_{\text{narrowed}}\}\enspace.
\end{gather*}
To see that we can determine $\ell_i'$ and $u_i'$ in strongly polynomial time, we exploit that by \cref{obs:simultTUappend},
\begin{align*}
\begin{array}{rcrcrcl}
	&      &             A x_A & + &   e f^\top x_B & \leq   & b_A \\
	&      &      g h^\top x_A & + &          B x_B & \leq   & b_B \\
\ell_0   & \leq & h^\top x_A & + & f^\top x_B & \leq   & u_0 \\
\ell_1 & \leq & & & f^\top x_B & \leq   & u_1 \\
\ell_2 & \leq & h^\top x_A & & & \leq   & u_2
\end{array}
\end{align*}
is an inequality system with a totally unimodular constraint matrix. Here, the last three constraints precisely encode the constraints $(\alpha=f^\top x_B,\beta=h^\top x_A)\in\Pi_{\text{narrowed}}$, so pairs in $\Pi_{\text{narrowed}}$ correspond to feasible solutions of the above system, and vice versa.
Due to total unimodularity, we can find integral solutions of this system minimizing or maximizing the linear functions $\alpha=f^\top x_B$, $\beta=h^\top x_A$, and $\alpha+\beta=f^\top x_B + h^\top x_A$ by solving the corresponding relaxations using the approach of Tardos~\cite{tardos_1986_strongly} in strongly polynomial time, and the corresponding optimal values are precisely the values $\ell_i'$ and $u_i'$ for $i\in\{0,1,2\}$ that we are looking for, and we have $u_i'-\ell_i'\leq u_i-\ell_i\leq m-|R|$ for all $i\in\{0,1,2\}$.
\end{proof}

To prove \cref{lem:feasibleResiduePolyhedron}, we will observe that $\Pi$ can be seen to essentially be a projection of the set of feasible solutions of the relaxation of the initial \RCCTUF{} problem. The following result will provide the necessary properties to conclude \cref{lem:feasibleResiduePolyhedron}.

\begin{theorem}\label{thm:TUprojectionEdges}
Let $T\in\{-1,0,1\}^{k\times n}$ be a totally unimodular matrix, let $b\in\mathbb{Z}^n$, and let $I\subseteq[n]$ be a subset of the column indices. Then, the axis-parallel projection $Q\subseteq \mathbb{R}^{I}$ of $P\coloneqq \{x\in\mathbb{R}^n\colon Tx\leq b\}$ on the variables $(x_i)_{i\in I}$ has the following property: For any edge direction $v\in\mathbb{Z}^I$ of $Q$, and any $w\in\mathbb{Z}^n$ that is TU-appendable to $T$ and supported on $I$, we have $w_I^\top v\in\{-1,0,1\}$.
\end{theorem}
Here, for a vector $w\in\mathbb{R}^n$ and a subset $I\subseteq [n]$, we denote by $w_I$ the restriction of $w$ to the coordinate indices in $I$.
Generally, note that for $I=[n]$, \cref{thm:TUprojectionEdges} is a statement about edge directions of polyhedra that are defined by totally unimodular matrices, characterized in terms of TU-appendable vectors.
This shows another use of the concept of TU-appendable vectors and gives a result that might find independent applications.

\begin{proof}[Proof of \cref{thm:TUprojectionEdges}]
Assume for the sake of deriving a contradiction that $Q$ has an edge direction $v\in\mathbb{R}^I$ such that there exists a vector $w\in\mathbb{Z}^n$ that is supported on $I$ and TU-appendable to $T$ such that $w_I^\top v\notin\{-1,0,1\}$.
Let $x_1,x_2\in\mathbb{Z}^I$ lie on an edge of $Q$ such that $x_1=x_2+v$, and observe that there exists $\lambda\in(0,1)$ such that $y\coloneqq (1-\lambda)x_1+\lambda x_2=x_1+\lambda v$ is not integral, but satisfies $w_I^\top y = \eta$ for some $\eta\in\mathbb{Z}$, for example $\lambda = \sfrac{1}{|w_I^\top v|}$.

Now let $\overline y$ be a preimage of $y$ under the axis-parallel projection from $P$ to $Q$, and observe that $\overline y$ is a fractional solution of the system
\begin{equation*}
\begin{array}{rcl}
T x & \leq & b \\
w^\top x & = & \eta
\end{array}\enspace,
\end{equation*}
which has a totally unimodular constraint matrix and integral right-hand sides, which implies that it describes an integral polyhedron.
Thus, $\overline y$ can be written as a convex combination $\overline y = \sum_{i=1}^k \lambda_i \overline z_i$ of integral vectors $\overline z_i\in \{x\in \mathbb{Z}^n\colon Tx \leq b, w^{\top} x = \eta\}$, with coefficients $\lambda_i\in(0,1)$ such that $\sum_{i=1}^k \lambda_i=1$.
Let $z_i\in\mathbb{Z}^I$ be obtained from $\overline z_i$ through axis-parallel projection to $\mathbb{R}^I$.
Hence, $z_i\in Q\cap \mathbb{Z}^I$, $w_I^\top z_i = \eta$, and $y = \sum_{i=1}^k \lambda_i z_i$.
Thus, we expressed $y$ as a convex combination of points $z_i\in Q$.
But recall that $y$ lies on an edge of $Q$, and the only way to express such a point as a convex combination of others with non-zero coefficients is to use points from the same edge only, hence all $z_i$ lie on the same edge.
However, as the edge direction is $v$ and $w_I^\top v\neq 0$, the point $y$ is the only point on the edge satisfying $w_I^\top x = \eta$, so we must have $z_i=y$ for all $i\in[k]$.
This contradicts that $z_i$ are integral, while $y$ is not.
Thus, our assumption was wrong and \cref{thm:TUprojectionEdges} follows.
\end{proof}

\begin{proof}[Proof of \cref{lem:feasibleResiduePolyhedron}]
Note that $\Pi$ contains precisely those pairs $(\alpha,\beta)\in\mathbb{Z}^2$ for which there exist $(x_A, x_B)\in\mathbb{Z}^{n_A}\times\mathbb{Z}^{n_B}$ such that $(x_A, x_B, \alpha, \beta)$ is a solution of the system
\begin{align}\label{eq:extendedSystem}
\begin{pmatrix}
A & ef^\top & 0 & 0 \\
gh^\top & B & 0 & 0 \\
0 & f^\top & -1 & 0 \\
h^\top & 0 & 0 & -1
\end{pmatrix}\begin{pmatrix}
x_A \\ x_B \\ \alpha \\ \beta
\end{pmatrix} \enspace \begin{matrix}
\leq \\ \leq \\ = \\ =
\end{matrix} \enspace \begin{pmatrix}
b_A \\ b_B \\ 0 \\ 0
\end{pmatrix}\enspace.
\end{align}
Let $P$ be the polyhedron defined by~\eqref{eq:extendedSystem}, and let $T$ be the constraint matrix in~\eqref{eq:extendedSystem}.
Observe that $T$ is totally unimodular by \cref{obs:simultTUappend}.
This has several implications:
First, $Q\coloneqq\operatorname{conv}(\Pi)$ is precisely the projection of $P$ to the variables $(\alpha, \beta)$.
Moreover, every integral point in this projection has an integral inverse image, hence $\Pi=\operatorname{conv}(\Pi)\cap\mathbb{Z}^2$.
Finally, by \cref{thm:TUprojectionEdges} applied with $I$ containing the indices of the variables $\alpha$ and $\beta$, we obtain that all edge directions $v\in\mathbb{Z}^2$ of $Q$ satisfy that for any integral vector $w$ that is TU-appendable to $T$ with support on the last two columns only, we have $w_I^\top v\in \{-1,0,1\}$.
Obviously, the unit vectors $w\in \{\pm e_\alpha, \pm e_\beta\}$ (i.e., the vectors that are all zero except for $\pm1$ entries in corresponding to the variables $\alpha$ and $\beta$, and hence correspond to $w_I\in\{\pm\begin{psmallmatrix}1 \\ 0
\end{psmallmatrix}, \pm\begin{psmallmatrix}0 \\ 1\end{psmallmatrix}\}$, respectively) are TU-appendable, and we claim that $w =\pm(e_\alpha+e_\beta)$ (corresponding to $w_I=\pm\begin{psmallmatrix}
1 \\ 1
\end{psmallmatrix}$) are, as well.

Assuming this claim, the conclusion is immediate: We know that all edge directions $v\in\mathbb{Z}^2$ of $Q$ are such that $w_I^\top v\in\{-1,0,1\}$ for all $w_I\in\{\pm\begin{psmallmatrix}1 \\ 0
\end{psmallmatrix}, \pm\begin{psmallmatrix}0 \\ 1\end{psmallmatrix}, \pm\begin{psmallmatrix}1 \\ 1\end{psmallmatrix}\}$.
This leaves $v\in\{\pm\begin{psmallmatrix}1 \\ 0
\end{psmallmatrix}, \pm\begin{psmallmatrix}0 \\ 1\end{psmallmatrix}, \pm\begin{psmallmatrix}1 \\ -1\end{psmallmatrix}\}$ as the only possible feasible edge directions, and hence the polyhedron $Q$ can be described by inequalities bounding $\alpha$, $\beta$, and $\alpha+\beta$ from above and/or below, as claimed by \cref{lem:feasibleResiduePolyhedron}.
Thus, we conclude the proof by showing the claim. To this end, define the three matrices
\begin{align*}
T' \coloneqq \begin{pmatrix}
A & ef^\top & 0 & 0 \\
gh^\top & B & 0 & 0 \\
0 & f^\top & -1 & 0 \\
h^\top & 0 & 0 & -1 \\
0 & 0 & 1 & 1
\end{pmatrix},\quad
T'' \coloneqq \begin{pmatrix}
A & ef^\top \\ gh^\top & B \\ 0 & f^\top \\ h^\top & 0
\end{pmatrix},
\quad\text{and}\quad
T'''\coloneqq\begin{pmatrix}
A & ef^\top & 0 & 0 \\
gh^\top & B & 0 & 0 \\
0 & f^\top & -1 & 0 \\
h^\top & 0 & 0 & -1 \\
h^\top & f^\top & 0 & 0
\end{pmatrix}.
\end{align*}
The matrices $T''$ and $T'''$ are auxiliary matrices we use in the following.
To show the claim we need to show that $T'$ is totally unimodular.
Indeed, this will show TU-appendability of both $\begin{psmallmatrix}1 \\ 1\end{psmallmatrix}$ and $\begin{psmallmatrix}-1 \\ -1\end{psmallmatrix}$ because changing the sign of a row preserve total unimodularity of a matrix.
To this end, consider any square submatrix $S=T'_{IJ}$ of $T'$, for two index subsets $I$ and $J$.
If $I$ does not contain all of the last three rows, we can perform a Laplace expansion of the determinant along unit rows and columns, which will suffice to get rid of the last two columns and the last row of $T'$ (if they are present in $S$), and get that the determinant of $S$ equals the determinant of a square submatrix of $T''$ in absolute value.
But $T''$ is totally unimodular due to \cref{obs:simultTUappend}, and hence the determinant of the submatrix that we are considering is in $\{-1,0,1\}$.
If, on the other hand, $S=T'_{IJ}$ contains all of the last three rows, we know that its determinant is equal to the determinant of the submatrix of $S'=T'''_{IJ}$, where $T'''$ is obtained from $T'$ by adding the penultimate and third to last row to the last one.
This operation does not change determinants, i.e., $\det(S) = \det(S')$.
But $T'''$ is totally unimodular by \cref{obs:simultTUappend}, and hence $\det(S')\in\{-1,0,1\}$.
In both cases, we obtain $\det(S)\in\{-1,0,1\}$, so $T'$ is totally unimodular.
\end{proof}

\subsubsection*{An averaging lemma and linear patterns\MOORdot}

In the proof of \cref{thm:TUprojectionEdges}, one key idea was to average two integral solutions $x_1$ and $x_2$ to obtain a fractional solution that has an integral scalar product with some TU-appendable vector $w$, and then decompose that fractional solution into other feasible vectors that have the same integral scalar product with $w$.
This idea can also be exploited to obtain the following result.
Here, for an \RCCTUF{} problem of the form given in~\eqref{eq:structured-problem} (or its relaxation), we say that an \RCCTUF{} solution (or solution to its relaxation) $x=(x_A, x_B)\in\mathbb{Z}^{n_A}\times\mathbb{Z}^{n_B}$ is a solution for $(\alpha, \beta)\in\mathbb{Z}^2$ if $f^\top x_B = \alpha$ and $h^\top x_A=\beta$.

\begin{lemma}[Averaging Lemma]\label{lem:averagingPatterns}
Consider the relaxation of an \RCCTUF{} problem of the form given in~\eqref{eq:structured-problem}.
Let $x^1$ and $x^2$ be solutions for $(\alpha_1, \beta_1)$ and $(\alpha_2, \beta_2)$, respectively.
Then, there exist solutions $x^3$ and $x^4$ for $(\alpha_3, \beta_3)$ and $(\alpha_4, \beta_4)$, respectively, such that $x^1+x^2=x^3+x^4$, as well as
\begin{equation}\label{eq:scalarProductBounds}
\begin{gathered}
\floor*{\frac{\alpha_1 + \alpha_2}{2}} \leq \alpha_3, \alpha_4 \leq \ceil*{\frac{\alpha_1 + \alpha_2}{2}} \enspace, \qquad
\floor*{\frac{\beta_1 + \beta_2}{2}} \leq \beta_3, \beta_4 \leq \ceil*{\frac{\beta_1 + \beta_2}{2}}\enspace,\quad\text{and} \\
\floor*{\frac{\alpha_1 + \beta_1 + \alpha_2 + \beta_2}{2}} \leq \alpha_3 + \beta_3, \alpha_4 + \beta_4 \leq \ceil*{\frac{\alpha_1 + \beta_1 + \alpha_2 + \beta_2}{2}}\enspace.
\end{gathered}
\end{equation}
\end{lemma}

\begin{proof}
Consider the linear inequality system
\begin{equation}\label{eq:boundedSystem}
\renewcommand{\arraystretch}{1.1}
\begin{array}{rcrcrcl}
&      &             A x_A & + &   e f^\top x_B & \leq   & b_A \\
&      &      g h^\top x_A & + &          B x_B & \leq   & b_B \\
\floor*{\frac12(\alpha_1 + \beta_1 + \alpha_2 + \beta_2)}   & \leq & h^\top x_A & + & f^\top x_B & \leq   & \ceil*{\frac12(\alpha_1 + \beta_1 + \alpha_2 + \beta_2)} \\
\floor*{\frac12(\alpha_1 + \alpha_2)} & \leq & & & f^\top x_B & \leq   & \ceil*{\frac12(\alpha_1 + \alpha_2)} \\
\floor*{\frac12(\beta_1 + \beta_2)} & \leq & h^\top x_A & & & \leq   & \ceil*{\frac12(\beta_1 + \beta_2)}
\end{array}
\renewcommand{\arraystretch}{1.0}
\end{equation}and note that the claim of the lemma is that this system has two integral solutions $x^3$ and $x^4$ with $x^1+x^2=x^3+x^4$.
To find these solutions, let $T$ and $q$ be such that $Tx\leq q$ is the system~\eqref{eq:boundedSystem}, and observe that $T$ is totally unimodular by \cref{obs:simultTUappend}.
Then, the system
\begin{equation}\label{eq:twoSidedSystem}
\left\{\enspace\begin{aligned}
Tx & \leq q \\
T(x^1 + x^2 - x) & \leq q
\end{aligned}\right.
\end{equation}
also has a totally unimodular constraint matrix, and $z := \frac12(x^1 + x^2)$ is a (potentially fractional) solution of it.
Because the bounds in the inequality constraints are all integral, we conclude that the linear system in~\eqref{eq:twoSidedSystem} also has an integral solution $x^3$.
Additionally, by symmetry it is immediate that $x^4\coloneqq x^1 + x^2 - x^3$ is another integral solution.
In particular, we thus found $x^3$ and $x^4$ that are feasible for~\eqref{eq:boundedSystem}, and they satisfy $x^1+x^2 = x^3 + x^4$, as desired.
\end{proof}

The above averaging lemma gives us a tool to analyze (narrowed) patterns $\pi\colon \Pi_{\text{narrowed}}\to 2^{\{0,\ldots,m-1\}}$, because if the difference of $(\alpha_1,\beta_1)$ and $(\alpha_2,\beta_2)$ is large enough, the inequalities in \cref{lem:averagingPatterns} will make sure that $(\alpha_3,\beta_3)$ and $(\alpha_4,\beta_4)$ are different from $(\alpha_1,\beta_1)$ and $(\alpha_2,\beta_2)$, and hence also the solutions $x^3$ and $x^4$ are different from $x^1$ and $x^2$.
Still, the relation $x^1 + x^2 = x^3 + x^4$ allows us to draw conclusions about feasible residues in $\pi(\alpha_3, \beta_3)$ and $\pi(\alpha_4,\beta_4)$, and in particular relate them to residues in $\pi(\alpha_1, \beta_1)$ and $\pi(\alpha_2,\beta_2)$.
We start by applying these ideas to narrowed patterns $\pi$ that satisfy $|\pi(\alpha,\beta)|=1$ for all $(\alpha,\beta)\in\Pi_{\text{narrowed}}$.
Again, we use the notation
\begin{equation*}
\mathcal{D}\coloneqq\left\{\pm\begin{pmatrix}
1 \\ 0
\end{pmatrix}, \pm \begin{pmatrix}
0 \\ 1
\end{pmatrix}, \pm \begin{pmatrix}
1 \\ -1
\end{pmatrix}\right\}
\end{equation*}
to denote the set of potential edge directions of $\operatorname{conv}(\Pi_{\text{narrowed}})$.

\begin{lemma}\label{lem:equalDiffOneStep}
Consider a narrowed pattern $\pi\colon\Pi_{\text{narrowed}}\to2^{\{0,\ldots,m-1\}}$, and let
$d_1,d_2\in\mathcal{D}$, $(\alpha,\beta)\in\mathbb{Z}^2$ such that $(\alpha,\beta)+\varepsilon_1d_1+\varepsilon_2d_2\in\Pi_{\text{narrowed}}$ for all $\varepsilon_1,\varepsilon_2\in\{0,1\}$, and let $r_{\varepsilon_1,\varepsilon_2}\in\{0,\ldots,m-1\}$ be such that $\pi((\alpha,\beta)+\varepsilon_1d_1+\varepsilon_2d_2)=\{r_{\varepsilon_1,\varepsilon_2}\}$ for all $\varepsilon_1,\varepsilon_2\in\{0,1\}$.
Then $r_{1,1} - r_{0,1} \equiv r_{1,0} - r_{0,0}\pmod{m}$.
\end{lemma}

\begin{proof}
We first observe that we can assume without loss of generality that either $d_1=d_2$, or
$$\{d_1,d_2\}\in\left\{\left\{\begin{pmatrix}
1 \\ 0
\end{pmatrix}, \begin{pmatrix}
0 \\ 1
\end{pmatrix}\right\},\left\{\begin{pmatrix}
1 \\ 0
\end{pmatrix}, \begin{pmatrix}
1 \\ -1
\end{pmatrix}\right\}, \left\{\begin{pmatrix}
0 \\ 1
\end{pmatrix}, \begin{pmatrix}
-1 \\ 1
\end{pmatrix}\right\}\right\}\enspace.$$
Indeed, the above situation can always be achieved by changing the sign of $d_1$ and/or $d_2$. Changing the sign of $d_1$ can be done by choosing $(\alpha',\beta')=(\alpha,\beta)+d_1$ and the directions $d_1'=-d_1$ and $d_2'=d_2$, as we have $\{(\alpha',\beta')+\varepsilon_1d_1'+\varepsilon_2d_2'\colon\varepsilon_1,\varepsilon_2\in\{0,1\}\}=\{(\alpha,\beta)+\varepsilon_1d_1+\varepsilon_2d_2\colon\varepsilon_1,\varepsilon_2\in\{0,1\}\}$, and the statement that we want to show transforms accordingly.
Analogously, we may also change the sign of $d_2$.

Now let $x^1$ be a solution for $(\alpha_1,\beta_1)=(\alpha,\beta)$, and let $x^2$ be a solution for $(\alpha_2,\beta_2)=(\alpha,\beta)+d_1+d_2$.
Applying \cref{lem:averagingPatterns} to these solutions, we obtain that there exist solutions $x^3$ and $x^4$ for $(\alpha_3,\beta_3)$ and $(\alpha_4,\beta_4)$, respectively, such that $x^1+x^2=x^3+x^4$, and the inequalities in~\eqref{eq:scalarProductBounds} are satisfied.
On a case-by-case basis, it is immediate to see that with the above assumptions, the inequalities in~\eqref{eq:scalarProductBounds} imply that $(\alpha_3,\beta_3), (\alpha_4,\beta_4)\in\{(\alpha,\beta)+d_1, (\alpha,\beta)+d_2\}$.
Moreover, because $x^1+x^2=x^3+x^4$ also implies $(\alpha_1,\beta_1)+(\alpha_2,\beta_2)=(\alpha_3,\beta_3)+(\alpha_4,\beta_4)$, we must even have $\{(\alpha_3,\beta_3), (\alpha_4,\beta_4)\}=\{(\alpha,\beta)+d_1, (\alpha,\beta)+d_2\}$. We thus assume without loss of generality that $(\alpha_3,\beta_3)=(\alpha,\beta)+d_1$ and $(\alpha_4,\beta_4)=(\alpha,\beta)+d_2$.

By definition, we then have $r_{0,0} = \gamma_B^\top x^1_B$, $r_{1,0} = \gamma_B^\top x^3_B$, $r_{0,1} = \gamma_B^\top x^4_B$, and $r_{1,1} = \gamma_B^\top x^2_B$.
The equality $x^1+x^2 = x^3+x^4$ also implies $x^1_B+x^2_B = x^3_B+x^4_B$, and hence
$$
r_{1,1} - r_{0,1} \equiv \gamma_B^\top x^2_B - \gamma_B^\top x^4_B = \gamma_B^\top x^3_B - \gamma_B^\top x^1_B \equiv r_{1,0} - r_{0,0}\pmod{m}\enspace,
$$
as desired.
\end{proof}

In what follows, for any $(\alpha_1,\beta_1),(\alpha_2,\beta_2)\in\mathbb{Z}^2$, we define
$$
D_{(\alpha_1,\beta_1),(\alpha_2,\beta_2)} \coloneqq \left\{(\alpha,\beta)\in\mathbb{Z}^2\colon\begin{array}{c} \min\{\alpha_1+\beta_1,\alpha_2+\beta_2\}  \leq \alpha+\beta \leq  \max\{\alpha_1+\beta_1,\alpha_2+\beta_2\}\\
\min\{\alpha_1,\alpha_2\}  \leq \alpha \leq \max\{\alpha_1,\alpha_2\}\\
\min\{\beta_1,\beta_2\} \leq \beta \leq  \max\{\beta_1,\beta_2\}
\end{array} \right\}\enspace.
$$
In particular, if $(\alpha_1,\beta_1),(\alpha_2,\beta_2)\in\Pi_{\text{narrowed}}$ for some domain $\Pi_{\text{narrowed}}$ of a narrowed pattern, then by \cref{lem:feasibleResiduePolyhedron}, we always also have $D_{(\alpha_1,\beta_1),(\alpha_2,\beta_2)}\subseteq\Pi_{\text{narrowed}}$.
Also, if $(\alpha_3,\beta_3)\in D_{(\alpha_1,\beta_1),(\alpha_2,\beta_2)}$, then $D_{(\alpha_1,\beta_1),(\alpha_3,\beta_3)}\subseteq D_{(\alpha_1,\beta_1),(\alpha_2,\beta_2)}$.

Moreover, we define a \emph{distance} notion for two pairs $(\alpha_1,\beta_1),(\alpha_2,\beta_2)\in\mathbb{Z}^2$ as follows:
Consider the graph $G$ on $\mathbb{Z}^2$ where two points $x,y\in\mathbb{Z}^2$ are connected by an edge if and only if $x-y\in\mathcal{D}$, and define the distance between $(\alpha_1,\beta_1)$ and $(\alpha_2,\beta_2)$ to be the length of a shortest path in $G$ that connects the two points. It is easy to see that such a shortest path has all intermediate points within $D_{(\alpha_1,\beta_1),(\alpha_2,\beta_2)}$.
Concretely, if $(\alpha_1,\beta_1)$ and $(\alpha_2,\beta_2)$ are at distance $t$, there are $d_1,\ldots,d_t\in\mathcal{D}$ such that
\begin{enumerate}
\item $(\alpha_1,\beta_1) + \sum_{i=1}^\ell d_i\in D_{(\alpha_1,\beta_1),(\alpha_2,\beta_2)}$ for all $\ell\in[t]$, and
\item $(\alpha_1,\beta_1) + \sum_{i=1}^t d_i =(\alpha_2,\beta_2)$.
\end{enumerate}

\begin{lemma}\label{lem:equalDiff}
Consider a narrowed pattern $\pi\colon\Pi_{\text{narrowed}}\to2^{\{0,\ldots,m-1\}}$ and a subset $\Pi_0\subseteq \Pi_{\text{narrowed}}$ of the form
$$
\Pi_0 = \left\{(\alpha,\beta)\in\mathbb{Z}^2\colon \ell_0\leq \alpha+\beta\leq u_0,\ \ell_1\leq \alpha\leq u_1,\ \ell_2 \leq \beta\leq u_2\right\}
$$
with $|\pi(\alpha,\beta)|=1$ for all $(\alpha,\beta)\in\Pi_{0}$, and let $r(\alpha,\beta)\in\{0,\ldots,m-1\}$ be such that $\pi(\alpha,\beta)=\{r(\alpha,\beta)\}$.
Then, for every $d\in \mathcal{D}$, there exists $r_d\in\{0,\ldots,m-1\}$ such that for any $(\alpha,\beta)\in\Pi_0$ with $(\alpha,\beta)+d\in\Pi_0$,
$$
r((\alpha,\beta)+d) - r(\alpha,\beta) \equiv r_d \pmod{m}\enspace.
$$
\end{lemma}

\begin{proof}
Fix $d\in\mathcal{D}$.
To derive the lemma, it is enough to show that for all $(\alpha_1,\beta_1), (\alpha_2,\beta_2)\in\Pi_0$ with $(\alpha_1,\beta_1)+d,(\alpha_2,\beta_2)+d\in\Pi_0$, we have
\begin{equation}\label{eq:residueDiff}
r((\alpha_1,\beta_1)+d)-r(\alpha_1,\beta_1) \equiv r((\alpha_2,\beta_2)+d)-r(\alpha_2,\beta_2) \pmod{m}\enspace.
\end{equation}
Note that if the distance between $(\alpha_1,\beta_1)$ and $(\alpha_2,\beta_2)$ is~$0$, there is nothing to show.
Moreover, if that distance is~$1$, then a corresponding shortest path connecting $(\alpha_1,\beta_1)$ and $(\alpha_2,\beta_2)$ consists of a single step $d'\in\mathcal{D}$, i.e. $(\alpha_2,\beta_2)=(\alpha_1,\beta_1)+d'$, and~\eqref{eq:residueDiff} follows from applying \cref{lem:equalDiffOneStep} to $(\alpha_1,\beta_1)$ and the directions $d, d'\in\mathcal{D}$.

More generally, let us assume by induction that~\eqref{eq:residueDiff} holds whenever the distance of $(\alpha_1,\beta_1)$ and $(\alpha_2,\beta_2)$ is less than $t$, for some $t\geq 2$, and take two such pairs of distance equal to $t$.
Then, a corresponding shortest path connecting the two points can be represented by $d_1,\ldots,d_t\in\mathcal{D}$.
Let $(\alpha',\beta')=(\alpha_1,\beta_1)+d_1$. By applying \cref{lem:equalDiffOneStep} to $(\alpha_1,\beta_1)$ and the directions $d,d_1\in\mathcal{D}$, we obtain
\begin{equation}\label{eq:firstStep}
r((\alpha_1,\beta_1)+d)-r(\alpha_1,\beta_1) \equiv r((\alpha',\beta')+d)-r(\alpha',\beta') \pmod{m}\enspace.
\end{equation}
Note that this invocation of \cref{lem:equalDiffOneStep} requires $(\alpha_1,\beta_1)+d+d_1\in\Pi_{\text{narrowed}}$, which holds because of the following:
A shortest path $P$ connecting $(\alpha_1,\beta_1)$ and $(\alpha_2,\beta_2)$ is inside $D_{(\alpha_1,\beta_1),(\alpha_2,\beta_2)}$, and shifting the whole path by $d$ gives a shortest path connecting $(\alpha_1,\beta_1)+d$ and $(\alpha_2,\beta_2)+d$ that is inside $D_{(\alpha_1,\beta_1)+d,(\alpha_2,\beta_2)+d}\subseteq\Pi_{\text{narrowed}}$.
Thus, in particular, because $(\alpha',\beta')$ is on $P$, $(\alpha',\beta')+d=(\alpha_1,\beta_1)+d+d_1$ is on the shifted path, and thus in $\Pi_{\text{narrowed}}$.

Additionally, because $(\alpha',\beta')$ and $(\alpha_2,\beta_2)$ are of distance at $t-1$, the inductive assumption gives that
\begin{equation}\label{eq:induction}
r((\alpha',\beta')+d)-r(\alpha',\beta') \equiv r((\alpha_2,\beta_2)+d)-r(\alpha_2,\beta_2)\pmod{m}\enspace.
\end{equation}
Together,~\eqref{eq:firstStep} and~\eqref{eq:induction} imply the desired~\eqref{eq:residueDiff}, thus completing the inductive step.
\end{proof}

\begin{corollary}\label{cor:linearSubpattern}
Consider a narrowed pattern $\pi\colon\Pi_{\text{narrowed}}\to2^{\{0,\ldots,m-1\}}$ and a subset $\Pi_0\subseteq\Pi_{\text{narrowed}}$ of the form
$$
\Pi_0 = \left\{(\alpha,\beta)\in\mathbb{Z}^2\colon \ell_0\leq \alpha+\beta\leq u_0,\ \ell_1\leq \alpha\leq u_1,\ \ell_2 \leq \beta\leq u_2\right\}
$$
with $|\pi(\alpha,\beta)|=1$ for all $(\alpha,\beta)\in\Pi_0$, and let $r(\alpha,\beta)\in\{0,\ldots,m-1\}$ be such that $\pi(\alpha,\beta)=\{r(\alpha,\beta)\}$.
Then, there exist $r_0, r_1, r_2\in\{0,\ldots,m-1\}$ such that for all $(\alpha,\beta)\in\Pi_0$,
$$
r(\alpha,\beta) \equiv r_0 + r_1\alpha + r_2\beta \pmod{m}\enspace.
$$
\end{corollary}

\begin{proof}
Fix $(\alpha_0,\beta_0)\in\Pi_0$. Then for any $(\alpha,\beta)\in\Pi_0$, there exists $t\in\mathbb{Z}_{\geq 0}$ and $d_1,\ldots,d_t\in\mathcal{D}$ such that
\begin{enumerate*}
\item $(\alpha_\ell,\beta_\ell)\coloneqq (\alpha_0,\beta_0) + \sum_{i=1}^\ell d_i\in \Pi_0$ for all $\ell\in[t]$, and
\item $(\alpha_t,\beta_t)=(\alpha,\beta)$.
\end{enumerate*}
Now observe that we can write
$$
r(\alpha,\beta) = r(\alpha_0,\beta_0) + \sum_{i=0}^{t-1}r((\alpha_i,\beta_i)+d_i) - r(\alpha_i,\beta_i)\enspace.
$$
By \cref{lem:equalDiff}, we know that for every $d\in\mathcal{D}\cap \{d_1,\ldots,d_t\}$, there exists $r_d\in\mathbb{Z}$ such that $r((\alpha_i,\beta_i)+d) - r(\alpha_i,\beta_i)\equiv r_d\pmod*{m}$ for all $i\in[t]$ with $d_i=d$.
Observe that by definition, we must also have $r_{-d} \equiv -r_{d}\pmod*{m}$, hence by aggregating terms in the above sum, we obtain
\begin{equation}\label{eq:residue}
r(\alpha,\beta) = r(\alpha_0,\beta_0) + a \cdot r_{\begin{psmallmatrix}
1 \\ 0
\end{psmallmatrix}} + b \cdot r_{\begin{psmallmatrix}
0 \\ 1
\end{psmallmatrix}} + c \cdot r_{\begin{psmallmatrix}
1 \\ -1
\end{psmallmatrix}}\enspace,
\end{equation}
where the coefficients $a,b,c\in\mathbb{Z}$ satisfy
\begin{equation}\label{eq:path}
\begin{pmatrix}
\alpha \\ \beta
\end{pmatrix} = \begin{pmatrix}
\alpha_0 \\ \beta_0
\end{pmatrix} + a\cdot\begin{pmatrix}
1 \\ 0
\end{pmatrix} + b\cdot\begin{pmatrix}
0 \\ 1
\end{pmatrix} + c\cdot\begin{pmatrix}
1 \\ -1
\end{pmatrix}\enspace.
\end{equation}
The latter equation follows from aggregating terms in the sum in $(\alpha,\beta)=(\alpha_0,\beta_0)+\sum_{i=1}^t d_i$.
We now distinguish two cases:
\begin{description}
\item[Case 1:] One constraint in $\Pi_0$ is tight for all points in $\Pi_0$.\\
In this case, two among the three coefficients $a$, $b$, and $c$ will be zero for any choice of $(\alpha,\beta)\in\Pi_0$.
If $c=0$ is one of the zero coefficients, then~\eqref{eq:path} implies that $a=\alpha-\alpha_0$ and $b=\beta-\beta_0$, and~\eqref{eq:residue} gives that for all $(\alpha,\beta)\in\Pi_0$ we have
$$
r(\alpha,\beta) = r(\alpha_0,\beta_0) + (\alpha-\alpha_0)\cdot r_{\begin{psmallmatrix}
1 \\ 0
\end{psmallmatrix}} + (\beta-\beta_0)\cdot r_{\begin{psmallmatrix}
0 \\ 1
\end{psmallmatrix}}\enspace,
$$
which is linear in $\alpha$ and $\beta$, as required.
Otherwise, $a=b=0$ and~\eqref{eq:path} implies that $c=\alpha-\alpha_0$, and hence by~\eqref{eq:residue},
$$
r(\alpha,\beta) = r(\alpha_0,\beta_0) + (\alpha-\alpha_0)\cdot r_{\begin{psmallmatrix}
1 \\ -1
\end{psmallmatrix}}\enspace,
$$
which is of the desired form, as well.

\item[Case 2:] No constraint in $\Pi_0$ is tight for all points in $\Pi_0$.\\
This implies that there is a pair $(\alpha',\beta')\in\Pi_{0}$ such that either
\begin{enumerate}
\item $
\begin{psmallmatrix} \alpha'\\ \beta'\end{psmallmatrix},
\begin{psmallmatrix} \alpha'\\ \beta'\end{psmallmatrix}+\begin{psmallmatrix} 1\\0\end{psmallmatrix},
\begin{psmallmatrix} \alpha'\\ \beta'\end{psmallmatrix}+\begin{psmallmatrix} 0\\1\end{psmallmatrix} \in\Pi_0$, or

\item $\begin{psmallmatrix} \alpha'\\ \beta'\end{psmallmatrix},
\begin{psmallmatrix} \alpha'\\ \beta'\end{psmallmatrix}-\begin{psmallmatrix} 1\\0\end{psmallmatrix},
\begin{psmallmatrix} \alpha'\\ \beta'\end{psmallmatrix}-\begin{psmallmatrix} 0\\1\end{psmallmatrix} \in\Pi_0$.
\end{enumerate}
\medskip
In the first case, we get
\begin{multline*}
r_{\begin{psmallmatrix} 1 \\-1 \end{psmallmatrix}}  \equiv r\left((\alpha',\beta')+(1,0)\right) - r\left((\alpha',\beta')+(0,1)\right) \\
= \Big( r\left((\alpha',\beta')+(1,0)\right) - r\left(\alpha',\beta'\right)\Big) - \Big( r\left((\alpha',\beta')+(0,1)\right) - r(\alpha',\beta')\Big) \equiv r_{\begin{psmallmatrix} 1 \\0 \end{psmallmatrix}}- r_{\begin{psmallmatrix} 0 \\1 \end{psmallmatrix}}\enspace,
\end{multline*}
and in the second case, we get
\begin{multline*}
r_{\begin{psmallmatrix} 1 \\-1 \end{psmallmatrix}}  \equiv r\left((\alpha',\beta')-(0,1)\right) - r\left((\alpha',\beta')-(1,0)\right) \\
= \Big( r(\alpha',\beta') - r\left((\alpha',\beta')-(1,0)\right)\Big) - \Big( r(\alpha',\beta')-r\left((\alpha',\beta')-(0,1)\right)\Big) \equiv r_{\begin{psmallmatrix} 1 \\0 \end{psmallmatrix}}- r_{\begin{psmallmatrix} 0 \\1 \end{psmallmatrix}}\enspace.
\end{multline*}
Note that this gives the same relation among the different vectors $r_d$ in both cases.
Using this in~\eqref{eq:residue} together with the fact that~\eqref{eq:path} implies $a+c=\alpha-\alpha_0$ and $b-c=\beta-\beta_0$, we obtain that for all $(\alpha,\beta)\in\Pi_{0}$, we have
\begin{multline*}
r(\alpha,\beta) \equiv r(\alpha_0,\beta_0) + a \cdot r_{\begin{psmallmatrix}
1 \\ 0
\end{psmallmatrix}} + b \cdot r_{\begin{psmallmatrix}
0 \\ 1
\end{psmallmatrix}} + c \cdot \left( r_{\begin{psmallmatrix} 1 \\0 \end{psmallmatrix}}- r_{\begin{psmallmatrix} 0 \\1 \end{psmallmatrix}}\right) \\
= r(\alpha_0,\beta_0) + (\alpha-\alpha_0) \cdot r_{\begin{psmallmatrix}
1 \\ 0
\end{psmallmatrix}} + (\beta-\beta_0) \cdot r_{\begin{psmallmatrix}
0 \\ 1
\end{psmallmatrix}} \pmod{m}\enspace,
\end{multline*}
which is again a relation of the desired form.\qedhere
\end{description}
\end{proof}

\subsubsection*{Proof of \cref{thm:integration}\MOORdot}

We actually prove a slightly more general version of \cref{thm:integration}, in order not only to apply it to linear patterns $\pi$, but also to linear sub-patterns of a pattern $\pi$.
To this end, let us formally repeat the definition of sub-patterns.

\begin{definition}
Let $\pi\colon\Pi_{\text{narrowed}}\to 2^{\{0,\ldots,m-1\}}$ be a narrowed pattern stemming from an \RCCTUF{} problem of the form given in \eqref{eq:structured-problem}.
We say that $\widetilde \pi\colon \widetilde\Pi\to 2^{\{0,\ldots,m-1\}}$ is a \emph{sub-pattern} of $\pi$ if the following holds:
\begin{enumerate}
\item $\widetilde \Pi \subseteq \Pi_{\text{narrowed}}$.
\item There are $\ell_i, u_i\in\mathbb{Z}$ for $i\in \{0,1,2\}$ such that
\begin{equation*}
\widetilde\Pi = \left\{(\alpha,\beta)\in\mathbb{Z}^2\colon \ell_0\leq \alpha+\beta\leq u_0,\ \ell_1\leq \alpha\leq u_1,\ \ell_2 \leq \beta\leq u_2\right\}\enspace.
\end{equation*}
\item $\widetilde{\pi}(\alpha, \beta)\subseteq\pi(\alpha,\beta)$ for all $(\alpha, \beta)\in \widetilde\Pi$.
\end{enumerate}
\end{definition}

Moreover, we say that a solution $x=(x_A,x_B)$ of an \RCCTUF{} problem of the form given in \eqref{eq:structured-problem} is \emph{covered} by a sub-pattern $\widetilde{\pi}$ if $\gamma^\top x_B\in\widetilde{\pi}(\alpha,\beta)$ for $\alpha=f^\top x_B$ and $\beta=h^\top x_A$.

\begin{theorem}\label{thm:integrationSubpatterns}
Consider an \RCCTUF{} problem of the form given in~\eqref{eq:structured-problem}, let $\pi$ be an associated narrowed pattern, and let $\widetilde\pi$ be a linear sub-pattern of $\pi$ with domain given by $\widetilde\Pi=\{(\alpha,\beta)\in\mathbb{Z}^2\colon \ell_0\leq \alpha+\beta\leq u_0,\ \ell_1\leq \alpha\leq u_1,\ \ell_2 \leq \beta\leq u_2\}$, where $\ell_i, u_i\in\mathbb{Z}$ for $i\in\{0,1,2\}$. Then, we can in strongly polynomial time determine $r_0, r_1, r_2\in \{0,1,\ldots, m-1\}$ such that the \RCCTUF{} problem
\begin{equation}\label{eq:integrationProblem}
\begin{array}{rcrcrcrcl}
       &      &             A x_A & + &   e y_1 &   &         & \leq   & b_A \\
       &      &        h^\top x_A &   &         & - &     y_2 &   =    & 0 \\
\ell_0 & \leq &                   &   &     y_1 & + &     y_2 & \leq   & u_0 \\
\ell_1 & \leq &                   &   &     y_1 &   &         & \leq   & u_1 \\
\ell_2 & \leq &                   &   &         &   &     y_2 & \leq   & u_2 \\
       &      & \gamma_A^\top x_A & + & r_1 y_1 & + & r_2 y_2 & \in    & r_0 + R \pmod{m}\\
       &      &               x_A &   &         &   &         & \in    & \mathbb{Z}^{n_A}\\
       &      &                   &   &     y_1 & , &     y_2 & \in    & \mathbb{Z}\\
\end{array}
\end{equation}
has a feasible solution if and only if the original \RCCTUF{} problem has one that is covered by $\widetilde{\pi}$. Moreover, a solution of one problem can be transformed into one of the other in strongly polynomial time.
\end{theorem}

\begin{proof}
By \cref{cor:linearSubpattern}, there exist $r_0, r_1, r_2\in\{0,\ldots,m-1\}$ such that $r(\alpha,\beta)\coloneqq -r_0+\alpha r_1 + \beta r_2$ has the following property for each $(\alpha,\beta)\in\widetilde\Pi$:
If $x_B$ is a solution of the $B$-problem for $(\alpha,\beta)$, then $\gamma_B^\top x_B \equiv r(\alpha,\beta)\pmod{m}$.
We claim that \cref{thm:integrationSubpatterns} holds for this choice of $r_0$, $r_1$, and $r_2$.

To see this, first let $(x_A, x_B)\in\mathbb{Z}^{n_A+n_B}$ be a solution of the original \RCCTUF{} problem that is covered by $\widetilde\pi$, i.e., a solution with scalar products $(\alpha,\beta)\in\widetilde\Pi$. We claim that $(x_A, \alpha,\beta)$ is a solution of~\eqref{eq:integrationProblem}. Indeed, feasibility for the original problem gives
\begin{equation*}
\begin{array}{rcrcl}
A x_A & + & ef^\top x_B & \leq & b_A \\
gh^\top x_A & + & B x_B & \leq & b_B \\
\gamma_A^\top x_A & + & \gamma_B^\top x_B & \in & R \pmod{m}\enspace,
\end{array}
\end{equation*}
and the first constraint is equivalent to $Ax_A+e\alpha \leq b_A$. Moreover, $h^\top x_A - \beta=0$ is satisfied by definition of $\beta$, and the constraints in the third, forth, and fifth line of~\eqref{eq:integrationProblem} are satisfied by $(y_1,y_2)=(\alpha,\beta)$ because $(\alpha,\beta)\in\widetilde\Pi$. Finally, the congruency constraint is satisfied because $\gamma_A^\top x_A - r_0 + \alpha r_1 + \beta r_2 = \gamma_A^\top x_A + r(\alpha,\beta) \equiv \gamma_A^\top x_A + \gamma_B^\top x_B \in R\pmod*{m}$ is equivalent to $\gamma_A^\top x_A + r_1\alpha + r_2\beta\in r_0+R\pmod*{m}$.

On the other hand, for any solution $(x_A,\alpha,\beta)$ of~\eqref{eq:integrationProblem}, we get that $(\alpha,\beta)\in\widetilde{\Pi}$ due to the constraints in~\eqref{eq:integrationProblem}, and hence any solution $x_B$ of the relaxation of the $B$-problem satisfies $\gamma_B^\top x_B\equiv r(\alpha,\beta)\pmod*{m}$.
From the same arguments as before, it follows that $(x_A,x_B)$ is feasible for the original \RCCTUF{} problem.

To conclude the proof, observe that transforming the solution of the original problem to a solution of~\eqref{eq:integrationProblem} only requires the computation of $\alpha$ and $\beta$.
For the other way round, we need to compute a feasible solution to the relaxation of the $B$-problem, which can be done in strongly polynomial time using the algorithm of \textcite{tardos_1986_strongly}.
\end{proof}

\begin{proof}[Proof of \cref{thm:integration}]
Because $\pi$ is a linear pattern by assumption, we can apply \cref{thm:integrationSubpatterns} with $\widetilde{\pi}=\pi$, and \cref{thm:integration} immediately follows.
\end{proof}

\subsubsection*{More properties of patterns and a proof of \cref{lem:coveringPattern}\MOORdot}

After having studied linear patterns and sub-patterns so far in this section, we now focus on non-linear patterns $\pi$, i.e., patterns that have at least one pair $(\alpha,\beta)$ with $|\pi(\alpha,\beta)|\geq 2$ in their domain.
The first lemma below shows how the property of having $|\pi(\alpha,\beta)|\geq 2$ propagates over the domain of a pattern, again using our averaging lemma, \cref{lem:averagingPatterns}.

With the ultimate goal of this subsection being to prove \cref{lem:coveringPattern}, we first show \cref{lem:type2step}, which showcases one important and repeatedly used situation in which \cref{lem:coveringPattern} holds.
We remark at this point that the requirement $|R|\geq m-2$ stated in \cref{lem:coveringPattern} is only due to \cref{lem:type2step}.
Hence, future attempts of overcoming this barrier using the ideas presented here will have to exploit setups beyond the one in \cref{lem:type2step}.
In contrast, the assumption in \cref{lem:coveringPattern} of $m$ being a prime number is exploited in several places.

Also, we remark that in this part, we aim at providing tools for analyzing (narrowed) patterns in a slightly more general setup than what we actually need.
More precisely, in our concrete case it would be enough to analyze narrowed patterns that are contained in a rectangular box of scalar product pairs $(\alpha,\beta)$ of dimensions $3\times 3$ (this follows by \cref{lem:patternsBounds}, for example, and our assumption $|R|\geq m-2$).
Still, we aim for the slightly more general presentation of our methods, which may be useful in potential future work on these topics, in particular when dropping the assumption $|R|\geq m-2$.

We start by observing that \cref{lem:coveringPattern} trivially holds in the case where the pattern $\pi$ is linear, as we can then choose $\widetilde{\pi}=\pi$.
In case of a non-linear pattern, we know that there is at least one pair $(\alpha,\beta)$ of scalar products in the domain of the pattern such that $|\pi(\alpha,\beta)|\geq 2$. If there exists a solution for such scalar products $(\alpha,\beta)$, item~\ref{lemitem:type1} of \cref{lem:coveringPattern} applies, so having many $(\alpha,\beta)$ with $|\pi(\alpha,\beta)|\geq 2$ is desirable.
Luckily, the subsequent lemma proves that such $(\alpha,\beta)$ cannot appear in a very isolated way.

\begin{lemma}\label{lem:pushingTwos}
Consider a narrowed pattern $\pi\colon\Pi_{\text{narrowed}}\to2^{\{0,\ldots,m-1\}}$, let $(\alpha,\beta)\in\Pi_{\text{narrowed}}$ and $d\in\mathcal{D}$ such that $(\alpha,\beta)+d, (\alpha,\beta)+2d\in\Pi_{\text{narrowed}}$.
If $|\pi(\alpha,\beta)|\geq 2$, then $|\pi((\alpha,\beta)+d)|\geq 2$, as well.
\end{lemma}

\begin{proof}
Let $x^1$ and $y^1$ be feasible solutions of the relaxation of the underlying problem for scalar products $(\alpha, \beta)$ with different residues, i.e., $\gamma_B^\top x^1_B\not\equiv \gamma_B^\top y^1_B\pmod*{m}$, and let $x^2$ be any solution of the relaxation of the problem for scalar products $(\alpha,\beta)+2d$.

Applying the averaging lemma (\cref{lem:averagingPatterns}) to the solutions $x^1$ and $x^2$, and to the solutions $y^1$ and $x^2$, we obtain solutions $x^3,x^4$ and solutions $y^3,y^4$, respectively, such that $x^1 + x^2 = x^3 + x^4$ and $y^1 + x^2 = y^3 + y^4$.
Moreover, the inequalities~\eqref{eq:scalarProductBounds} in \cref{lem:averagingPatterns} state that all of $x^3$, $x^4$, $y^3$, and $y^4$ are solutions for the scalar products $(\alpha,\beta)+d$.
Consequently, $\{(\gamma_B^\top x^3_B\bmod{m}), (\gamma_B^\top x^4_B\bmod{m}), (\gamma_B^\top y^3_B\bmod{m}), (\gamma_B^\top y^4_B\bmod{m})\}\subseteq\pi((\alpha,\beta)+d)$.
To get that $|\pi((\alpha,\beta)+d)|\geq 2$, note that these residues satisfy
\begin{equation*}
\gamma_B^\top x^1_B + \gamma_B^\top x^2_B \equiv \gamma_B^\top x^3_B + \gamma_B^\top x^4_B \pmod{m}\enspace, \quad\text{and}\quad \gamma_B^\top y^1_B + \gamma_B^\top x^2_B \equiv \gamma_B^\top y^3_B + \gamma_B^\top y^4_B \pmod{m}\enspace,
\end{equation*}
and hence, because $\gamma_B^\top x^1_B\not\equiv \gamma_B^\top y^1_B$, at least two of the residues among $(\gamma_B^\top x^3_B\bmod{m})$, $(\gamma_B^\top x^4_B\bmod{m})$, $(\gamma_B^\top y^3_B\bmod{m})$, and $(\gamma_B^\top y^4_B\bmod{m})$ must be distinct, which proves the lemma.
\end{proof}

Even non-linear patterns $\pi$ might have several $(\alpha,\beta)$ in their support that satisfy $|\pi(\alpha,\beta)|=1$.
In \cref{lem:coveringPattern}, such squares may be covered by a linear sub-pattern, but it turns out that in general, there is no sub-pattern covering all pairs $(\alpha,\beta)$ with $|\pi(\alpha,\beta)|=1$.
The following lemma describes a configuration that allows for dealing with such pairs in a different way.

\begin{lemma}\label{lem:type2step}
Consider an \RCCTUF{} problem of the form given in~\eqref{eq:structured-problem} with prime modulus $m$ and $|R|\geq m-2$.
Let $\pi\colon\Pi_{\text{narrowed}}\to2^{\{0,\ldots,m-1\}}$ be an associated narrowed pattern, and let $(\alpha,\beta)\in\mathbb{Z}^2$ and $d\in\mathcal{D}$ with
\begin{align*}
(\alpha,\beta),(\alpha,\beta)+d,(\alpha,\beta)+2d\in\Pi_{\text{narrowed}}\enspace,\quad
|\pi(\alpha,\beta)|=1\enspace, \quad\text{and}\quad
|\pi((\alpha,\beta)+d)|\geq 2\enspace.
\end{align*}
If the problem has a solution with scalar products $(\alpha,\beta)$, one of the following holds:
\begin{enumerate}
\item $(\alpha,\beta)$ satisfies case~\ref{lemitem:goodCombo} of \cref{lem:coveringPattern}.
\item There is a solution with scalar products $(\alpha,\beta)+d$, i.e., $(\alpha,\beta)+d$ satisfies case~\ref{lemitem:type1} of \cref{lem:coveringPattern}.
\end{enumerate}
\end{lemma}

\begin{proof}
Assume that $(\alpha,\beta)$ does not satisfy case~\ref{lemitem:goodCombo} of \cref{lem:coveringPattern}, i.e. it is not true that for any solution $x_A$ of the $A$-problem for scalar products $(\alpha,\beta)$, there exists a solution $x_B$ of the $B$-problem such that $(x_A,x_B)$ is feasible for the original problem.
Recall that given a feasible solution $x_A$ of the relaxation of the $A$-problem and a feasible solution $x_B$ of the relaxation of the $B$-problem for the same scalar products $(\alpha,\beta)$, the combined solution $(x_A,x_B)$ is feasible for the original problem if and only if it satisfies the congruency constraint $\gamma_A^\top x_A + \gamma_B^\top x_B\in R\pmod*{m}$.
As $|\pi(\alpha,\beta)|=1$ by assumption, $r = (\gamma_B^\top x_B \bmod{m})$ is the same for all feasible solutions $x_B$ of the relaxation of the $B$-problem.
Consequently, the only reason why a combined solution $(x_A,x_B)$ can be infeasible is that $\gamma_A^\top x_A\not\in R-r\pmod*{m}$.
On the other hand, because by assumption, the problem has a feasible solution with scalar products $(\alpha,\beta)$, there must also be another solution $x_A'$ with $\gamma_A^\top x_A'\in R-r\pmod*{m}$.

Define $\pi_A\colon\Pi_{\text{narrowed}}\to 2^{\{0,1,\ldots,m-1\}}$ such that $\pi_A(\alpha',\beta')$ denotes, for every $(\alpha',\beta')\in\Pi_{\text{narrowed}}$, the set of residues $\gamma_A^\top x_A$ that can be achieved by solutions of the relaxation of the $A$-problem with scalar products $(\alpha',\beta')$.
Hence, $\pi_A$ is defined identically to $\pi$, with the only difference that $\pi_A$ captures attainable residues in the $A$-problem instead of the $B$-problem.
Hence, by symmetry between the $A$-problem and $B$-problem, properties holding for $\pi$ and the $B$-problem also hold for $\pi_A$ and the $A$-problem.
In particular, the previous argument showed that $|\pi_A(\alpha,\beta)|\geq 2$, and by definition of $\Pi_{\text{narrowed}}$, we know that the relaxation of the $A$-problem is feasible for all $(\alpha',\beta')\in\Pi_{\text{narrowed}}$.
Thus, applying \cref{lem:pushingTwos} to $\pi_A$, we obtain that $|\pi_A((\alpha,\beta)+d)|\geq 2$, as well.

The residues that a solution $(x_A,x_B)$ of the relaxation of the original problem can achieve for scalar product $(\alpha,\beta)+d$ are given by the set $\pi_A((\alpha,\beta)+d)+\pi((\alpha,\beta)+d)$.
By the Cauchy-Davenport Inequality (\cref{lem:cauchyDavenport}), which we can apply as $m$ is a prime number by assumption, we have
$$
|\pi_A((\alpha,\beta)+d)+\pi((\alpha,\beta)+d)| \geq \min\{m, |\pi_A((\alpha,\beta)+d)|+|\pi((\alpha,\beta)+d)|-1\} \geq \min\{m, 3\}\enspace.
$$
As the set $R$ of target residues satisfies $|R|\geq m-2$, we conclude that at least one of the achievable residues is a target residue, and hence there exists a solution of the problem with scalar products $(\alpha,\beta)+d$.
\end{proof}

To prove \cref{lem:coveringPattern}, we distinguish two cases based on whether the \emph{interior} of the pattern domain is empty or not, where interior is defined as follows.

\begin{definition}
For a set $\Pi\subseteq \mathbb{Z}^n$ of the form
\begin{equation}\label{eq:patternShape}
\Pi = \left\{(\alpha,\beta)\in\mathbb{Z}^2\colon \ell_0\leq \alpha+\beta\leq u_0,\ \ell_1\leq \alpha \leq u_1,\ \ell_2\leq \beta\leq u_2\right\}
\end{equation}
with $\ell_i, u_i\in\mathbb{Z}$ for $i\in \{0,1,2\}$, we say that $(\alpha,\beta)\in\Pi$ is in the \emph{interior} of $\Pi$ if none of the constraints in~\eqref{eq:patternShape} are tight for $(\alpha,\beta)$.
Else, we say that $(\alpha,\beta)$ is on the \emph{boundary} of $\Pi$.
\end{definition}

In fact, for non-linear patterns $\pi$ whose support has non-empty interior, we show that any pair $(\alpha,\beta)$ with $|\pi(\alpha,\beta)|=1$ is part of a configuration of the type described by \cref{lem:type2step}, leading to the following.

\begin{lemma}\label{lem:nonEmptyInterior}
Consider a non-linear narrowed pattern $\pi$ for a feasible \RCCTUF{} problem as given in~\eqref{eq:structured-problem} with prime modulus $m$ and $|R|\geq m-2$. If the domain of $\pi$ has non-empty interior, then~\ref{lemitem:goodCombo} or~\ref{lemitem:type1} in \cref{lem:coveringPattern} holds.
\end{lemma}

To prove this lemma, we study the structure of patterns more closely. We start with an observation, where again, $\mathcal{D}\coloneqq\{\pm\begin{psmallmatrix}
1 \\ 0
\end{psmallmatrix}, \pm\begin{psmallmatrix}
0 \\ 1
\end{psmallmatrix}, \pm\begin{psmallmatrix}
1 \\ -1
\end{psmallmatrix}\}$
denotes the set of all potential edge directions of the convex hull of a pattern domain.

\begin{observation}
Let $\Pi\subseteq \mathbb{Z}^n$ be of the form
\begin{equation}
\Pi = \left\{(\alpha,\beta)\in\mathbb{Z}^2\colon \ell_0\leq \alpha+\beta\leq u_0,\ \ell_1\leq \alpha \leq u_1,\ \ell_2\leq \beta\leq u_2\right\}
\end{equation}
with $\ell_i, u_i\in\mathbb{Z}$ for $i\in \{0,1,2\}$.
Then $(\alpha,\beta)\in\Pi$ is in the interior of $\Pi$ if and only if $(\alpha,\beta)+d\in\Pi$ for all $d\in\mathcal{D}$.
\end{observation}

\begin{lemma}\label{lem:neighbouringTwo}
Consider a narrowed pattern $\pi\colon\Pi_{\text{narrowed}}\to2^{\{0,\ldots,m-1\}}$ such that there exists $(\alpha_1,\beta_1)\in\Pi_{\text{narrowed}}$ with $|\pi(\alpha_1,\beta_1)|\geq 2$.
Then, for every $(\alpha_2,\beta_2)\in\Pi_{\text{narrowed}}\setminus\{(\alpha_1,\beta_1)\}$, there exists $d\in\mathcal{D}$ such that $(\alpha_2,\beta_2)+d\in D_{(\alpha_1,\beta_1),(\alpha_2,\beta_2)}$ and $|\pi((\alpha_2,\beta_2)+d)|\geq 2$.
\end{lemma}

\begin{proof}
For any two pairs $(\alpha,\beta),(\alpha',\beta')\in\mathbb{Z}^2$, denote
$$
\Delta((\alpha,\beta),(\alpha',\beta'))\coloneqq\max\{|\alpha-\alpha'|, |\beta-\beta'|, |(\alpha+\beta)-(\alpha'+\beta')|\}\enspace.
$$
We prove that \cref{lem:neighbouringTwo} holds by induction on $\Delta=\Delta((\alpha_1,\beta_1),(\alpha_2,\beta_2))$.
For the base case, note that $\Delta=1$ implies that there exists $d\in\mathcal{D}$ such that $(\alpha_1,\beta_1)=(\alpha_2,\beta_2)+d$, so there is nothing to show.
Thus, assume that \cref{lem:neighbouringTwo} holds if $\Delta<t$ for some $t\in\mathbb{Z}_{\geq 2}$, and consider a situation with $\Delta = t$.
Let $x^1$ and $y^1$ be two solutions for scalar products $(\alpha_1,\beta_1)$ with $\gamma_B^\top x^1_B\not\equiv\gamma_B^\top y^1_B\pmod*{m}$.
These solutions exist because by assumption, $|\pi(\alpha_1,\beta_1)|\geq 2$.
Additionally, let $x^2$ be a solution for scalar products $(\alpha_2,\beta_2)$.
Applying the averaging lemma (\cref{lem:averagingPatterns}) to $x^1$ and $x^2$, and to $y^1$ and $x^2$, we obtain solutions $x^3,x^4$ and $y^3,y^4$, respectively, where $x^1+x^2=x^3+x^4$ and $y^1+x^2=y^3+y^4$.
Observe that the inequalities~\eqref{eq:scalarProductBounds} leave only one option for each of $(\alpha_3,\beta_3)$ and $(\alpha_4,\beta_4)$, and both of these are within $D_{(\alpha_1,\beta_1),(\alpha_2,\beta_2)}$. In particular, they satisfy
$$
\Delta((\alpha_3,\beta_3),(\alpha_2,\beta_2)) \leq \left\lceil\sfrac{t}{2}\right\rceil
 \quad\text{and}\quad
\Delta((\alpha_4,\beta_4),(\alpha_2,\beta_2)) \leq \left\lceil\sfrac{t}{2}\right\rceil\enspace.
$$
We claim that either $|\pi(\alpha_3,\beta_3)|\geq 2$ or $|\pi(\alpha_4,\beta_4)|\geq 2$, which allows us to apply the inductive assumption, thus finishing the proof.

To show the claim, assume for the sake of deriving a contradiction that $|\pi(\alpha_3,\beta_3)| = |\pi(\alpha_4,\beta_4)| = 1$.
Without loss of generality, let $x^3$ and $y^3$ be solutions for $(\alpha_3,\beta_3)$, while $x^4$ and $y^4$ are solutions for $(\alpha_4,\beta_4)$.
Then, by the assumption, $\gamma_B^\top x^3_B\equiv \gamma_B^\top y^3_B\pmod*{m}$, and $\gamma_B^\top x^4_B\equiv \gamma_B^\top y^4_B\pmod*{m}$.
Thus, we also obtain
$$
\gamma_B^\top x^1_B + \gamma_B^\top x^2_B
= \gamma_B^\top x^3_B + \gamma_B^\top x^4_B
\equiv \gamma_B^\top y^3_B + \gamma_B^\top y^4_B
= \gamma_B^\top y^1_B + \gamma_B^\top x^2_B \pmod{m}\enspace,
$$
but this contradicts the choice of $x^1$ and $y^1$ such that $\gamma_B^\top x^1_B\not\equiv\gamma_B^\top y^1_B\pmod*{m}$.
\end{proof}

\begin{lemma}\label{lem:interiorTwos}
Consider a non-linear narrowed pattern $\pi\colon\Pi_{\text{narrowed}}\to2^{\{0,\ldots,m-1\}}$.
Then, for every $(\alpha,\beta)$ in the interior of $\Pi_{\text{narrowed}}$, we have $|\pi(\alpha,\beta)|\geq2$.
\end{lemma}

\begin{proof}
Because the pattern $\pi$ is non-linear, there exists $(\alpha_1,\beta_1)\in\Pi_{\text{narrowed}}$ such that $|\pi(\alpha_1,\beta_1)|\geq2$.
Consider any $(\alpha,\beta)$ different from $(\alpha_1,\beta_1)$ in the interior of $\Pi_{\text{narrowed}}$.
Then by \cref{lem:neighbouringTwo}, there exists $d\in\mathcal{D}$ such that $(\alpha,\beta)+d\in\Pi_{\text{narrowed}}$ and $|\pi((\alpha,\beta)+d)|\geq 2$.
Because $(\alpha,\beta)$ is in the interior of $\Pi_{\text{narrowed}}$, we also have that $(\alpha+\beta)-d\in\Pi_{\text{narrowed}}$.
Consequently, applying \cref{lem:pushingTwos}, we obtain that $|\pi(\alpha,\beta)|\geq 2$, as well.
\end{proof}

Having the above at hand, we are now ready to prove \cref{lem:nonEmptyInterior}.

\begin{proof}[Proof of \cref{lem:nonEmptyInterior}]
If the problem has a solution for scalar products $(\alpha, \beta)\in\Pi_{\text{narrowed}}$ with $|\pi(\alpha,\beta)|\geq 2$, then case~\ref{lemitem:type1} of \cref{lem:coveringPattern} holds.
Thus, assume that this is not the case, i.e., the problem only has solutions for scalar products $(\alpha,\beta)\in\Pi_{\text{narrowed}}$ with $|\pi(\alpha,\beta)|=1$.

By \cref{lem:interiorTwos}, there is a scalar product $(\alpha',\beta')$ in the interior of $\Pi_{\text{narrowed}}$ with $|\pi(\alpha',\beta')|\geq 2$.
Applying \cref{lem:neighbouringTwo} to $(\alpha',\beta')$ and $(\alpha,\beta)$, we obtain that there exists $d\in\mathcal{D}$ such that $(\alpha,\beta)+d\in D_{(\alpha,\beta),(\alpha',\beta')}\subseteq\Pi_{\text{narrowed}}$ and $|\pi((\alpha,\beta)+d)|\geq 2$.
Note that because $(\alpha',\beta')$ is in the interior of $\Pi_{\text{narrowed}}$, we have $(\alpha',\beta')+d\in\Pi_{\text{narrowed}}$, and thus also $D_{(\alpha,\beta),(\alpha',\beta')+d}\subseteq\Pi_{\text{narrowed}}$.
As $(\alpha,\beta)+d\in D_{(\alpha,\beta),(\alpha',\beta')}$, it is also true that $(\alpha,\beta)+2d\in D_{(\alpha,\beta),(\alpha',\beta')+d}$, so we conclude that $(\alpha,\beta)+2d\in\Pi_{\text{narrowed}}$.

Observe that $(\alpha,\beta)$ and $d$ thus satisfy the assumptions of \cref{lem:type2step}.
Because we assumed that there are no scalar product pairs satisfying case~\ref{lemitem:type1} of \cref{lem:coveringPattern}, \cref{lem:type2step} implies that here, $(\alpha,\beta)$ satisfies case~\ref{lemitem:goodCombo} of \cref{lem:coveringPattern}.
\end{proof}

To prove \cref{lem:coveringPattern}, it remains to deal with patterns whose domain has empty interior, which is covered by the statement below.

\begin{lemma}\label{lem:emptyInterior}
Consider a non-linear narrowed pattern $\pi$ for a feasible \RCCTUF{} problem as given in~\eqref{eq:structured-problem} with prime modulus $m$ and $|R|\geq m-2$. If the domain of $\pi$ has empty interior, \cref{lem:coveringPattern} holds.
\end{lemma}

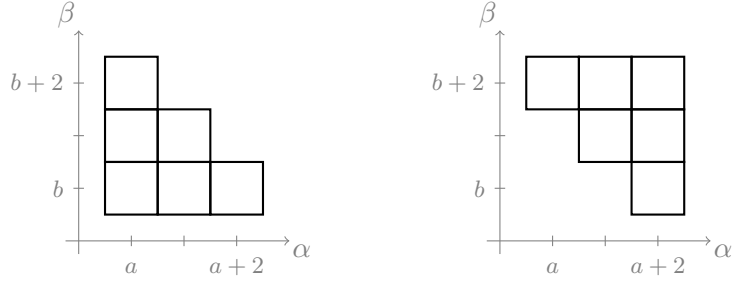
\begin{figure}[!ht]
\centering
\newcommand{\figscale}{0.7}
\begin{tikzpicture}[scale=\figscale]
\begin{scope}[every node/.style={draw, thick, rectangle, minimum width=\figscale cm, minimum height=\figscale cm, inner sep=0pt, text width=\figscale cm, align=center}]
\foreach \x/\y/\fill/\text in {-1/ 1/none/,  -1/ 0/none/,  0/ 0/none/, -1/-1/none/,   0/-1/none/, 1/-1/none/}
\node[fill=\fill, font=\small] (\x\y) at (\x, \y) {\vphantom{ly}$\text$};
\end{scope}

\begin{scope}[every path/.style={gray}]
\draw[->] (-2.25, -2) to node[pos=1, anchor=north west, inner sep=1pt] {$\alpha$} (2, -2);
\draw[->] (-2, -2.25) to node[pos=1, anchor=south east, inner sep=1pt] {$\beta$} (-2, 2);
\foreach \x/\xtext in {-1/a, 0/, 1/a+2}
\draw (\x, -2.1) to node[pos=0, anchor=north, font=\footnotesize] {$\xtext$\vphantom{l}} (\x, -1.9);
\foreach \y/\ytext in {-1/b, 0/, 1/b+2}
\draw (-2.1, \y) to node[pos=0, anchor=east, font=\footnotesize] {$\ytext$\vphantom{l}} (-1.9, \y);
\end{scope}

\begin{scope}[xshift=8cm]
\begin{scope}[every node/.style={draw, thick, rectangle, minimum width=\figscale cm, minimum height=\figscale cm, inner sep=0pt, text width=\figscale cm, align=center}]
\foreach \x/\y/\fill/\text in {-1/ 1/none/,  0/ 1/none/,   1/ 1/none/,
0/ 0/none/, 1/ 0/none/,1/-1/none/}
\node[fill=\fill, font=\small] (\x\y) at (\x, \y) {\vphantom{ly}$\text$};
\end{scope}

\begin{scope}[every path/.style={gray}]
\draw[->] (-2.25, -2) to node[pos=1, anchor=north west, inner sep=1pt] {$\alpha$} (2, -2);
\draw[->] (-2, -2.25) to node[pos=1, anchor=south east, inner sep=1pt] {$\beta$} (-2, 2);
\foreach \x/\xtext in {-1/a, 0/, 1/a+2}
\draw (\x, -2.1) to node[pos=0, anchor=north, font=\footnotesize] {$\xtext$\vphantom{l}} (\x, -1.9);
\foreach \y/\ytext in {-1/b, 0/, 1/b+2}
\draw (-2.1, \y) to node[pos=0, anchor=east, font=\footnotesize] {$\ytext$\vphantom{l}} (-1.9, \y);
\end{scope}
\end{scope}

\end{tikzpicture} \caption{Shapes $\Pi_1^{(a,b)}$ (left) and $\Pi_2^{(a,b)}$ (right).}
\label{fig:boundedPatternShapes}
\end{figure}

Before proving \cref{lem:emptyInterior}, we first observe structural properties of pattern domains with an empty interior.
The possible shapes of such domains is very restricted.
In particular, the subsequent lemma shows that they are either flat, or contained in small shapes  $\Pi_0^{(a,b)}$ and $\Pi_1^{(a,b)}$ for $a,b\in\mathbb{Z}$ given by
\begin{align*}
&\Pi_1^{(a,b)} \coloneqq \left\{(\alpha,\beta)\in\mathbb{Z}^2\colon \begin{array}{c}a\leq\alpha\leq a+2\\ b\leq\beta\leq b+2\\ a+b\leq\alpha+\beta\leq a+b+2\end{array}\right\}\\
\text{and}\quad
&\Pi_2^{(a,b)} \coloneqq \left\{(\alpha,\beta)\in\mathbb{Z}^2\colon \begin{array}{c}a\leq\alpha\leq a+2\\ b\leq\beta\leq b+2\\ a+b+2\leq\alpha+\beta\leq a+b+4\end{array}\right\}\enspace,
\end{align*}
as depicted in \cref{fig:boundedPatternShapes}.

In what follows, we define $\mathcal{D}^\perp\coloneqq \{\pm\begin{psmallmatrix}
1 \\ 0
\end{psmallmatrix}, \pm\begin{psmallmatrix}
0 \\ 1
\end{psmallmatrix}, \pm\begin{psmallmatrix}
1 \\ 1
\end{psmallmatrix}\}$, which is a set of vectors orthogonal to the potential edge directions $\mathcal{D}$ of the convex hull of a pattern support (see \cref{lem:feasibleResiduePolyhedron}).

\begin{lemma}\label{lem:structureEmptyInterior}
Let $\Pi\subseteq \mathbb{Z}^n$ be of the form
\begin{equation}\label{eq:generalDomainShape}
\Pi = \left\{(\alpha,\beta)\in\mathbb{Z}^2\colon \ell_0\leq \alpha+\beta\leq u_0,\ \ell_1\leq \alpha \leq u_1,\ \ell_2\leq \beta\leq u_2\right\}
\end{equation}
with $\ell_i, u_i\in\mathbb{Z}$ for $i\in \{0,1,2\}$, and assume that $\Pi$ has empty interior.
Then at least one of the following holds:
\begin{enumerate}
\item\label{lemitem:flatDirection} A direction in $\mathcal{D}^\perp$ is a flat direction of width at most $1$ for $\Pi$.
\item\label{lemitem:boundedShape} There are $a,b\in\mathbb{Z}$ and $i\in\{1,2\}$ such that $\Pi\subseteq\Pi_{i}^{(a,b)}$.
\end{enumerate}
\end{lemma}

\begin{proof}
Assume that item~\ref{lemitem:flatDirection} does not hold, i.e., the three directions $\begin{psmallmatrix}
1 \\ 0
\end{psmallmatrix}$, $\begin{psmallmatrix}
0 \\ 1
\end{psmallmatrix}$, and $\begin{psmallmatrix}
1 \\ 1
\end{psmallmatrix}$
are all of width at least $2$, and let $(\alpha_0,\beta_0)\in\arg\max_{(\alpha,\beta)\in\Pi}\left(\alpha-\beta\right)$. Starting from a general $\Pi$ of the form in~\eqref{eq:generalDomainShape}, there are three cases to distinguish:
\begin{description}
\item[Case 1:] $(\alpha_0,\beta_0)=(u_1,\ell_2)$ and the edge directions at $(\alpha_0,\beta_0)$ are $d_1=\begin{psmallmatrix}
0\\1
\end{psmallmatrix}$ and $d_2=\begin{psmallmatrix}
-1\\0
\end{psmallmatrix}$. \\
This implies that $(\alpha_0,\beta_0)+d_1,(\alpha_0,\beta_0)+d_2\in\Pi$, hence we must have $\ell_0\leq \alpha_0+\beta_0-1$ and $u_0\geq \alpha_0+\beta_0+1$.
Also note that because $\begin{psmallmatrix}
1 \\ 0
\end{psmallmatrix}$, $\begin{psmallmatrix}
0 \\ 1
\end{psmallmatrix}$ are directions of width at least $2$, we must also have $\ell_1\leq \alpha_0-2$ and $u_2\geq \beta_0+2$.
But this implies that $(\alpha_0-1, \beta_0+1)$ is in the interior of $\Pi$, contradicting the assumption.

\item[Case 2:] $(\alpha_0,\beta_0)=(u_1,\ell_0-u_1)$ and the edge directions at $(\alpha_0,\beta_0)$ are $d_1=\begin{psmallmatrix}
0\\1
\end{psmallmatrix}$ and $d_2=\begin{psmallmatrix}
-1\\1
\end{psmallmatrix}$. \\
Because of the width $2$ assumption, we must have $\ell_1\leq u_1 - 2 =\alpha_0-2$ and $u_0\geq \ell_0 + 2 = \alpha_0+\beta_0+2$.
Also, we must have $u_2\geq \beta_0+2$. If $u_2=\beta_0+2$, we obtain $\Pi\subseteq\Pi_2^{(\alpha_0-2,\beta_0)}$; if $u_2>\beta_0+2$, then $(\alpha_0-1, \beta_0+2)$ is in the interior of $\Pi$, which is a contradiction.

\item[Case 3:] $(\alpha_0,\beta_0)=(u_0-\ell_2,\ell_2)$ and the edge directions at $(\alpha_0,\beta_0)$ are $d_1=\begin{psmallmatrix}
-1\\0
\end{psmallmatrix}$ and $d_2=\begin{psmallmatrix}
-1\\1
\end{psmallmatrix}$. \\
Because of the width $2$ assumption, we must have $u_2\geq \ell_2 + 2 = \beta_0+2$ and $\ell_0\leq u_0 - 2= \alpha_0+\beta_0-2$.
Also, we must have $\ell_1\leq \alpha_0-2$. If $\ell_1=\alpha_0-2$, we obtain $\Pi\subseteq\Pi_1^{(\alpha_0-2,\beta_0)}$; if $\ell_1<\alpha_0-2$, then $(\alpha_0-2, \beta_0+1)$ is in the interior of $\Pi$, which is a contradiction.\qedhere
\end{description}
\end{proof}

\begin{proof}[Proof of \cref{lem:emptyInterior}]
When dealing with patterns and showing that \cref{lem:coveringPattern} holds, observe the following:
If there exists a solution for scalar products $(\alpha,\beta)\in \Pi_{\text{narrowed}}$ with $|\pi(\alpha,\beta)|\geq 2$, then item~\ref{lemitem:type1} of \cref{lem:coveringPattern} applies, so we can assume from now on that any scalar products $(\alpha,\beta)\in \Pi_{\text{narrowed}}$ for which there is a solution satisfy $|\pi(\alpha,\beta)|=1$.
To this end, we exploit two options: The first one is that such squares are contained in configurations of the type described by \cref{lem:type2step}; the second one is to find a linear sub-pattern that has the corresponding $(\alpha,\beta)$ in its domain and thus covers potential solutions for these scalar products.
We distinguish three cases based on the shape of the narrowed pattern domain $\Pi_{\text{narrowed}}$, which cover all the options by \cref{lem:structureEmptyInterior}:
\begin{description}
\item[Case 1:] There is a direction of width $0$ for $\Pi_{\text{narrowed}}$ in $\mathcal{D}^\perp$, but $\Pi_{\text{narrowed}}\not\subseteq \Pi_i^{(a,b)}$ for any $a,b\in\mathbb{Z}$ and $i\in\{1,2\}$.
\item[Case 2:] There is no direction of width $0$ but one of width $1$ for $\Pi_{\text{narrowed}}$ in $\mathcal{D}^\perp$, and $\Pi_{\text{narrowed}}\not\subseteq \Pi_i^{(a,b)}$ for any $a,b\in\mathbb{Z}$ and $i\in\{1,2\}$.
\item[Case 3:] There are $a,b\in\mathbb{Z}$ and $i\in\{1,2\}$ such that $\Pi_{\text{narrowed}}\subseteq \Pi_i^{(a,b)}$.
\end{description}

In case 1, observe that $\Pi_{\text{narrowed}}$ is bounded, hence there exist $(\alpha_0,\beta_0)\in\Pi_{\text{narrowed}}$, $d\in\mathcal{D}$ and $t\in\mathbb{Z}_{\geq 0}$ such that
$$\Pi_{\text{narrowed}} = \{(\alpha_0,\beta_0)+id\colon i\in\{0,\ldots, t\}\}\enspace,$$
and because $\Pi_{\text{narrowed}}\not\subseteq \Pi_i^{(a,b)}$ for any $a$, $b$, and $i$, we must have $t\geq 3$.
Because the pattern is non-linear, there exists $(\alpha_1,\beta_1)\in\Pi_{\text{narrowed}}$ with $|\pi(\alpha_1,\beta_1)|\geq 2$.
We claim that $|\pi((\alpha_0,\beta_0)+id)|\geq 2$ for all $i\in[t-1]$.

To see the claim, we first show that $|\pi((\alpha_0,\beta_0)+d)|\geq 2$.
If $(\alpha_0,\beta_0)\neq(\alpha_1,\beta_1)$, we may apply \cref{lem:neighbouringTwo} to $(\alpha_0,\beta_0)$ and $(\alpha_1,\beta_1)$ to obtain that $|\pi((\alpha_0,\beta_0)+d)|\geq 2$.
If, on the other hand, $(\alpha_0,\beta_0)=(\alpha_1,\beta_1)$, then apply \cref{lem:neighbouringTwo} to $(\alpha_0,\beta_0)+2d$ and $(\alpha_1,\beta_1)$, which also gives $|\pi((\alpha_0,\beta_0)+d)|\geq 2$.
Finally, applying \cref{lem:neighbouringTwo} once again to $(\alpha_0,\beta_0)+(i+1)d$ and $(\alpha_0,\beta_0)+d$ for $i\in\{2,\ldots,t-1\}$, we get that $|\pi(\alpha_0,\beta_0)+id)|\geq 2$.
Thus, the only potential scalar product pairs with $|\pi(\alpha,\beta)|=1$ are $(\alpha,\beta)\in\{(\alpha_0,\beta_0),(\alpha_0,\beta_0)+td\}$.
These $(\alpha,\beta)$ are part of a configuration as described by \cref{lem:type2step}, hence we get that if there is a solution for such $(\alpha,\beta)$, then either item~\ref{lemitem:goodCombo} or~\ref{lemitem:type1} of \cref{lem:coveringPattern} holds.

In case~2, we note that the condition on a flat direction of width~$1$ implies that there exists $(\alpha,\beta)\in\Pi_{\text{narrowed}}$ and two directions $d_1,d_2\in\mathcal{D}$ with $d_1\neq d_2$ and $d_1\neq-d_2$ such that
$$
\Pi_{\text{narrowed}}\subseteq \{(\alpha_0,\beta_0)+id_1+\varepsilon d_2\colon i\in\mathbb{Z}, \varepsilon\in\{0,1\}\}\enspace.
$$
Define $\Pi_0\coloneqq\Pi_{\text{narrowed}}\cap\{(\alpha_0,\beta_0)+id_1\colon i\in\mathbb{Z}\}$, and $\Pi_1\coloneqq\Pi_{\text{narrowed}}\cap\{(\alpha_0,\beta_0)+id_1+d_2\colon i\in\mathbb{Z}\}$.
If $|\Pi_0|<3$ or $|\Pi_1|<3$, then $\Pi_{\text{narrowed}}\subseteq\Pi_i^{(a,b)}$ for some $a,b\in\mathbb{Z}$ and $i\in\{0,1\}$, which we excluded in this case.
Thus, $|\Pi_0|\geq3$ and $|\Pi_1|\geq3$.
Observe that because the pattern $\pi$ is non-linear, for at least one $\varepsilon\in\{0,1\}$, there exist $(\alpha,\beta)\in\Pi_\varepsilon$ with $|\pi(\alpha,\beta)|\geq 2$, and hence we may apply the analysis from case~1 to such $\Pi_\varepsilon$ to see that if there is a solution for scalar products in $\Pi_\varepsilon$, then one of items~\ref{lemitem:goodCombo} or~\ref{lemitem:type1} of \cref{lem:coveringPattern} applies.
If not both $\Pi_\varepsilon$ fall into the previous case, then there is one remaining, say $\Pi_{\varepsilon'}$, such that for all $(\alpha,\beta)\in\Pi_{\varepsilon'}$, $|\pi(\alpha,\beta)|=1$.
Then $\widetilde{\pi}\coloneqq \pi|_{\Pi_{\varepsilon'}}$ is a linear sub-pattern of $\pi$, hence if there is a solution covered by $\widetilde{\pi}$, then item~\ref{lemitem:mainBranch} of \cref{lem:coveringPattern} applies.
This completes the analysis of case~2.

Finally, we deal with case~3 on a case-by-case basis, by going through potential narrowed pattern domain shapes that are contained in sets of the form $\Pi_0^{(a,b)}$ or $\Pi_1^{(a,b)}$ for some $(a,b)\in\mathbb{Z}^2$, presented here by increasing size of $|\Pi_{\text{narrowed}}|$.
\begin{itemize}
\item $|\Pi_{\text{narrowed}}|\leq 3$:
If $\Pi_{\text{narrowed}} = \{(\alpha_0,\beta_0)+id\colon i\in\{0,1,2\}\}$ for some $(\alpha_0,\beta_0)\in\mathbb{Z}^2$ and $d\in\mathcal{D}$, then the arguments from case~1 apply.
Otherwise, restricting $\pi$ to the subset of all $(\alpha,\beta)\in\Pi_{\text{narrowed}}$ with $|\pi(\alpha,\beta)|=1$ gives a sub-pattern $\widetilde{\pi}$ for which \cref{lem:coveringPattern} follows immediately.

\item $|\Pi_{\text{narrowed}}|= 4$: Because we require $\Pi_{\text{narrowed}}\subseteq \Pi_i^{(a,b)}$ for some $(a,b)\in\mathbb{Z}^2$ and $i\in\{0,1\}$, the only possible shapes of $\Pi_{\text{narrowed}}$ are the ones given in \cref{fig:patternShapes4}.

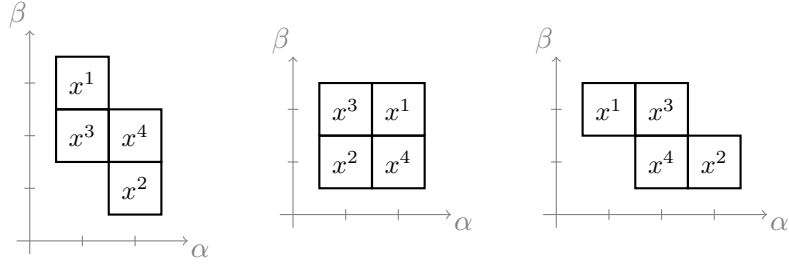
\begin{figure}[!ht]
\centering
\newcommand{\figscale}{0.7}
\begin{tikzpicture}[scale=\figscale]
\begin{scope}
\begin{scope}[every node/.style={draw, thick, rectangle, minimum width=\figscale cm, minimum height=\figscale cm, inner sep=0pt, text width=\figscale cm, align=center}]
\foreach \x/\y/\fill/\text in {-1/ 1/none/x^1,  -1/ 0/none/x^3,    0/ 0/none/x^4, 0/-1/none/x^2  }
\node[fill=\fill, font=\small] (\x\y) at (\x, \y) {\vphantom{l}$\text$};
\end{scope}

\begin{scope}[every path/.style={gray}]
\draw[->] (-2.25, -2) to node[pos=1, anchor=north west, inner sep=1pt] {$\alpha$} (1, -2);
\draw[->] (-2, -2.25) to node[pos=1, anchor=south east, inner sep=1pt] {$\beta$} (-2, 2);
\foreach \x in {-1, ..., 0}
\draw (\x, -2.1) to node[pos=0, anchor=north, font=\footnotesize] {} (\x, -1.9);
\foreach \y in {-1, ..., 1}
\draw (-2.1, \y) to node[pos=0, anchor=east, font=\footnotesize] {} (-1.9, \y);
\end{scope}
\end{scope}

\begin{scope}[xshift=5cm, yshift=-0.5cm]
\begin{scope}[every node/.style={draw, thick, rectangle, minimum width=\figscale cm, minimum height=\figscale cm, inner sep=0pt, text width=\figscale cm, align=center}]
\foreach \x/\y/\fill/\text in {-1/ 1/none/x^3,  0/ 1/none/x^1,   -1/ 0/none/x^2,    0/ 0/none/x^4}
\node[fill=\fill, font=\small] (\x\y) at (\x, \y) {\vphantom{l}$\text$};
\end{scope}

\begin{scope}[every path/.style={gray}]
\draw[->] (-2.25, -1) to node[pos=1, anchor=north west, inner sep=1pt] {$\alpha$} (1, -1);
\draw[->] (-2, -1.25) to node[pos=1, anchor=south east, inner sep=1pt] {$\beta$} (-2, 2);
\foreach \x in {-1, ..., 0}
\draw (\x, -1.1) to node[pos=0, anchor=north, font=\footnotesize] {} (\x, -0.9);
\foreach \y in {0, ..., 1}
\draw (-2.1, \y) to node[pos=0, anchor=east, font=\footnotesize] {} (-1.9, \y);
\end{scope}
\end{scope}

\begin{scope}[xshift=10cm, yshift=-0.5cm]
\begin{scope}[every node/.style={draw, thick, rectangle, minimum width=\figscale cm, minimum height=\figscale cm, inner sep=0pt, text width=\figscale cm, align=center}]
\foreach \x/\y/\fill/\text in {-1/ 1/none/x^1,  0/ 1/none/x^3,   0/ 0/none/x^4, 1/ 0/none/x^2 }
\node[fill=\fill, font=\small] (\x\y) at (\x, \y) {\vphantom{l}$\text$};
\end{scope}

\begin{scope}[every path/.style={gray}]
\draw[->] (-2.25, -1) to node[pos=1, anchor=north west, inner sep=1pt] {$\alpha$} (2, -1);
\draw[->] (-2, -1.25) to node[pos=1, anchor=south east, inner sep=1pt] {$\beta$} (-2, 2);
\foreach \x in {-1, ..., 1}
\draw (\x, -1.1) to node[pos=0, anchor=north, font=\footnotesize] {} (\x, -0.9);
\foreach \y in {0, ..., 1}
\draw (-2.1, \y) to node[pos=0, anchor=east, font=\footnotesize] {} (-1.9, \y);
\end{scope}
\end{scope}

\end{tikzpicture} \caption{Narrowed pattern domains if $|\Pi_{\text{narrowed}}|=4$ and $\Pi_{\text{narrowed}}\subseteq \Pi_i^{(a,b)}$ for some $(a,b)\in\mathbb{Z}^2$ and $i\in\{0,1\}$.}
\label{fig:patternShapes4}
\end{figure}

Now in any of these three cases, let $x^1$ and $x^2$ be solutions of the relaxation of the \RCCTUF{} problem for scalar products $(\alpha,\beta)$ located in the pattern $\Pi_{\text{narrowed}}$ as indicated in \cref{fig:patternShapes4}.
Applying the averaging lemma (\cref{lem:averagingPatterns}) to $x^1$ and $x^2$, we obtain solutions $x^3$ and $x^4$, and by the inequalities~\eqref{eq:scalarProductBounds} in~\cref{lem:averagingPatterns} and the property that $x^1+x^2=x^3+x^4$, we may assume that $x^3$ and $x^4$ are solutions for scalar products located in the pattern $\Pi_{\text{narrowed}}$ as indicated in \cref{fig:patternShapes4}, as well.

Now observe that the residues $\gamma_B^\top x_B^i$ satisfy $\gamma_B^\top x^1_B + \gamma_B^\top x^2_B = \gamma_B^\top x^3_B + \gamma_B^\top x^4_B$, and hence the sub-pattern $\widetilde\pi$ that maps $(\alpha,\beta)\in\Pi_{\text{narrowed}}$ to the residue $\gamma_B^\top x^i_B$, where $x^i$ is the solution for $(\alpha,\beta)$, is a linear sub-pattern.
More importantly, it is a linear sub-pattern that covers all $(\alpha,\beta)\in\Pi_{\text{narrowed}}$ with $|\pi(\alpha,\beta)|=1$, and hence \cref{lem:coveringPattern} follows.

\item $|\Pi_{\text{narrowed}}|= 5$: Again, requiring $\Pi_{\text{narrowed}}\subseteq \Pi_i^{(a,b)}$ for some $(a,b)\in\mathbb{Z}^2$ and $i\in\{0,1\}$, we can immediately enumerate the possible shapes of $\Pi_{\text{narrowed}}$, giving the list in \cref{fig:patternShapes5}.

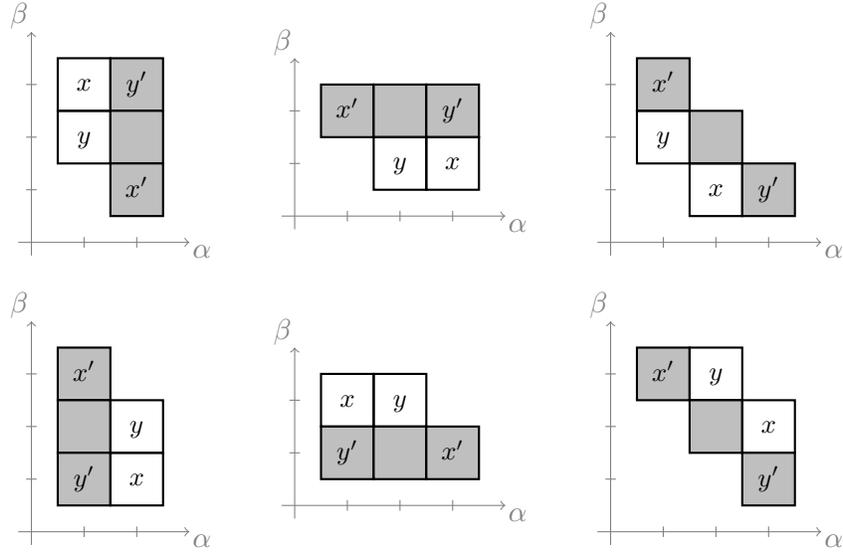
\begin{figure}[!ht]
\centering

\newcommand{\figscale}{0.7}
\begin{tikzpicture}[scale=\figscale]
\begin{scope}
\begin{scope}[every node/.style={draw, thick, rectangle, minimum width=\figscale cm, minimum height=\figscale cm, inner sep=0pt, text width=\figscale cm, align=center}]
\foreach \x/\y/\fill/\text in {-1/ 1/none/x,  0/ 1/lightgray/y',   -1/ 0/none/y,    0/ 0/lightgray/, 0/-1/lightgray/x'  }
\node[fill=\fill, font=\small] (\x\y) at (\x, \y) {\vphantom{ly}$\text$};
\end{scope}

\begin{scope}[every path/.style={gray}]
\draw[->] (-2.25, -2) to node[pos=1, anchor=north west, inner sep=1pt] {$\alpha$} (1, -2);
\draw[->] (-2, -2.25) to node[pos=1, anchor=south east, inner sep=1pt] {$\beta$} (-2, 2);
\foreach \x in {-1, ..., 0}
\draw (\x, -2.1) to node[pos=0, anchor=north, font=\footnotesize] {} (\x, -1.9);
\foreach \y in {-1, ..., 1}
\draw (-2.1, \y) to node[pos=0, anchor=east, font=\footnotesize] {} (-1.9, \y);
\end{scope}
\end{scope}

\begin{scope}[xshift=5cm, yshift=-0.5cm]
\begin{scope}[every node/.style={draw, thick, rectangle, minimum width=\figscale cm, minimum height=\figscale cm, inner sep=0pt, text width=\figscale cm, align=center}]
\foreach \x/\y/\fill/\text in {-1/ 1/lightgray/x',  0/ 1/lightgray/,   1/ 1/lightgray/y',
0/ 0/none/y, 1/ 0/none/x }
\node[fill=\fill, font=\small] (\x\y) at (\x, \y) {\vphantom{ly}$\text$};
\end{scope}

\begin{scope}[every path/.style={gray}]
\draw[->] (-2.25, -1) to node[pos=1, anchor=north west, inner sep=1pt] {$\alpha$} (2, -1);
\draw[->] (-2, -1.25) to node[pos=1, anchor=south east, inner sep=1pt] {$\beta$} (-2, 2);
\foreach \x in {-1, ..., 1}
\draw (\x, -1.1) to node[pos=0, anchor=north, font=\footnotesize] {} (\x, -0.9);
\foreach \y in {0, ..., 1}
\draw (-2.1, \y) to node[pos=0, anchor=east, font=\footnotesize] {} (-1.9, \y);
\end{scope}
\end{scope}

\begin{scope}[xshift=11cm]
\begin{scope}[every node/.style={draw, thick, rectangle, minimum width=\figscale cm, minimum height=\figscale cm, inner sep=0pt, text width=\figscale cm, align=center}]
\foreach \x/\y/\fill/\text in {-1/ 1/lightgray/x',  -1/ 0/none/y,  0/ 0/lightgray/, 0/-1/none/x, 1/-1/lightgray/y'}
\node[fill=\fill, font=\small] (\x\y) at (\x, \y) {\vphantom{ly}$\text$};
\end{scope}

\begin{scope}[every path/.style={gray}]
\draw[->] (-2.25, -2) to node[pos=1, anchor=north west, inner sep=1pt] {$\alpha$} (2, -2);
\draw[->] (-2, -2.25) to node[pos=1, anchor=south east, inner sep=1pt] {$\beta$} (-2, 2);
\foreach \x in {-1, ..., 1}
\draw (\x, -2.1) to node[pos=0, anchor=north, font=\footnotesize] {} (\x, -1.9);
\foreach \y in {-1, ..., 1}
\draw (-2.1, \y) to node[pos=0, anchor=east, font=\footnotesize] {} (-1.9, \y);
\end{scope}
\end{scope}

\begin{scope}[yshift=-5.5cm]
\begin{scope}
\begin{scope}[every node/.style={draw, thick, rectangle, minimum width=\figscale cm, minimum height=\figscale cm, inner sep=0pt, text width=\figscale cm, align=center}]
\foreach \x/\y/\fill/\text in {-1/ 1/lightgray/x',  -1/ 0/lightgray/,    0/ 0/none/y, -1/-1/lightgray/y',   0/-1/none/x  }
\node[fill=\fill, font=\small] (\x\y) at (\x, \y) {\vphantom{ly}$\text$};
\end{scope}

\begin{scope}[every path/.style={gray}]
\draw[->] (-2.25, -2) to node[pos=1, anchor=north west, inner sep=1pt] {$\alpha$} (1, -2);
\draw[->] (-2, -2.25) to node[pos=1, anchor=south east, inner sep=1pt] {$\beta$} (-2, 2);
\foreach \x in {-1, ..., 0}
\draw (\x, -2.1) to node[pos=0, anchor=north, font=\footnotesize] {} (\x, -1.9);
\foreach \y in {-1, ..., 1}
\draw (-2.1, \y) to node[pos=0, anchor=east, font=\footnotesize] {} (-1.9, \y);
\end{scope}
\end{scope}

\begin{scope}[xshift=5cm, yshift=-0.5cm]
\begin{scope}[every node/.style={draw, thick, rectangle, minimum width=\figscale cm, minimum height=\figscale cm, inner sep=0pt, text width=\figscale cm, align=center}]
\foreach \x/\y/\fill/\text in {-1/ 1/none/x,  0/ 1/none/y,   -1/ 0/lightgray/y',    0/ 0/lightgray/, 1/ 0/lightgray/x' }
\node[fill=\fill, font=\small] (\x\y) at (\x, \y) {\vphantom{ly}$\text$};
\end{scope}

\begin{scope}[every path/.style={gray}]
\draw[->] (-2.25, -1) to node[pos=1, anchor=north west, inner sep=1pt] {$\alpha$} (2, -1);
\draw[->] (-2, -1.25) to node[pos=1, anchor=south east, inner sep=1pt] {$\beta$} (-2, 2);
\foreach \x in {-1, ..., 1}
\draw (\x, -1.1) to node[pos=0, anchor=north, font=\footnotesize] {} (\x, -0.9);
\foreach \y in {0, ..., 1}
\draw (-2.1, \y) to node[pos=0, anchor=east, font=\footnotesize] {} (-1.9, \y);
\end{scope}
\end{scope}

\begin{scope}[xshift=11cm]
\begin{scope}[every node/.style={draw, thick, rectangle, minimum width=\figscale cm, minimum height=\figscale cm, inner sep=0pt, text width=\figscale cm, align=center}]
\foreach \x/\y/\fill/\text in {-1/ 1/lightgray/x',  0/ 1/none/y,   0/ 0/lightgray/, 1/ 0/none/x,1/-1/lightgray/y'}
\node[fill=\fill, font=\small] (\x\y) at (\x, \y) {\vphantom{ly}$\text$};
\end{scope}

\begin{scope}[every path/.style={gray}]
\draw[->] (-2.25, -2) to node[pos=1, anchor=north west, inner sep=1pt] {$\alpha$} (2, -2);
\draw[->] (-2, -2.25) to node[pos=1, anchor=south east, inner sep=1pt] {$\beta$} (-2, 2);
\foreach \x in {-1, ..., 1}
\draw (\x, -2.1) to node[pos=0, anchor=north, font=\footnotesize] {} (\x, -1.9);
\foreach \y in {-1, ..., 1}
\draw (-2.1, \y) to node[pos=0, anchor=east, font=\footnotesize] {} (-1.9, \y);
\end{scope}
\end{scope}
\end{scope}

\end{tikzpicture} \caption{Narrowed pattern domains if $|\Pi_{\text{narrowed}}|=5$ and $\Pi_{\text{narrowed}}\subseteq \Pi_i^{(a,b)}$ for some $(a,b)\in\mathbb{Z}^2$ and $i\in\{0,1\}$.}
\label{fig:patternShapes5}
\end{figure}

Observe that each of the six pattern domains in \cref{fig:patternShapes5} contains a segment of the form $\{(\alpha_0,\beta_0)+id\colon i\in\{0,1,2\}\}$ for some $(\alpha_0,\beta_0)\in\Pi_{\text{narrowed}}$ and $d\in\mathcal{D}$, namely the segments marked in gray in \cref{fig:patternShapes5}.
If for some $(\alpha,\beta)$ on such a segment, we have $|\pi(\alpha,\beta)|\geq 2$, then the arguments of case~1 apply, and they show that if there is a solution with scalar products on the segment, then either item~\ref{lemitem:goodCombo} or~\ref{lemitem:type1} of \cref{lem:coveringPattern} applies.
The remaining scalar products (i.e., those not covered by the segment) can then be treated as in the case $|\Pi_{\text{narrowed}}|\leq 3$.

Thus, let us assume that for all $(\alpha,\beta)$ that are marked gray in \cref{fig:patternShapes5}, $|\pi(\alpha,\beta)|=1$.
This implies that at least one of the other two $(\alpha,\beta)$ in the pattern (marked with $x$ and $y$ in \cref{fig:patternShapes5}) must satisfy $|\pi(\alpha,\beta)|\geq 2$.
In fact, we claim that in this case, both of the other two have that property.
This is enough to conclude because then, if there exist solutions for these scalar product pairs, item~\ref{lemitem:type1} of \cref{lem:coveringPattern} applies. Hence, restricting $\pi$ to the subset of all $(\alpha,\beta)\in\Pi_{\text{narrowed}}$ with $|\pi(\alpha,\beta)|=1$ (i.e., those in the segment) gives a sub-pattern $\widetilde{\pi}$ for which \cref{lem:coveringPattern} follows immediately.

To see the claim, we first assume that the pair $(\alpha,\beta)$ marked with an $x$ in \cref{fig:patternShapes5} satisfies $|\pi(\alpha,\beta)|\geq 2$.
Let the pair marked $x'$ be $(\alpha',\beta')$, and apply \cref{lem:neighbouringTwo} to $(\alpha,\beta)$ and $(\alpha',\beta')$ to obtain that there exists $d'\in\mathcal{D}$ such that $(\alpha',\beta')+d'\in\Pi_{\text{narrowed}}$ and $|\pi((\alpha',\beta')+d')|\geq 2$.
By assumption, $(\alpha',\beta')+d'$ can thus not be within the gray segment, and in all cases, it is immediate to see that $(\alpha',\beta')+d$ must correspond to the spot in the pattern marked with $y$ in \cref{fig:patternShapes5}.
The same argument works with the roles of $x$ and $x'$ interchanged with $y$ and $y'$, so the claim follows.

\item $|\Pi_{\text{narrowed}}|= 6$, i.e., $\Pi_{\text{narrowed}} = \Pi_i^{(a,b)}$ for some $(a,b)\in\mathbb{Z}^2$ and $i\in\{1,2\}$.
Note that in such domains, every $(\alpha,\beta)$ is contained in a boundary segment of the form $\{(\alpha_0,\beta_0)+id\colon i\in\{0,1,2\}\}$ for some $(\alpha_0,\beta_0)\in\Pi_{\text{narrowed}}$ and $d\in\mathcal{D}$.
As the pattern is non-linear, at least one of these segments contains $(\alpha,\beta)$ such that $|\pi(\alpha,\beta)|\geq 2$. Hence, the arguments of case~1 apply again, and if there is a solution with scalar products in that segment, then item~\ref{lemitem:goodCombo} or~\ref{lemitem:type1} of \cref{lem:coveringPattern} applies.
The remaining three scalar products (i.e., those not covered by the segment) can then be treated as in the case $|\Pi_{\text{narrowed}}|=3$.
\end{itemize}
To finish the proof, observe that in every case where a linear sub-pattern $\tilde\pi$ was identified, this could be done in strongly polynomial time in the size of the underlying \RCCTUF{} problem.
\end{proof}

\begin{proof}[Proof of \cref{lem:coveringPattern}]
If $\pi$ is linear, we may choose $\widetilde{\pi} = \pi$, and item~\ref{lemitem:mainBranch} of \cref{lem:coveringPattern} applies.
For non-linear $\pi$, by \cref{lem:nonEmptyInterior} we have that  \cref{lem:coveringPattern} holds if the domain of $\pi$ has non-empty interior, and by \cref{lem:emptyInterior} it holds for domains with empty interior.
\end{proof}

\subsection[Proof of \texorpdfstring{\cref{thm:decompProgress}}{Theorem~\ref{thm:decompProgress}}]{Proof of \cref{thm:decompProgress}\MOORdot}
\label{sec:proofThmDecompProgress}

We can (after potentially permuting rows and columns of the constraint matrix such that $A$ and $B$ change their roles) assume that the matrix $B$ has at most as many columns as $A$, i.e., $p=\min \{n_A,n_B\}=n_B$.
Furthermore, by \cref{lem:patternsBounds}, we can in strongly polynomial time determine $\ell_i, u_i\in\mathbb{Z}$ with $u_i-\ell_i\leq m-|R|$ for $i\in\{0,1,2\}$ such that if the \RCCTUF{} problem has a solution, then it has one with $\ell_0 \leq \alpha+\beta \leq u_0$, $\ell_1 \leq \alpha \leq u_1$, and $\ell_2\leq \beta \leq u_2$.
By \cref{lem:narrowedPiIneqs}, we can even choose these $\ell_i$ and $u_i$ for $i\in\{0,1,2\}$ such that the corresponding narrowed pattern $\pi\colon \Pi_{\text{narrowed}}\to2^{\{0,\ldots,m-1\}}$ has domain
$$
\Pi_{\text{narrowed}}=\left\{(\alpha,\beta)\in\mathbb{Z}^2\colon \ell_0 \leq \alpha+\beta \leq u_0, \ell_1 \leq \alpha \leq u_1, \ell_2\leq \beta \leq u_2\right\}\enspace.
$$
For each $(\alpha,\beta)\in\Pi_{\text{narrowed}}$, we can now in strongly polynomial time compute the following:
\begin{itemize}
\item A solution $x_{A}^{\alpha,\beta}$ for the relaxation of the $A$-problem for scalar products $(\alpha,\beta)$.
\item Exactly $t^{\alpha,\beta} \coloneqq \min\{|\pi(\alpha,\beta)|, m-\ell+1\}$ solutions $x_{B,i}^{\alpha,\beta}$ of the relaxation of the $B$-problem with pairwise different residues $r_i^{\alpha,\beta}\coloneqq\gamma_B^\top x_{B,i}^{\alpha,\beta}$ for $i \in [t^{\alpha,\beta}]$.
\end{itemize}
Note that computing the solutions $x_A^{\alpha,\beta}$ boils down to obtaining optimal vertex solutions to linear programs with a constraint matrix with bounded entries, which we can do in strongly polynomial time using the framework of \textcite{tardos_1986_strongly}.
For fixed $(\alpha,\beta)$, computing the solutions $x_{B,i}^{\alpha,\beta}$ can be done by solving $m-\ell+1$ many $B$-problems with scalar products $(\alpha,\beta)$, i.e., by recursively calling our procedure, where we start with the full set $R_{B,1}=\{0,\ldots,m-1\}$ of feasible target residues to get a solution $x_{B,1}^{\alpha,\beta}$, and then iterate using $R_{B,i+1}=R_{B,i}\setminus\{r_{i}^{\alpha,\beta}\}$, until we have $m-\ell+1$ many different residues, or we arrive at an infeasible problem.
In the latter case, we computed $\pi(\alpha,\beta)=\{r_i^{\alpha,\beta}\colon i=1,\ldots, t^{\alpha,\beta}\}$, while in the first case, we just obtained a subset of $\pi(\alpha,\beta)$.
Also note that each of the solutions $x_{B,i}^{\alpha,\beta}$ is obtained from an \RCCTUF{} problem with $p$ variables, modulus $m$, and at most $\ell$ many target residues, and we solved at most $m-\ell+1\leq 3$ of them for each $(\alpha,\beta)\in\Pi_{\text{narrowed}}$.
As $|\Pi_{\text{narrowed}}|<(m-\ell+1)^2$, this procedure needed less than $3(m-\ell+1)^2$ many recursive calls in total, in accordance with the claim in \cref{thm:decompProgress}.

Now invoke \cref{lem:coveringPattern}. We can directly check whether case~\ref{lemitem:goodCombo} applies by going through all $(\alpha,\beta)\in\Pi_{\text{narrowed}}$ with $|\pi(\alpha,\beta)|=1$. If case~\ref{lemitem:goodCombo} applies for some $(\alpha,\beta)\in \Pi_{\text{narrowed}}$, the combination $(x_A^{\alpha,\beta},x_{B,1}^{\alpha,\beta})$ must be a solution to the \RCCTUF{} problem.

If case~\ref{lemitem:type1} of \cref{lem:coveringPattern} applies, we can find a solution as follows:
For $(\alpha,\beta)\in\Pi_{\text{narrowed}}$ with $|\pi(\alpha,\beta)|\geq m-\ell+1$, we can in fact immediately find a solution because by construction, all combinations $(x_A^{\alpha,\beta},x_{B,i}^{\alpha,\beta})$ for $i=1,\ldots,m-\ell+1$ are feasible for the relaxation of the \RCCTUF{} problem, and then have pairwise different residues $\gamma_A^\top x_A^{\alpha,\beta}+\gamma_B^\top x_{B,i}^{\alpha,\beta}$.
But in this case, one of them must have a residue in the set $R$ that has size $\ell$, thus giving a feasible solution.
If on the other hand $1<|\pi(\alpha,\beta)|\leq m-\ell$, we reduce the problem to the modified $A$-problem
\begin{equation*}
	\begin{array}{rcl}
		A x_A & \leq & b_A - \alpha e \\
		h^\top x_A & = & \beta \\
		\gamma_A^\top x_A & \in & R'   \pmod{m} \enspace,
	\end{array}
\end{equation*}
where $R'=R-\pi(\alpha,\beta)$.
This problem has a solution if and only if the original \RCCTUF{} problem has one:
Note that any solution $x_A$ of its relaxation can be combined with any solution $x_B$ of the relaxation of the $B$-problem to obtain a solution $(x_A,x_B)$ that is feasible for the relaxation of the original \RCCTUF{} problem. Moreover, the residues in $R'$ are precisely those that allow us to obtain a combined solution $(x_A,x_B)$ that also satisfies the original congruency constraint.
Since $m$ is a prime number and $|\pi(\alpha,\beta)|>1$, \cref{lem:cauchyDavenport} guarantees that $|R'|\geq |R|+1=\ell+1$.
To sum up, if case~\ref{lemitem:type1} of \cref{lem:coveringPattern} applies, we either find a feasible solution in strongly polynomial time, or we construct at most $|\Pi_{\text{narrowed}}|\leq (m-\ell+1)^2$ many new \RCCTUF{} problems with $n-p$ variables, modulus $m$, and at least $\ell+1$ target residues such that at least one of them has a feasible solution that, as seen immediately from the above discussion, can be transformed to a solution of the initial problem in strongly polynomial time.

If the above strategy to obtain a solution in case~\ref{lemitem:type1} of \cref{lem:coveringPattern} fails, we know that case~\ref{lemitem:mainBranch} of \cref{lem:coveringPattern} applies.
In this case, we know that the problem has a solution that is covered by the linear sub-pattern $\widetilde{\pi}$. Applying \cref{thm:integrationSubpatterns}, we reduce the problem to an \RCCTUF{} problem with $n-p+2$ variables, modulus $m$, and $\ell$ target residues, with the additional property that the inequality system has an equality constraint.
This equality constraint allows for applying \cref{thm:projection} to eliminate one variable and obtain an equivalent \RCCTUF{} problem with $n-p+1$ variables, modulus $m$ and $\ell$ target residues.
It remains to observe that a solution of this problem can be immediately transformed back to a solution of the intermediate problem, and that solution can, by \cref{thm:integrationSubpatterns}, be transformed back to a solution of the original problem in strongly polynomial time.

Altogether, after solving less than $3(m-\ell+1)^2$ many \RCCTUF{} problems with at most $p$ variables and further strongly polynomial time operations, we can either obtain a feasible solution, or construct a family $\mathcal{F}$ of problems that have the properties claimed by \cref{thm:decompProgress}.\hfill\qedsymbol

\subsection[Proof of \texorpdfstring{\cref{thm:projection}}{Theorem~\ref{thm:projection}}]{Proof of \cref{thm:projection}\MOORdot}
\label{sec:proofThmProjection}

By performing a pivoting operation (see \cref{def:pivoting}) on the element $\alpha$, we obtain a new TU matrix which has $A - \alpha a_i a^\top_2$ as a submatrix. Hence, the latter is also TU. Moreover, the two systems are equivalent since
{\normalsize
\[
\begin{array}{r@{\;}c@{\;}l}
	Ax \; + \; a_1y & \leq & b \\
	a_2^\top x \; + \; \alpha y & = & \beta
\end{array}
\iff
\begin{array}{r@{\;}c@{\;}l}
	Ax  +  a_1 \alpha (\beta -a_2^\top x) & \leq & b \\
	     y & = & \alpha (\beta -a_2^\top x)
\end{array}
\iff
\begin{array}{r@{\;}c@{\;}l}
	(A - \alpha a_1 a^\top_2)  x & \leq & b - \alpha \beta a_1 \\
	  y & = & \alpha (\beta -a_2^\top x)
\end{array},
\]
}
where we use that $\alpha \in \{-1,1\}$ since the matrix is TU and $\alpha \neq 0$. This completes the proof.\hfill\qedsymbol

\subsection[Proof of \texorpdfstring{\cref{thm:pivoting}}{Theorem~\ref{thm:pivoting}}]{Proof of \cref{thm:pivoting}\MOORdot}
\label{sec:proofThmPivoting}

Consider an \RCCTUF{} problem
$$ Tx \leq b,\ \gamma^\top x \in R \pmod*{m},\ x\in\mathbb{Z}^n\enspace, $$
where $T$ falls into case~\ref{thmitem:TUdecomp_pivot} of \cref{thm:TUdecomp}, and assume without loss of generality that the desired pivoted matrix arises from $T$ by pivoting on the element in the first row and column.

Observe that due to \cref{lem:productInterval}, we can determine $u\in\mathbb{Z}$ such that the initial \RCCTUF{} problem is feasible if and only if
\begin{equation}\label{eq:boundy1}
 Tx \leq b,\ y_1\leq u,\ \gamma^\top x \in R \pmod*{m},\ x\in\mathbb{Z}^n
\end{equation}
is feasible.
Let $T \coloneqq \begin{psmallmatrix}
	\varepsilon & p^\top \\ q & C
\end{psmallmatrix}$,
and let $Q \in \mathbb{Z}^{n \times n}$ be the unimodular matrix that corresponds to the column operations such that the first row of $TQ$ is equal to the vector $(1,0,\ldots,0)$. Then, if $e_1$ denotes the first $n$-dimensional unit vector,
\begin{align*}
	\begin{pmatrix}
		T  \\
		e_1^\top
	\end{pmatrix}
	Q =
	\begin{pmatrix}
		\varepsilon & p^\top \\ q & C \\ 1 & 0
	\end{pmatrix}
	Q =
	\begin{pmatrix}
		1 & 0 \\  \varepsilon q & C - \varepsilon q p^\top \\ \varepsilon & -\varepsilon p^\top
	\end{pmatrix}\quad .
\end{align*}
Thus, substituting $x=Qy$ and observing that $x\in\mathbb{Z}^n$ if and only if $y\in\mathbb{Z}^n$, we can rewrite the system in~\eqref{eq:boundy1} as
\begin{equation}
 \begin{pmatrix}
		1 & 0 \\  \varepsilon q & C - \varepsilon q p^\top \\ \varepsilon & -\varepsilon p^\top
	\end{pmatrix}y \leq \begin{pmatrix}
b \\ u
\end{pmatrix},\ (\gamma^\top Q) y \in R \pmod*{m},\ y\in\mathbb{Z}^n\enspace,
\end{equation}
which is of the desired form.\hfill\qedsymbol

\appendix

\section[Detecting unboundedness of CCTU problems]{Detecting unboundedness of \CCTU{} problems}\label{app:unboundedness}

\begin{lemma}\label{lem:unboundedness}
A \CCTU{} problem is unbounded if and only if it is feasible and its relaxation is unbounded.
Moreover, given a feasible solution $x_0\in \mathbb{Z}^n$ to an unbounded \CCTU{} problem $\min\{c^\top x\colon Tx \leq b,\ \gamma^\top x \equiv r \pmod*{m},\ x\in\mathbb{Z}^n \}$, one can efficiently determine a vector $v\in \mathbb{Z}^n$ such that $x_0 + k\cdot v$ is feasible for any $k\in \mathbb{Z}_{\geq 0}$ and $c^\top v < 0$.
\end{lemma}

\begin{proof}
If a \CCTU{} problem is unbounded, then it obviously has a feasible solution and its relaxation is unbounded. To show the other direction, consider a feasible \CCTU{} problem
$$
\min\left\{c^\top x\colon Tx \leq b,\ \gamma^\top x \equiv r \pmod*{m},\ x\in\mathbb{Z}^n \right\}
$$
whose relaxation is unbounded. Thus, there exists a point $x_0\in\mathbb{Z}^n$ that is feasible for the problem, and a direction $r\in\mathbb{Z}^n$ with $c^\top r < 0$ such that for any point $x$ that is feasible for the relaxation, $x+r$ is feasible for the relaxation, as well.
This implies that $x_k = x_0+mkr$ satisfies $Tx_k\leq b$, $x_k\in\mathbb{Z}^n$, and $\gamma^\top x_k \equiv \gamma^\top x_0\equiv r\pmod*{m}$ for all $k\in\mathbb{Z}_{>0}$, and thus every such $x_k$ is feasible for the \CCTU{} problem. As $c^\top x_{k} = c^\top x_0 + mk c^\top r \to-\infty$ for $k\to\infty$, we conclude that the \CCTU{} problem is unbounded.

Moreover, note that if the relaxation is unbounded, then one can obtain in polynomial time a vector $r\in \mathbb{Z}^n$ as described above, i.e., with $c^{\top} r <0$ and $T r \leq 0$. Hence, the vector $v\coloneqq m\cdot r$ can be computed efficiently and has the properties claimed by the lemma.
\end{proof}

We remark that \cref{lem:unboundedness} extends to \RCCTUF{} problems and their optimization counterparts, as well.

\begingroup
\setlength{\emergencystretch}{0.5em}
\printbibliography
\endgroup

\end{document}